\documentclass[a4paper,10pt]{article}
\usepackage{geometry}
\usepackage{epsfig}
\usepackage{amsmath}
\usepackage{amssymb}
\usepackage{amsthm}
\usepackage{mathrsfs}
\usepackage{color}
\usepackage{amscd}
\usepackage{euscript}
\usepackage{ stmaryrd }
\usepackage{afterpage}
\usepackage{authblk}
\usepackage{enumitem}
\usepackage{mathtools}

\usepackage{ stmaryrd }
\usepackage{mathabx}
\usepackage{multicol}

\theoremstyle{plain} 
\newtheorem{thm}{Theorem}[section] 
 
\newtheorem{lem}[thm]{Lemma} 
\newtheorem{prop}[thm]{Proposition} 
\theoremstyle{definition} 
\newtheorem{defn}[thm]{Definition}
\newtheorem{rmk}[thm]{Remark}

\newtheorem{ex}[thm]{Example}

\DeclareFontEncoding{LS1}{}{}
\DeclareFontSubstitution{LS1}{stix}{m}{n}
\DeclareSymbolFont{symbols2}{LS1}{stixfrak} {m} {n}
\DeclareMathSymbol{\operp}{\mathbin}{symbols2}{"A8}
\DeclareMathOperator{\spn}{span}
\DeclareMathOperator{\rk}{rank}
\usepackage{tikz}
\usepackage{tikz-cd}
\usepackage{diagbox} % Per dividere le celle in diagonale
\usepackage{booktabs}
\usepackage{makecell}

\title{{\bf Dual slice functions and dual quaternionic polynomials}}

\author{Giulio Binosi\\
\small Istituto Nazionale di Alta Matematica\\
\small Unit\`a di Ricerca di Firenze c/o DiMaI ``U. Dini'' Universit\`a di Firenze\\
\small Viale Morgagni 67/A, I-50134 Firenze, Italy\\
\small ORCID 0000-0002-4733-6180\\
\small binosi@altamatematica.it

\and
Caterina Stoppato\\
\small DiMaI ``U. Dini'' Universit\`a di Firenze\\
\small Viale Morgagni 67/A, I-50134 Firenze, Italy\\
\small ORCID 0000-0001-9859-6559\\
\small caterina.stoppato@unifi.it}

\date{ }%{\small
\begin{document}

\maketitle

\begin{abstract}
    This work constructs a new function class over the algebra of dual quaternions: dual slice functions, which include one-sided dual quaternionic polynomials. The peculiar properties of dual slice functions therefore provide new tools to understand factorization of dual quaternionic polynomials, which was extensively studied in recent years for its applications to mechanism science. Dual slice functions are strictly related to the well-established theory of slice and slice regular functions over quaternions and dual quaternions.
\end{abstract}

%%%%%%%%%%%%%%%%%%%%%%%%%%%%%

\section{Introduction}\label{sec:introduction}

Let $\mathbb{H}$ be the real algebra of quaternions, namely the $4$-dimensional vector space $\mathbb{R}^4$, endowed with the associative and non commutative product acting on the standard basis $\{1,i,j,k\}$ as follows: $1$ is the identity element of the algebra and $i^2=j^2=k^2=ijk=-1$. We can embed $\mathbb{R}\subset\mathbb{H}$ as the vector subspace generated by $1$, which is the center of the algebra $\mathbb{H}$. We equip $\mathbb{H}$ with an involutive antiautomorphism $\mathbb{H}\to\mathbb{H},\ x\mapsto x^c $, where $(r_0+r_1i+r_2j+r_3k)^c=r_0-r_1i-r_2j-r_3k$. This makes $\mathbb{H}$ a real $^\ast$-algebra. We also consider the $8$-dimensional associative algebra of dual quaternions $\mathbb{DH}=\mathbb{H}+\epsilon\mathbb{H}=\{x_1+\epsilon x_2:x_1,x_2\in\mathbb{H}\}$, where, for any two dual quaternions $x=x_1+\epsilon x_2$ and $y=y_1+\epsilon y_2$, we set
\begin{equation*}
    x+y=x_1+y_1+\epsilon(x_2+y_2);\qquad
    xy=x_1y_1+\epsilon(x_1y_2+x_2y_1).
\end{equation*}
In particular, $\epsilon^2=0$ and the center of $\mathbb{DH}$ is the set $\mathbb{DR}\coloneqq\mathbb{R}+\epsilon\mathbb{R}$ of \emph{dual numbers}. The elements of $\mathbb{DC}\coloneqq\mathbb{C}+\epsilon\mathbb{C}$ will be called  \emph{dual complex numbers}. $\mathbb{DH}$ is a $^\ast$-algebra, endowed with the involutive antiautomorphism $\mathbb{DH}\to\mathbb{DH},\ x\mapsto x^c$, defined as $(x_1+\epsilon x_2)^c\coloneqq x_1^c+\epsilon x_2^c$, for all $x_1,x_2\in\mathbb{H}$. We point out that this involution fixes every point of $\mathbb{DR}$. We also point out that other notations, such as $\bar x$ or $x^*$, are used instead of $x^c$ in part of the existing literature.

For any real $\ast$-algebra $A$ and any $x\in A$, the work~\cite{perotti} set the notations
\begin{equation*}
    t(x)\coloneqq x+x^c,\qquad n(x)=xx^c.
\end{equation*}
Clearly, $t$ is $\mathbb{R}$-linear and $n$ is a homogeneous polynomial map of degree two in the real coordinates of $x$. In the special case when $A=\mathbb{H}$, for all $x_1\in\mathbb{H}$ we have $t(x_1),n(x_1)\in\mathbb{R}$. When $A=\mathbb{DH}$, for all $x=x_1+\epsilon x_2\in\mathbb{DH},a\in\mathbb{DR}$ we have $t(x)=t(x_1)+\epsilon t(x_2)\in\mathbb{DR}$, $n(x)=n(x_1)+\epsilon t(x_1x_2^c)\in\mathbb{DR}$, as well as $t(a)=2a$ and $n(a)=a^2$. When $A$ is associative,~\cite[Remark 1.8]{gpsalgebra} guarantees that $n(xy)=n(x)n(y)=n(y)n(x)=n(yx)$ for all $x,y\in A$.  When $A=\mathbb{DH}$, we also see that $t(ax)=ax+x^ca^c=ax+ax^c=at(x), n(ax)=n(a)n(x)=a^2n(x)$ and $n(x^c)=n(x)$ for all $x\in\mathbb{DH},a\in\mathbb{DR}$. Now fix $x\in A$. The element $x$ is invertible, with inverse $x^{-1}=n(x)^{-1}x^c$, if, and only if, $n(x),n(x^c)$ are invertible elements of $A$. This happens for all $x\in\mathbb{H}^*\coloneqq \mathbb{H}\setminus\{0\}$, whence $\mathbb{H}$ is a division algebra. In $A=\mathbb{DH}$ we have instead $n^{-1}(0)=\epsilon\mathbb{H}$ and $\epsilon\mathbb{H}^*$ is the set of zero divisors in $\mathbb{DH}$ (see~\cite[Example 1.13]{gpsalgebra}). On the other hand, $\mathbb{DH}\setminus\epsilon\mathbb{H}$ is a multiplicative group where every $x=x_1+\epsilon x_2$ has $x^{-1}=x_1^{-1}-\epsilon x_1^{-1}x_2x_1^{-1}$. Within any real $\ast$-algebra $A$, the \emph{quadratic cone} is defined as
\begin{equation*}
     Q_A\coloneqq\mathbb{R}\cup\{x\in A: t(x),n(x)\in\mathbb{R}, 4n(x)>t^2(x)\}\,.
\end{equation*}
On the one hand, $Q_\mathbb{H}=\mathbb{H}$. On the other hand, if we represent a dual quaternion $x$ in real components as $r_0+ir_1+jr_2+kr_3+\epsilon(r_4+ir_5+jr_6+kr_7)$, we find that $Q_\mathbb{DH}$ is contained in the \emph{Study quadric} $\mathscr{S}^7\coloneqq \{r_0r_4+r_1r_5+r_2r_6+r_3r_7=0\}$. More precisely, considering the hyperplane $\mathscr{H}:r_4=0$, we have
\begin{equation*}
    \mathscr{S}^7\cap\mathscr{H} = Q_\mathbb{DH}\cup(\mathbb{R}+\epsilon\operatorname{Im}(\mathbb{H})).
\end{equation*}
Belonging to this set is known in literature as the \emph{Study condition}. In particular, $Q_\mathbb{DH}$ is a $6$-dimensional semialgebraic subset of the $8$-dimensional real vector space $\mathbb{DH}$. More details about $\mathbb{H}$ and $\mathbb{DH}$ will be provided in the forthcoming Subsection~\ref{subsec:dualquaternions}.

Dual quaternions are commonly used, see~\cite{Daniilidis,selig}, for an alternate representation of the group $SE(3)$ of proper rigid body transformations of $\mathbb{R}^3$. Firstly, $\mathbb{R}^3$ can be identified with the real affine subspace $1+\epsilon\operatorname{Im}(\mathbb{H})$ of $\mathbb{DH}$. Secondly, the multiplicative subgroup $G:=\{h\in\mathbb{DH}\,:\,n(h)\in\mathbb{R}^*\}$ acts on the aforementioned affine subspace as follows:
\begin{equation}\label{eq:rigidbodytrasnformation}
1+\epsilon p_2 \longmapsto \frac{h\,(1+\epsilon p_2)\,\tilde h}{n(h)} =1+\epsilon(h_1p_2h_1^{-1}+2h_2h_1^{-1})
\end{equation}
where $\tilde h=h_1^c-\epsilon h_2^c$. In particular: translation by a vector $2h_2h_1^{-1}$ is obtained when $h\in \mathbb{R}+\epsilon\operatorname{Im}(\mathbb{H})$; rotation by angle $\theta$ (with $\cos(\frac\theta2)=\frac{\operatorname{Re}(h_1)}{|h_1|}$) about an axis with Pl\"ucker coordinates $\big(\frac{\operatorname{Im}(h_1)}{|\operatorname{Im}(h_1)|},\frac{h_2}{|\operatorname{Im}(h_1)|}\big)$ is obtained when $h\in Q_{\mathbb{DH}}\setminus\mathbb{R}$. Overall, the nested subgroups $\{h_1\in\mathbb{H} : n(h_1)=1\}$ and $\{h\in\mathbb{DH} : n(h)=1\}$ are double coverings of $SO(3)$ and $SE(3)$, respectively. More details can be found in the cited references and in~\cite{gstdualquaternions}.

The article \cite{hegedus} singled out, among dual quaternionic polynomials, the  special subclass of \emph{motion polynomials} $P(x)$ (see the forthcoming Definition~\ref{def:motionpolynomial}). For such a polynomial, the inclusion $P(t)\in G$ holds true for all $t\in\mathbb{R}$. Choosing $h=P(t)$ in~\eqref{eq:rigidbodytrasnformation} gives rigid body motions whose trajectories are rational curves. A linear monic polynomial $P(x)=x-h^{(1)}$ is a motion polynomial if, and only if, $h^{(1)}$ fulfills the Study condition. In particular: if $h^{(1)}$ belongs to $\mathbb{R}+\epsilon\operatorname{Im}\mathbb{H}$, then so does $P(t)=t-h^{(1)}$ for all $t\in\mathbb{R}$; if $h^{(1)}$ belongs to $Q_{\mathbb{DH}}$ then so does $P(t)=t-h^{(1)}$ for all $t\in\mathbb{R}$. The former setup may be used to parametrize a motion of a linkage with one prismatic joint; the latter setup, where the rotation axis is constant in $t$ but the angle varies with $t$, may be used to parametrize a motion of a linkage with one revolute joint. A product of linear monic polynomials
\[P(x)=(x-h^{(1)})\cdot\ldots\cdot(x-h^{(n)})\]
is a motion polynomial if $h^{(1)},\ldots,h^{(n)}$ all fulfill the Study condition. If $s$ of them belong to $\mathbb{R}+\epsilon\operatorname{Im}\mathbb{H}$ and $n-s$ belong to $Q_{\mathbb{DH}}$, we parametrize a motion of a linkage with $s$ prismatic joints and $n-s$ revolute joints. Any alternate factorization of the motion polynomial $P(x)$ into motion linear factors corresponds to another linkage producing the same exact motion. This motivated studying possible factorizations of dual quaternionic polynomials into linear factors $x-h^{(\ell)}$ with $h^{(\ell)}$ fulfilling the Study condition, see~\cite{gstdualquaternions,lischichoschrocker1,lischichoschrocker2,lischroeckerskopenkovscharler,schicho,siegelescharlerschroecker}. Other works, such as~\cite{pfurnerschroeckerhusty,siegelepfurnerschroecker} addressed the problem of factorization of dual quaternionic polynomials without the Study condition. Factorization (or even existence of zeros) is not guaranteed for dual quaternionic polynomials; not even for motion polynomials.

Both for the aforementioned applications and for an intrinsic interest, we became curious to understand zeros and factorizations of dual quaternionic polynomials in full detail. To better study dual quaternionic polynomials, we embedded them into the newly-defined class of \emph{dual slice functions} over dual quaternions. The construction and study of this class turned out to be very articulate and provided the material for the present article. The study of the zero sets of dual slice functions and polynomials over dual quaternions is postponed to a forthcoming paper, along with the related problem of factorization. Dual slice functions extend \emph{slice functions}, defined in~\cite{perotti} on open subsets $\Omega_D$ of the quadratic cone $Q_\mathbb{DH}$, to domains that are open subsets of $\mathbb{DH}$. In this setting, the special subclass of \emph{dual slice regular functions} is obtained by extending \emph{slice regular functions}. Quaternionic slice regular functions are presented in detail in the monograph~\cite{librospringer2}. Slice and slice regular functions over general alternative $\ast$-algebras have been defined in~\cite{perotti} and studied in subsequent literature, including~\cite{gpsalgebra}. The special case of dual quaternions was addressed in~\cite{gstdualquaternions}.

The present work is structured as follows. We start with preliminaries, which comprise: details about quaternions and dual quaternions in Subsection~\ref{subsec:dualquaternions}; the definition of \emph{stem functions} on open subsets $D$ of $\mathbb{R}_\mathbb{C}$ and the related construction of slice functions on $\Omega_D$, in Subsection~\ref{subsec:slicefunctions}; details about quaternionic slice functions in Subsection~\ref{subsec:quaternionicslicefunctions}; details about dual quaternionic slice functions in Subsection~\ref{subsec:dualquatenionicslicefunctions}. Section~\ref{sec:beyond} is devoted to a decomposition of the whole space $\mathbb{DH}$ into ``dual complex slices'', analogous to the standard decomposition of the quadratic cone $Q_A$ of any alternative $\ast$-algebra $A$ into ``complex slices''. This section also sets up appropriate open subsets $\widecheck{\Omega}_D$ of $\mathbb{DH}$, with $\widecheck{\Omega}_D\cap Q_\mathbb{DH}=\Omega_D$. Section~\ref{sec:dualstem} is devoted to natural extensions of stem functions on $D$ to $\check{D}\coloneqq D+\epsilon\mathbb{R}_\mathbb{C}\subset\mathbb{DR}_\mathbb{C}$. This gives rise to the class of \emph{dual stem functions} $\check{D}\to\mathbb{DH}_\mathbb{C}$. Section~\ref{sec:dualslice} defines the class $\mathcal{DS}(\widecheck{\Omega}_D)$ of dual slice functions $\widecheck{\Omega}_D\to\mathbb{DH}$, based on dual stem functions on $\check{D}$. Section~\ref{sec:dualslicealgebra} endows the class $\mathcal{DS}(\widecheck{\Omega}_D)$ with a $\ast$-algebra structure and a right $\mathbb{DH}$-module structure. It also shows that restriction to $\Omega_D$ is a $\ast$-algebra isomorphism (and a right $\mathbb{DH}$-module isomorphism) from $\mathcal{DS}(\widecheck{\Omega}_D)$ to the class $\mathcal{S}(\Omega_D)$ of slice functions on $\Omega_D$. Section~\ref{sec:dualsliceexpression} proves an important result, namely Theorem~\ref{thm:dualslice}: for every $\check{f}\in\mathcal{DS}(\widecheck{\Omega}_D)$, if we set $f\coloneqq\check{f}_{|_{\Omega_D}}$ and let $^\pi\!f$ denote the projection onto $\mathbb{H}$ of $\check{f}_{|_{\Omega_D\cap\mathbb{H}}}$, then
\[\check{f}(x)= f(x_1)+\epsilon \,d\left(^\pi f\right)_{x_1}\!(x_2)\,,\]
for all $x=x_1+\epsilon x_2\in\widecheck{\Omega}_D$ with $x_1\not\in\mathbb{R}$. The last condition can be removed under mild extra hypotheses on $\check{f}$. Thus, Theorem~\ref{thm:dualslice} states that, for $x_1$ fixed, the map $\mathbb{H}\to\mathbb{DH},\ x_2\mapsto\check{f}(x_1+\epsilon x_2)$ is a real affine map whose value at $0$ equals $f(x_1)$ and whose real differential at $0$ acts exactly like $\epsilon \,d\left(^\pi f\right)_{x_1}$. Section~\ref{sec:dualslicepolynomials} establishes another important result, namely Theorem~\ref{thm:dualslicepolynomials}. Among other things, this theorem states that $\mathcal{DS}(\mathbb{DH})$ includes, as a $\ast$-subalgebra and a right $\mathbb{DH}$-submodule, the class $\mathbb{DH}[x]$ of dual quaternionic polynomials. In particular, the peculiar expression of $\check{f}$ achieved in the main Theorem~\ref{thm:dualslice} applies to all dual quaternionic polynomials. This fact, along with other properties of dual slice functions, paves the way for our forthcoming study of zeros and factorizations of dual quaternionic polynomials.

%%%%%%%%%%%%%%%%%%%%%%%%%%%%%

\section{Preliminaries}
\subsection{Quaternions and dual quaternions}\label{subsec:dualquaternions}

Let $A$ be either $\mathbb{H}$ or $\mathbb{DH}$. The set of \emph{imaginary units} of $A$ is defined as
\begin{equation*}
    \mathbb{S}_A\coloneqq\{x\in A: t(x)=0,n(x)=1\}
\end{equation*}
and equals $\{x\in A: x^2=-1\}$ according to \cite[Remark 2.6]{gstdualquaternions}. The quadratic cone can be expressed as
\begin{equation}\label{eq:decompositionofcone}
     Q_A=\bigcup_{J\in\mathbb{S}_A}\mathbb{C}_J,
\end{equation}
where $\mathbb{C}_J\coloneqq\spn(1,J)$ is a $\ast$-subalgebra of $A$ isomorphic to the complex field $\mathbb{C}$. In other words, every element $x\in Q_A$ can be expressed as $x=\alpha+J\beta$ for some $\alpha,\beta\in\mathbb{R}$ and some $J\in\mathbb{S}_A$. For this specific $x$, we set $\operatorname{Re}(x)\coloneqq \alpha=t(x)/2$ and $\operatorname{Im}(x)\coloneqq J\beta=x-\operatorname{Re}(x)$. When $A=\mathbb{H}$ we find that $t$ is twice the projection along the real axis and $n$ is the squared Euclidean norm, whence 
\begin{equation*}
    \mathbb{S}_\mathbb{H}=\{a_1i+a_2j+a_3k:a_1^2+a_2^2+a_3^2=1\}\cong\mathbb{S}^2,\qquad  Q_\mathbb{H}=\mathbb{H},\qquad\operatorname{Im}(\mathbb{H})=\spn(i,j,k).
\end{equation*}
as well as $\operatorname{Re}(a_0+a_1i+a_2j+a_3k)=a_0$, $\operatorname{Im}(a_0+a_1i+a_2j+a_3k)=a_1i+a_2j+a_3k$.
When $A=\mathbb{DH}$, we have $t(x_1+\epsilon x_2)=2\operatorname{Re}(x_1)+2\epsilon\operatorname{Re}(x_2)\in\mathbb{DR},n(x_1+\epsilon x_2)=|x_1|^2+2\epsilon\langle x_1,x_2\rangle\in\mathbb{DR}$, whence
\begin{equation*}
    \begin{split}
        \mathbb{S}_\mathbb{DH}&=\{J_1+\epsilon J_2:J_1\in\mathbb{S}_\mathbb{H}, J_2\in\operatorname{Im}(\mathbb{H}), J_1\perp J_2\}=\bigcup_{J_1\in\mathbb{S}_\mathbb{H}}T_{J_1},\\
         Q_\mathbb{DH}&=\mathbb{R}\cup\{x_1+\epsilon x_2: x_1\in\mathbb{H}\setminus\mathbb{R},x_2\in\operatorname{Im}(\mathbb{H}), x_1\perp x_2\}=\mathbb{R}\cup \bigcup_{x_1\in\mathbb{H}\setminus\mathbb{R}}T_{x_1},\\
         \operatorname{Im}(Q_\mathbb{DH})&=\{x_1+\epsilon x_2:x_1,x_2\in\operatorname{Im}(\mathbb{H}),x_1\perp x_2\}=\bigcup_{x_1\in\operatorname{Im}(\mathbb{H})}T_{x_1},
    \end{split}
\end{equation*}
where $T_{x_1}$ is the affine $2$-plane $x_1+\epsilon\{x_2\in\operatorname{Im}(\mathbb{H}):x_1\perp x_2\}$, for any $x_1\in\mathbb{H}\setminus\mathbb{R}$. Later in our work, we will also use the affine $4$-space $H_{x_1}\coloneqq x_1+\epsilon\mathbb{H}$ for $x_1\in\mathbb{H}$, whose intersection with $Q_\mathbb{DH}$ is exactly $T_{x_1}$ if $x_1\notin\mathbb{R}$. Moreover, for all $x_1+\epsilon x_2\in Q_{\mathbb{DH}}$, 
\begin{equation}\label{eq:realandimaginarypartdq}
    \operatorname{Re}_\mathbb{DH}(x_1+\epsilon x_2)=\operatorname{Re}_\mathbb{H}(x_1),\qquad \operatorname{Im}_\mathbb{DH}(x_1+\epsilon x_2)=\operatorname{Im}_\mathbb{H}(x_1)+\epsilon x_2.
\end{equation}
In particular, given $x_1,x_2\in\mathbb{H}$ and assuming $x_1=\alpha+J_1\beta$ with $\alpha,\beta\in\mathbb{R}$ and $J_1\in\mathbb{S}_\mathbb{H}$, if the dual quaternion $x_1+\epsilon x_2$ belongs to $Q_{\mathbb{DH}}$, then $x_1+\epsilon x_2=\alpha+J\beta$, where $J=J_1+\epsilon x_2\beta^{-1}$. As a consequence, $T_{x_1}\subset\mathbb{DH}$ is the tangent plane to the $2$-sphere $\alpha+\mathbb{S}_\mathbb{H}\beta$ at $x_1$. Moreover, if $x_1\in\mathbb{C}_{J_1}\setminus\mathbb{R}$, for some $J_1\in\mathbb{S}_\mathbb{H}$ then $T_{x_1}=x_1+\epsilon\mathbb{C}_{J_1}^\perp$, where $\mathbb{C}_{J_1}^\perp$ is the orthogonal complement in $\mathbb{H}$ of the plane $\mathbb{C}_{J_1}$.

We will need the following notations.

\begin{defn}
    We define $A[x]\coloneqq\left\{\sum_{\ell=0}^nx^\ell a_\ell:n\in\mathbb{N},\{a_\ell\}_{\ell=0}^n\subset A\right\}$ and set
       \begin{align*}
    &\left(\sum_{\ell'=0}^{\phantom{'}n'}x^{\ell'}a_{\ell'}\right)\cdot\left(\sum_{\ell''=0}^{\phantom{''}n''}x^{\ell''}b_{\ell''}\right)\coloneqq\sum_{\ell=0}^{n'+n''}x^\ell\sum_{s=0}^\ell a_sb_{\ell-s}\,,\\
        &\left(\sum_{\ell'=0}^{\phantom{'}n'}x^{\ell'}a_{\ell'}\right)^c\coloneqq\sum_{\ell'=0}^{\phantom{'}n'}x^{\ell'}a_{\ell'}^c\,,\\
        &N\left(\sum_{\ell'=0}^{\phantom{'}n'}x^{\ell'}a_{\ell' }\right)\coloneqq\left(\sum_{\ell'=0}^{\phantom{'}n'}x^{\ell'}a_{\ell'}\right)\cdot\left(\sum_{\ell'=0}^{\phantom{'}n'}x^{\ell'}a_{\ell'}\right)^c=\sum_{\ell=0}^{2n'}x^\ell c_\ell,\quad c_\ell\coloneqq\sum_{s=0}^\ell a_{s}a_{\ell-s}^c\,,
    \end{align*}
    whence $c_{2u}=n(a_u)+\sum_{s=0}^{u-1}t(a_sa_{2u-s}^c)$ and $c_{2u+1}=\sum_{s=0}^u t(a_sa_{2u-s+1}^c)$. 
    
    We also define $A[x,x^c]\coloneqq\left\{\sum_{\ell,m=0}^nx^\ell(x^c)^ma_{\ell m}:n\in\mathbb{N},\{a_{\ell m}\}_{\ell,m=0}^n\subset A\right\}$ and set
    \begin{align*}
        &\left(\sum_{\ell',m'=0}^{\phantom{'}n'}x^{\ell'}(x^c)^{m'}a_{\ell' m'}\right)\cdot\left(\sum_{\ell'',m''=0}^{\phantom{''}n''}x^{\ell''}(x^c)^{m''}b_{\ell'' m''}\right)\coloneqq\sum_{\ell,m=0}^{n'+n''}x^\ell(x^c)^m\sum_{s=0}^\ell\sum_{t=0}^m a_{st}b_{\ell-s,m-t}\,,\\
        &\left(\sum_{\ell',m'=0}^{\phantom{'}n'}x^{\ell'}(x^c)^{m'}a_{\ell' m'}\right)^c\coloneqq\sum_{\ell',m'=0}^{\phantom{'}n'}x^{\ell'}(x^c)^{m'}a_{\ell' m'}^c\,,\\
        &N\left(\sum_{\ell',m'=0}^{\phantom{'}n'}x^{\ell'}(x^c)^{m'}a_{\ell' m'}\right)\coloneqq\sum_{\ell,m=0}^{2n'}x^\ell(x^c)^m\sum_{s=0}^\ell\sum_{t=0}^m a_{st}a_{\ell-s,m-t}^c.
    \end{align*}
\end{defn}

By direct inspection, $A[x]$ and $A[x,x^c]$ are nested real $\ast$-algebras and right $A$-modules. Furthermore, $\mathbb{R}[x],\mathbb{H}[x],\mathbb{DH}[x]$ are nested real $\ast$-algebras. Similar considerations apply to $\mathbb{R}[x],\mathbb{DR}[x],\mathbb{DH}[x]$, to $\mathbb{R}[x,x^c],\mathbb{H}[x,x^c],\mathbb{DH}[x,x^c]$ and to $\mathbb{R}[x,x^c],\mathbb{DR}[x,x^c],\mathbb{DH}[x,x^c]$. We point out that $N(P)(x)$ belongs to $\mathbb{R}[x,x^c]$ when $P(x)\in\mathbb{H}[x,x^c]$ and that $N(P)(x)$ belongs to $\mathbb{DR}[x,x^c]$ when $P(x)\in\mathbb{DH}[x,x^c]$. Let us recall from~\cite{hegedus} an important concept mentioned in Section~\ref{sec:introduction}.

\begin{defn}\label{def:motionpolynomial}
    An element $P(x)=\sum_{\ell=0}^nx^\ell a_\ell$ of $\mathbb{DH}[x]$ having degree $n$ is called a \emph{motion polynomial} if its leading coefficient $a_n$ is an invertible element of $\mathbb{DH}$ and if $N(P)(x)$ belongs to $\mathbb{R}[x]$.
\end{defn}

For our purposes, we will also use the complexified algebra $A_\mathbb{C}\coloneqq A\otimes\mathbb{C}=\{a+e_1b:a,b\in A\}$, where $\{1,e_1\}$ denotes an orthonormal basis of $\mathbb{C}$ and the operations of sum and multiplication are defined by $(a+e_1b)+(c+e_1d)\coloneqq a+c+e_1(b+d)$ and $(a+e_1b)(c+e_1d)\coloneqq ac-bd+e_1(ad+bc)$.  
$A_\mathbb{C}$ is a $^\ast$-algebra with $(a+e_1b)^c\coloneqq a^c+e_1b^c$. The latter map clearly fixes every element of $\mathbb{R}_\mathbb{C}$. Additionally, we equip $A_\mathbb{C}$ with the conjugation $\overline{a+e_1b}\coloneqq a-e_1b$. In general, this additional involution does not define a $\ast$-algebra structure on $A_\mathbb{C}$, but only on its commutative subalgebras. These include $\mathbb{R}_\mathbb{C}\simeq\mathbb{C}$ and, in the special case when $A=\mathbb{DH}$, also $\mathbb{DR}_\mathbb{C}\simeq\mathbb{DC}$. We point out that $\left(\,\overline{a+e_1b}\,\right)^c=a^c-e_1b^c=\overline{(a+e_1b)^c}$ for all $a,b\in A$. We now recall a few more definitions from~\cite{perotti}. For each $J\in\mathbb{S}_A$, we define the surjective real linear map
\begin{equation*}
\phi_J:A_\mathbb{C}\to A,\qquad\phi_J(a+e_1b)\coloneqq a+Jb.
\end{equation*}
The restriction of $\phi_J$ to $\mathbb{R}_\mathbb{C}\simeq\mathbb{C}$ is a real $^\ast$-algebra isomorphism from $(\mathbb{R}_\mathbb{C},+,\cdot,\bar{\phantom{z}})\simeq(\mathbb{C},+,\cdot,\bar{\phantom{z}})$ to $(\mathbb{C}_J,+,\cdot,\,^c)$ because of the following remark.

\begin{rmk}\label{rmk:phiJ}
    The expression $\phi_J((a+e_1b)(c+e_1d))=ac-bd+J(ad+bc)$ equals $\phi_J(a+e_1b)\phi_J(c+e_1d)=ac+JbJd+aJd+Jbc$ whenever $a,b$ belong to the center of $A$. The expression $\phi_J\left(\overline{a+e_1b}\right)=\phi_J(a-e_1b)=a-Jb$ equals $\phi_J(a+e_1b)^c=(a+Jb)^c=a^c-b^cJ$ whenever $a=a^c,b=b^c$ and $b$ belongs to the center of $A$. Finally, for real $a,b$, the equality $\phi_J(a+e_1b)=a+Jb=0$ implies $a=0=b$.
\end{rmk}

For any subset $D\subset \mathbb{R}_\mathbb{C}$, invariant under conjugation, the \emph{circularization} of $D$ is defined as
\begin{equation*}
    \Omega_D=\Omega_{D,A}\coloneqq\bigcup_{J\in\mathbb{S}_A}\phi_J(D)=\{\alpha+J\beta:\alpha+e_1\beta\in D, J\in\mathbb{S}_A\}\subseteq Q_A.
\end{equation*}
A subset of $Q_A$ is called \emph{circular} if it is the circularization of some open and invariant set $D\subset\mathbb{R}_\mathbb{C}$. We call $\Omega_D$ a \emph{slice domain} if $D$ is open and connected and if $D\cap\mathbb{R}\neq\emptyset$. $\Omega_D$ is called a \emph{product domain} if $D$ is open and has two connected components swapped by conjugation. By construction, for any $x\in\Omega_D$, there exist $z\in D, J\in\mathbb{S}_A$ such that $x=\phi_J(z)$: namely, if $x=\alpha+J\beta$, then $z=\alpha+e_1\beta$. If this is the case, we set $\mathbb{S}_{x}=\mathbb{S}_{x}^A\coloneqq\Omega_{\{z\}}=\alpha+\beta\mathbb{S}_A$. We note that $\mathbb{S}_x=\{x\}$ when $x\in\mathbb{R}$. In the special case when $A=\mathbb{DH}$, we point out that $\mathbb{S}_{x_1+\epsilon x_2}^{\mathbb{DH}}=\mathbb{S}_{x_1}^{\mathbb{DH}}$ for all $x_1+\epsilon x_2\in Q_\mathbb{DH}$.

Let us define the orthogonal projections $\pi_{\mathbb{H}}:\mathbb{DH}\to \mathbb{H}$ and $\pi_{\epsilon\mathbb{H}}:\mathbb{DH}\to\epsilon\mathbb{H}$ as $\pi_{\mathbb{H}}(x_1+\epsilon x_2)\coloneqq x_1$ and $\pi_{\epsilon\mathbb{H}}(x_1+\epsilon x_2)\coloneqq\epsilon x_2$ for $x_1,x_2\in\mathbb{H}$. Clearly, the fiber of $\pi_\mathbb{H}$ over any $x_1\in\mathbb{H}$ is the affine $4$-space $\pi_{\mathbb{H}}^{-1}(x_1)=H_{x_1}$. We remark that
\begin{equation*}
\mathbb{S}_{x_1+\epsilon x_2}^{\mathbb{DH}}\supset\pi_\mathbb{H}(\mathbb{S}_{x_1+\epsilon x_2}^{\mathbb{DH}})=\mathbb{S}_{x_1+\epsilon x_2}^{\mathbb{DH}}\cap\mathbb{H}=\mathbb{S}_{x_1}^{\mathbb{H}}.
\end{equation*}
Actually, the restriction $\pi_{\mathbb{H}}:\mathbb{S}_{x_1+\epsilon x_2}^{\mathbb{DH}}\to\mathbb{S}_{x_1}^{\mathbb{H}}$ is a vector fiber bundle, which we may think of as the tangent bundle to $\mathbb{S}_{x_1}^{\mathbb{H}}$. Moreover, 
\begin{equation}\label{eq:domainprojection}
    \Omega_{D,\mathbb{DH}}\supset\pi_\mathbb{H}(\Omega_{D,\mathbb{DH}})=\Omega_{D,\mathbb{DH}}\cap\mathbb{H}=\Omega_{D,\mathbb{H}}.
\end{equation}
While the fiber of $\pi_\mathbb{H}:Q_{\mathbb{DH}}\to\mathbb{H}$ over any $x_1\in\mathbb{R}$ is just the singleton $\{x_1\}$, the restriction $\pi_{\mathbb{H}}:Q_{\mathbb{DH}}\setminus\mathbb{R}\to\mathbb{H}\setminus\mathbb{R}$ is a vector fiber bundle whose fibers are planes. Indeed, it is the pullback bundle of $\pi_{\mathbb{H}}:\mathbb{S}_{\mathbb{DH}}\to\mathbb{S}_\mathbb{H}$ via the map
\begin{equation*}
p:\mathbb{H}\setminus\mathbb{R}\to\mathbb{S}_\mathbb{H},\qquad p(x)=\frac{\operatorname{Im}(x)  }{|\operatorname{Im}(x)|},
\end{equation*}
as it holds $p^*\mathbb{S}_\mathbb{DH}=\left\{(x_1,J_1+\epsilon J_2)\in (\mathbb{H}\setminus\mathbb{R})\times\mathbb{S}_\mathbb{DH}:J_1=\frac{\operatorname{Im}(x_1)}{|\operatorname{Im}(x_1)|}\right\}\cong Q_\mathbb{DH}\setminus\mathbb{R}$.

\subsection{General facts about slice and slice regular functions}\label{subsec:slicefunctions}

Some interesting classes of functions have been constructed and studied in~\cite{perotti}. The first one is the class of stem functions $D\to A_\mathbb{C}$.

\begin{defn}
    Let $D\subset\mathbb{R}_\mathbb{C}$ be an open set that is invariant with respect to conjugation, namely $z\in D$ if and only if $\overline{z}\in D$. A function $F:D\to A_\mathbb{C}$ is called \emph{stem function} if it is complex intrinsic, i.e. it satisfies
    \begin{equation*}
    F\left(\overline{z}\right)=\overline{F(z)},\qquad\forall z\in D.
    \end{equation*}
    If $F=F_0+e_1F_1$, where $F_0,F_1:D\to A$ are the $A$-valued components of $F$, we can equivalently require the following even-odd properties
    \begin{equation*}
        F_0\left(\overline{z}\right)=F_0(z),\qquad F_1\left(\overline{z}\right)=-F_1(z).
    \end{equation*}
    We denote by $Stem(D,A_\mathbb{C})$ the set of stem functions on $D$. We also set $Stem^k(D,A_\mathbb{C})\coloneqq Stem(D,A_\mathbb{C})\cap C^k(D,A_\mathbb{C})$ for any $k\in\mathbb{N}\cup\{\infty,\omega\}$. The symbol $Stem^{pol}(D,A_\mathbb{C})$ will denote the set of stem functions $F:D\to A_\mathbb{C}$ such that $F(\alpha+e_1\beta)$ is a polynomial function in the real variables $\alpha,\beta$.
\end{defn}    
    We remark that $Stem(D,A_\mathbb{C})$ is a real $^\ast$-algebra and a right $A$-module, with
    \begin{equation*}
        (F+G)(z)\coloneqq F(z)+G(z),\quad(F\cdot G)(z)\coloneqq F(z)G(z),\quad F^c(z)\coloneqq F(z)^c,\quad (Fa)(z)\coloneqq F(z)a, 
    \end{equation*}
    for all $F,G\in Stem(D,A_\mathbb{C})$ and $a\in A$.
\begin{rmk}\label{rmk:stempol}
    A stem function $F$ belongs to $Stem^{pol}(D,A_\mathbb{C})$ if, and only if, there exist $n\in\mathbb{N}$ and $\{a_{\ell m}\}_{\ell,m=0}^n\subset A$ such that $F(z)=\sum_{\ell,m=0}^nz^\ell\bar z^ma_{\ell m}$ for all $z\in D$. For the ``if'' part, it suffices consider each monomial $M(z)\coloneqq z^\ell\bar z^m$ (with $\ell,m\in\mathbb{N}$): it clearly fulfills $M(\bar z)=\overline{M(z)}$; additionally, $M(\alpha+e_1\beta)=(\alpha+e_1\beta)^\ell(\alpha-e_1\beta)^m$ is visibly a polynomial function of $\alpha,\beta$. Conversely, for $z=\alpha+e_1\beta\in D$, every polynomial function in $\alpha=\frac{z+\bar z}{2},\beta=(2e_1)^{-1}(z-\bar z)$ can obviously be expressed as $F(z)=\sum_{\ell,m=0}^nz^\ell\bar z^ma_{\ell m}$ for some $n\in\mathbb{N}$ and some finite sequence $\{a_{\ell m}\}_{\ell,m=0}^n\subset A_\mathbb{C}$. Now, condition $F(\bar z)=\overline{F(z)}$ yields $\sum_{\ell,m=0}^n\bar z^\ell z^ma_{\ell m}=\sum_{\ell,m=0}^n\bar z^\ell z^m\bar a_{\ell m}$, whence $a_{\ell m}=\bar a_{\ell m}$ for all $\ell,m\in\{0,\ldots,n\}$, i.e., $\{a_{\ell m}\}_{\ell,m=0}^n\subset A$.

    For every $A'\in\{\mathbb{R},\mathbb{DR},\mathbb{H}\}$ we note, for future use, that $F\in Stem^{pol}(D,A_\mathbb{C})$ takes values in $A'_\mathbb{C}$ if, and only if, the sequence $\{a_\ell\}_{\ell=0}^n$ of its coefficients is contained in $A'$.
\end{rmk}

The work~\cite{perotti} also gave the next definition.

\begin{defn}
    The Wirtinger operators $\frac{\partial}{\partial z},\frac{\partial}{\partial \overline{z}}:Stem^1(D,A_\mathbb{C})\to Stem^0(D,A_\mathbb{C})$ are defined through the formulas
    \begin{equation*}
        \frac{\partial F}{\partial z}\coloneqq\frac{1}{2}\left(\frac{\partial F}{\partial\alpha}-e_1\frac{\partial F}{\partial\beta}\right),\qquad \frac{\partial F}{\partial \overline{z}}\coloneqq\frac{1}{2}\left(\frac{\partial F}{\partial\alpha}+e_1\frac{\partial F}{\partial\beta}\right)
    \end{equation*}
    and are right $A$-module morphisms. We call $F$ \emph{holomorphic} if $\frac{\partial F}{\partial \overline{z}}=0$, \emph{antiholomorphic} if $\frac{\partial F}{\partial z}=0$. The symbol $HolStem(D,A_\mathbb{C})$ will denote the right $A$-submodule of $Stem^1(D,A_\mathbb{C})$ consisting of holomorphic stem functions.
\end{defn}

We point out that $\frac{\partial(FG)}{\partial z}=\frac{\partial F}{\partial z}G+F\frac{\partial G}{\partial z}$ and $\frac{\partial (FG)}{\partial \overline{z}}=\frac{\partial F}{\partial \overline{z}}G+F\frac{\partial G}{\partial \overline{z}}$. As a consequence, $HolStem(D,A_\mathbb{C})$ is a real $*$-subalgebra of $Stem^1(D,A_\mathbb{C})$. Moreover, $\frac{\partial F}{\partial \overline{z}}=0$ is equivalent to the Cauchy-Riemann equations
    \begin{equation}\label{eq:CR}
        0\equiv2\left(\frac{\partial F}{\partial \overline{z}}\right)_0=\frac{\partial F_0}{\partial\alpha}-\frac{\partial F_1}{\partial\beta},\qquad 0\equiv2\left(\frac{\partial F}{\partial \overline{z}}\right)_1=\frac{\partial F_1}{\partial\alpha}+\frac{\partial F_0}{\partial\beta}
    \end{equation}
for the components $F_0,F_1$ of $F$. We also point out that, for every $F\in Stem(D,A_\mathbb{C})$, the function $\overline{F}$ (mapping each $z\in D$ to $\overline{F(z)}=F(\bar z)$) still belongs to $Stem(D,A_\mathbb{C})$. Moreover, $\overline{F}$ is antiholomorphic if $F$ is holomorphic.

\begin{ex}
    If $F(z)=\sum_{\ell,m=0}^nz^\ell\bar z^ma_{\ell m}$ for all $z\in D$, where $\{a_{\ell m}\}_{\ell,m=0}^n\subset A$, then
    \begin{equation*}
        \frac{\partial F}{\partial z}(z)=\sum_{\ell,m=0}^n\ell z^{\ell-1}\bar z^ma_{\ell m}\,,\quad \frac{\partial F}{\partial \bar z}(z)=\sum_{\ell,m=0}^nm z^\ell\bar z^{m-1}a_{\ell m}\,.
    \end{equation*}
    Here, we adopted the following conventions: $\ell z^{\ell-1}\equiv0$ for $\ell=0$, $m\bar z^{m-1}\equiv0$ for $m=0$. In particular, $F$ is holomorphic if, and only if $a_{\ell m}=0$ for all $\ell\in\mathbb{N},m\in\mathbb{N}^*$. Thus, $Stem^{pol}(D,A_\mathbb{C})\cap HolStem(D,A_\mathbb{C})$ consists of all functions $F(z)=z^na_n+\ldots+za_1+a_0$ with $n\in\mathbb{N}$ and $\{a_\ell\}_{\ell=0}^n\subset A$. We note, for future use, that such an $F$ has
    \begin{equation*}
        \frac{\partial F}{\partial z}(z)=nz^{n-1}a_n+\ldots+2za_2+a_1\,.
    \end{equation*}
    Three examples of elements of $Stem^{pol}(D,A_\mathbb{C})$ not belonging to $HolStem(D,A_\mathbb{C})$ are $RE(z)\coloneqq\frac{z+\bar z}{2}$, and $IM(z)=\frac{z-\bar z}{2}$ and $IM^2(z)=\frac{z^2-2z\bar z+\bar z^2}{4}$, which will soon be useful.
\end{ex}

Stem functions are a tool to construct two other interesting classes of functions, see~\cite{perotti}: slice functions and slice regular functions $\Omega_D\to A$.

\begin{defn}
    Assume $D\subset \mathbb{R}_\mathbb{C}$ to be invariant under conjugation. A function $f:\Omega_D\subset A\to A$ is called a \emph{slice function} if there exists a stem function $F=F_0+e_1F_1:D\to A_\mathbb{C}$ such that, for any $x=\phi_J(z)\in \Omega_D$ it holds
    \begin{equation*}
        f(x)=F_0(z)+JF_1(z).
    \end{equation*}
    If this is the case, we write $f=\mathcal{I}(F)$. The function $f$ is called \emph{slice-preserving} if $F$ takes values in $\mathbb{R}_\mathbb{C}$, i.e., if $F_0,F_1$ take values in $\mathbb{R}$. We denote the set of slice functions on $\Omega_D$ by $\mathcal{S}(\Omega_D)$ and the set $\mathcal{I}(Stem(D,\mathbb{R}_\mathbb{C}))$ of slice-preserving slice functions on $\Omega_D$ by $\mathcal{S}_\mathbb{R}(\Omega_D)$, respectively. For $k\in\mathbb{N}\cup\{\infty,\omega\}$, we denote by the symbols $\mathcal{S}^k(\Omega_D),\mathcal{S}^{pol}(\Omega_D)$ the images through $\mathcal{I}$ of $Stem^k(D,A_\mathbb{C}),Stem^{pol}(D,A_\mathbb{C})$, respectively.
    
    A slice function $\mathcal{I}(F):\Omega_D\to A$ is termed \emph{slice regular} if the inducing stem function $F$ is holomorphic. The symbol $\mathcal{SR}(\Omega_D)$ will denote the set of slice regular functions $f:\Omega_D\to A$, which is the image through $\mathcal{I}$ of $HolStem(D,A_\mathbb{C})$. Furthermore, we set $\mathcal{SR}_\mathbb{R}(\Omega_D)\coloneqq\mathcal{S}_\mathbb{R}(\Omega_D)\cap\mathcal{SR}(\Omega_D)$ and $\mathcal{SR}^{pol}(\Omega_D)\coloneqq\mathcal{S}^{pol}(\Omega_D)\cap\mathcal{SR}(\Omega_D)$. In the special case when $A=\mathbb{DH}$, we also set $\mathcal{S}_\mathbb{DR}(\Omega_D)\coloneqq\mathcal{I}(Stem(D,\mathbb{DR}_\mathbb{C}))$, as well as $\mathcal{SR}_\mathbb{DR}(\Omega_D)\coloneqq\mathcal{S}_\mathbb{DR}(\Omega_D)\cap\mathcal{SR}(\Omega_D)$.
\end{defn}
Equivalently, we can define $f=\mathcal{I}(F)$ as the unique function that makes the following diagram commutative for any $J\in\mathbb{S}_A$:
\begin{center}
    \begin{tikzcd}
	{D} && {A_\mathbb{C}} \\
	& \circlearrowleft \\
	{ Q_A\supset\Omega_D} && {A.}
	\arrow["F", from=1-1, to=1-3]
	\arrow["{\phi_J}"', from=1-1, to=3-1]
	\arrow["f"', from=3-1, to=3-3]
	\arrow["{\phi_J}", from=1-3, to=3-3]
\end{tikzcd}
\end{center}

\begin{ex}
    The map $\operatorname{Re}:Q_A\to\mathbb{R}\subset A,\ x\mapsto\frac{x+x^c}{2}$ is an element of $\mathcal{S}_\mathbb{R}^{pol}(Q_A)$, induced by the polynomial stem function $RE:\mathbb{R}_\mathbb{C}\to\mathbb{R},\ z\mapsto\frac{z+\bar z}{2}$. The map $\operatorname{Im}:Q_A\to A,\ x\mapsto\frac{x-x^c}{2}$ is another element of $\mathcal{S}_\mathbb{R}^{pol}(Q_A)$, induced by the polynomial stem function $IM:\mathbb{R}_\mathbb{C}\to e_1\mathbb{R},\ z\mapsto\frac{z-\bar z}{2}$. The map $Q_A\to A,\ x\mapsto\frac{x^2-2xx^c+(x^c)^2}{4}$ is yet another element of $\mathcal{S}_\mathbb{R}^{pol}(Q_A)$, which by direct computation is induced by the polynomial stem function $IM^2:\mathbb{R}_\mathbb{C}\to \mathbb{R},\ z\mapsto\frac{z^2-2z\bar z+\bar z^2}{4}$. Another direct computation shows that $\frac{x^2-2xx^c+(x^c)^2}{4}=\operatorname{Im}(x)^2$.
\end{ex}

Since $A$ is associative, the map $\mathcal{I}:Stem(D,A_\mathbb{C})\to\mathcal{S}(\Omega_D)$ mapping each stem function into its induced slice function is a right $A$-module isomorphism. This fact is well known, see for instance \cite[Proposition 6.7]{gsunifiedtheory}. There exist unique definitions of product $f\cdot g$ and conjugation $f^c$ of elements $f,g$ of $\mathcal{S}(\Omega_D)$ that make $\mathcal{S}(\Omega_D)$ a real $^\ast$-algebra and $\mathcal{I}$ an isomorphism of $^\ast$-algebras. The notations $N(f)\coloneqq f\cdot f^c,T(f)\coloneqq f+f^c$ will also be useful. For more details about the definitions of $f\cdot g,f^c,N(f)$, see~\cite[Definitions 9 \& 11]{perotti}. Additionally, for any $\ell\in\mathbb{N}$, the symbol $f^{\bullet\ell}(x)=f(x)^{\bullet\ell}$ will denote the $\ell$-th power of $f$ with respect to the multiplicative operation $\cdot$. The equalities $(f\cdot g)(x)=f(x)g(x)$, $f^c(x)=(f(x^c))^c$, $N(f)(x)=n(f(x))$ and $T(f)(x)=t(f(x))$ are true when $f\in\mathcal{S}_\mathbb{R}(\Omega_D)$, but false in general: see~\cite[Remark 7]{perotti}. Each of the subclasses $\mathcal{SR}(\Omega_D),\mathcal{S}_\mathbb{R}(\Omega_D),\mathcal{SR}_\mathbb{R}(\Omega_D)$ is a $^\ast$-subalgebra of $\mathcal{S}(\Omega_D)$.

A slice function $f=\mathcal{I}(F_0)$, induced by an $A$-valued stem function $F=F_0+e_10=F_0:D\to A\subset A_\mathbb{C}$, is called a \emph{circular slice function}. It is constant on each ``sphere'' $\alpha+\mathbb{S}_A\beta$, with $\alpha+e_1\beta\in D$, i.e., on each $\mathbb{S}_x^A$ with $x\in\Omega_D$. The work~\cite{perotti} also gave the next definition.

\begin{defn}
    Let $f=\mathcal{I}(F)\in\mathcal{S}(\Omega_D)$, with $F=F_0+e_1F_1$. The \emph{spherical value} and \emph{spherical derivative} of $f$ are defined, respectively, as
    \begin{align*}
        v_sf(x)&\coloneqq\frac{1}{2}(f(x)+f(x^c)),\qquad\forall\, x\in\Omega_D,\\
        \partial_s f(x)&\coloneqq(x-x^c)^{-1}(f(x)-f(x)^c),\qquad\forall\, x\in\Omega_D\setminus\mathbb{R}.
    \end{align*}
    The functions $v_sf,\partial_s f$ are also denoted by $f^\circ_s,f'_s$, respectively.
\end{defn}
Note that $f^\circ_s,f'_s$ are slice functions, induced respectively by $F_0$ and by the stem function $\alpha+e_1\beta\mapsto F_1(\alpha+e_1\beta)/\beta$. In particular, $f^\circ_s$ and $f'_s$ are circular slice functions. Moreover, by construction, 
\begin{equation}\label{eq:sphericalrepresentationformula}
    f(x)=f^\circ_s(x)+\operatorname{Im}(x)f'_s(x)=f^\circ_s(x)+(\operatorname{Im}\cdot f'_s)(x)
\end{equation}
for all $x\in\Omega_D\setminus\mathbb{R}$ and $f(x)=f^\circ_s(x)$ for all $x\in\Omega_D\cap\mathbb{R}$. In the special case when $F_1\in\mathcal{C}^1(D,A_\mathbb{C})$ (true if $f\in\mathcal{S}^1(\Omega_D)$), we can extend $f'_s$ to a unique element of $\mathcal{S}^0(\Omega_D)$ by setting $f'_s(\alpha)\coloneqq\frac{\partial F_1}{\partial\beta}(\alpha)$ for all $\alpha\in\Omega_D\cap\mathbb{R}$. In this situation, formula~\eqref{eq:sphericalrepresentationformula} is valid throughout $\Omega_D$. Moreover, $f$ is circular if, and only if, $f'_s\equiv0$ in $\Omega_D\setminus\mathbb{R}$, i.e., $f=f^\circ_s$ throughout $\Omega_D$. Formula~\eqref{eq:sphericalrepresentationformula} is useful to study the zero set $V(f)$ of any slice function $f:\Omega_D\to A$, see the forthcoming Subsections~\ref{subsec:quaternionicslicefunctions} and~\ref{subsec:dualquatenionicslicefunctions}. The work~\cite{gpsalgebra} showed that
\begin{align}
    (f\cdot g)(x)&=f^\circ_s(x)g^\circ_s(x)+\operatorname{Im}(x)^2f'_s(x)g'_s(x)+\operatorname{Im}(x)\left(f^\circ_s(x)g'_s(x)+f'_s(x)g^\circ_s(x)\right)\,,\label{eq:sliceproduct}\\
    f^c(x)&=f^\circ_s(x)^c+\operatorname{Im}(x)f'_s(x)^c\,,\label{eq:sliceconjugate}\\
    N(f)&=n\left(f^\circ_s(x)\right)+\operatorname{Im}(x)^2\,n\left(f'_s(x)\right)+\operatorname{Im}(x)\,t\left(f^\circ_s(x)f'_s(x)^c\right)\label{eq:normalfunction}
\end{align}
for $x\in\Omega_D\setminus\mathbb{R}$ and that $(f\cdot g)(x)=f^\circ_s(x)g^\circ_s(x)=f(x)g(x)$, $f^c(x)=f^\circ_s(x)^c=f(x)^c$, and $N(f)=n\left(f^\circ_s(x)\right)=n(f(x))$ for $x\in\Omega_D\cap\mathbb{R}$. The last three formulas hold true at all $x\in\Omega_D$ in the special case when $f,g$ are circular slice functions. This allows us to derive from~\eqref{eq:sliceproduct},~\eqref{eq:sliceconjugate},~\eqref{eq:normalfunction} the following remark.

\begin{rmk}\label{rmk:sliceproductofcomponents}
    For all $f,g\in\mathcal{S}(\Omega_D)$, the equalities
    \begin{align*}
        f\cdot g&=f^\circ_s\cdot g^\circ_s+\operatorname{Im}^2\cdot f'_s\cdot g'_s+\operatorname{Im}\cdot \left(f^\circ_s\cdot g'_s+f'_s\cdot g^\circ_s\right)\,,\\
        f^c&=(f^\circ_s)^c+\operatorname{Im}\cdot(f'_s)^c\,,\\
        N(f)&=N\left(f^\circ_s\right)+\operatorname{Im}^2\cdot N\left(f'_s\right)+\operatorname{Im}\cdot T\left(f^\circ_s\cdot(f'_s)^c\right)
    \end{align*}
    hold true in $\Omega_D\setminus\mathbb{R}$. If, moreover, $f,g\in\mathcal{S}^1(\Omega_D)$, then the same equalities hold true throughout $\Omega_D$.
\end{rmk}

We notice that mapping $f=\mathcal{I}(F)$ to the function $x\mapsto f(x^c)=f^\circ_s(x)-\operatorname{Im}(x)f'_s(x)$ preserves $\mathcal{S}(\Omega_D)$ and $\mathcal{S}^k(\Omega_D)$ for all $k\in\mathbb{N}\cup\{\infty,\omega\}$, but not $\mathcal{SR}(\Omega_D)$. This is because  $x\mapsto f(x^c)$ is still a slice function on $\Omega_D$, induced by $\overline{F}$.

We recall from~\cite[Proposition 7]{perotti} that $\mathcal{S}^{2k+1}(\Omega_D)\subset C^k(\Omega_D,A)$ for all $k\in\mathbb{N}$ and that $\mathcal{S}^\infty(\Omega_D)\subset C^\infty(\Omega_D,A)$, $\mathcal{S}^\omega(\Omega_D)\subset C^\omega(\Omega_D,A)$. Moreover, according to~\cite[Corollary 2.7]{global}, the restriction of an element of $\mathcal{S}^1(\Omega_D)$ to $\Omega_D\setminus\mathbb{R}$ is a $C^1$ function. Given  $f\in\mathcal{S}^1(\Omega_D)$, i.e., $f=\mathcal{I}(F)$ with $F\in Stem^1(D,A_\mathbb{C})$, let us set
\begin{equation*}
    \partial_c f\coloneqq \mathcal{I}\left(\frac{\partial F}{\partial z}\right),\quad\overline{\partial}_c f\coloneqq \mathcal{I}\left(\frac{\partial F}{\partial \bar z}\right)\,.
\end{equation*}
We point out that $\partial_c,\overline{\partial}_c:\mathcal{S}^1(\Omega_D)\to\mathcal{S}^0(\Omega_D)$ are right $A$-module morphisms and that $\partial_c(f\cdot g)=\partial_c f\cdot g+f\cdot\partial_c g$. If $f\in\mathcal{SR}(\Omega_D)$ (i.e., if $F$ is holomorphic), then $\partial_c f$ is called the \emph{Cullen} (or complex) \emph{derivative} of $f$ and denoted also by the symbol $f'_c$. Since $\frac{\partial F}{\partial z}$ is holomorphic in this case, clearly $f'_c$ still belongs to $\mathcal{SR}(\Omega_D)$. This is in contrast with the fact that $f^\circ_s,f'_s$ are only slice regular when they are locally constant, see \cite[Remark 5]{perotti}. In general, a slice function $f\in\mathcal{S}^1(\Omega_D)$ is slice regular if, and only if, $\overline{\partial}_c f$ vanishes identically.

\begin{ex}[{$\mathcal{S}^{pol}(Q_A)$}]\label{exm:Spol}
    Let $\ell,m\in\mathbb{N}$. Both $x\mapsto x^\ell$ and $x\mapsto (x^c)^m$ are slice-preserving slice functions $Q_A\to A$, induced by the polynomial stem functions $z\mapsto z^\ell$ and $z\mapsto \bar z^m$ from $\mathbb{R}_\mathbb{C}$ to itself (see~\cite[Examples 2]{perotti}). So is their product $x^\ell\cdot(x^c)^m=x^\ell(x^c)^m$, which is induced by the polynomial stem function $z\mapsto z^\ell\bar z^m$. Taking into account that $\mathcal{I}$ is a right $A$-module isomorphism, we conclude that $\mathcal{S}^{pol}(Q_A)\coloneqq\mathcal{I}(Stem^{pol}(\mathbb{R}_\mathbb{C},A_\mathbb{C}))$ is the set of all functions $f:Q_A\to A$ of the form \[f(x)=\sum_{\ell,m=0}^nx^\ell(x^c)^ma_{\ell m}\]
    with $n\in\mathbb{N}$ and $\{a_{\ell m}\}_{\ell,m=0}^n\subset A$. This $f$ is induced by the stem function $F(z)=\sum_{\ell,m=0}^nz^\ell\bar z^ma_{\ell m}$. Moreover,
    \begin{equation*}
        \partial_c f(x)=\sum_{\ell,m=0}^n\ell x^{\ell-1}(x^c)^ma_{\ell m}\,,\qquad
        \overline{\partial}_c f(x)=\sum_{\ell,m=0}^nm x^\ell(x^c)^{m-1}a_{\ell m}\,.
    \end{equation*}
    On the other hand,
    \begin{equation*}
        f^\circ_s(x)=\sum_{\ell,m=0}^n\operatorname{Re}(x^\ell(x^c)^m)a_{\ell m}\,,\qquad
        f'_s(x)=(\operatorname{Im}(x))^{-1}\sum_{\ell,m=0}^n\operatorname{Im}(x^\ell(x^c)^m)a_{\ell m}\,.
    \end{equation*}
Clearly, $\mathcal{S}^{pol}(Q_A)$ is the right $A$-module obtained by  restricting to $Q_A$ all elements of the right $A$-module $A[x,x^c]$. The $\ast$-algebra structure on $\mathcal{S}^{pol}(Q_A)\subset\mathcal{S}(Q_A)$, too, is consistent with that of $A[x,x^c]$. This is because the equalities $z^{\ell'}\bar z^{m'}a_{\ell' m'}z^{\ell''}\bar z^{m''}b_{\ell'' m''}=z^{\ell'+\ell''}\bar z^{m'+m''}a_{\ell' m'}b_{\ell'' m''}$ and $(z^{\ell'}\bar z^{m'}a_{\ell' m'})^c=z^{\ell'}\bar z^{m'}a_{\ell' m'}^c$ immediately imply both $(x^{\ell'}(x^c)^{m'}a_{\ell' m'})\cdot(x^{\ell''}(x^c)^{m''}b_{\ell'' m''})=x^{\ell'+\ell''}(x^c)^{m'+m''}a_{\ell' m'}b_{\ell'' m''}$ and $(x^{\ell'}(x^c)^{m'}a_{\ell' m'})^c=x^{\ell'}(x^c)^{m'}a_{\ell' m'}^c$.
\end{ex}

\begin{ex}[{$\mathcal{SR}^{pol}(Q_A)$}]\label{exm:SRpol}
    By definition, $\mathcal{SR}^{pol}(Q_A)\coloneqq\mathcal{S}^{pol}(Q_A)\cap\mathcal{SR}(Q_A)$. Its elements are exactly the functions $f\in\mathcal{S}^{pol}(Q_A)$ with $\overline{\partial}_cf\equiv0$. Thus, $\mathcal{SR}^{pol}(Q_A)$ is the set of all restrictions to $Q_A$ of elements of $A[x]$: namely, functions $f:Q_A\to A$ of the form
    \[f(x)=x^na_n+\ldots+xa_1+a_0\]
    with $n\in\mathbb{N}$ and $\{a_\ell\}_{\ell=0}^n\subset A$. In this situation, $f=\mathcal{I}(F)$ where $F(z)=z^na_n+\ldots+za_1+a_0$. Thus, $\mathcal{SR}^{pol}(Q_A)=\mathcal{I}(Stem^{pol}(\mathbb{R}_\mathbb{C},A_\mathbb{C})\cap HolStem(\mathbb{R}_\mathbb{C},A_\mathbb{C}))$. Moreover,
    \begin{equation*}
        f'_c(x)=nx^{n-1}a_n+\ldots+2xa_2+a_1\,.
    \end{equation*}
    On the other hand,
    \begin{align*}
        f^\circ_s(x)&=\operatorname{Re}(x^n)a_n+\ldots+\operatorname{Re}(x^2)a_2+\operatorname{Re}(x)a_1+a_0\,,\\
        f'_s(x)&=(\operatorname{Im}(x))^{-1}\operatorname{Im}(x^n)a_n+\ldots+(\operatorname{Im}(x))^{-1}\operatorname{Im}(x^2)a_2+a_1\,.
    \end{align*}
    Finally, the $A$-module structure and the $\ast$-algebra structure on $\mathcal{SR}^{pol}(Q_A)\subset\mathcal{S}(Q_A)$ are consistent with those of $A[x]$. For instance: if $y=\alpha+I\beta\in Q_A$ (with $\alpha,\beta\in\mathbb{R},I\in\mathbb{S}_A$) and $f(x)\coloneqq x-y$, then $f^c(x)=x-y^c$ and $N(f)(x)=x^2-xt(y)+n(y)=(x-\alpha)^2+\beta^2$. 
\end{ex}

\subsection{Quaternionic slice and slice regular functions}\label{subsec:quaternionicslicefunctions}

In this subsection we focus only on quaternions and overview some standard properties of quaternionic slice and slice regular functions. We begin with some properties of the zero set $V(f)$ of any quaternionic slice function $f$.

\begin{lem}\label{lem:quaternioniczeros}
    Let $\Omega_D$ be a circular domain in $A=\mathbb{H}$ and let $\alpha,\beta\in\mathbb{R}$ (with $\beta>0$) be such that $\alpha+\beta\mathbb{S}_\mathbb{H}\subset\Omega_D$. If $f\in\mathcal{S}(\Omega_D)$, one of the following facts holds true:
    \begin{enumerate}
        \item $V(f)\cap(\alpha+\beta\mathbb{S}_\mathbb{H})=\emptyset$;
        \item $V(f)\cap(\alpha+\beta\mathbb{S}_\mathbb{H})$ is a singleton and $f'_s$ does not vanish in $\alpha+\beta\mathbb{S}_\mathbb{H}$;
        \item $\alpha+\beta\mathbb{S}_\mathbb{H}\subset V(f)$ and $f^\circ_s,f'_s$ both vanish identically in $\alpha+\beta\mathbb{S}_\mathbb{H}$.
    \end{enumerate}
    In each of the former cases, respectively,
    \begin{enumerate}
        \item $V(f^c)\cap(\alpha+\beta\mathbb{S}_\mathbb{H})=\emptyset=V(N(f))\cap(\alpha+\beta\mathbb{S}_\mathbb{H})$;
        \item $V(f^c)\cap(\alpha+\beta\mathbb{S}_\mathbb{H})$ is a singleton and $V(N(f))\supseteq\alpha+\beta\mathbb{S}_\mathbb{H}$;
        \item $\alpha+\beta\mathbb{S}_\mathbb{H}$ is contained both in $V(f^c)$ and in $V(N(f))$.
    \end{enumerate}
\end{lem}

The first statement in Lemma~\ref{lem:quaternioniczeros} is an immediate consequence of formula~\eqref{eq:sphericalrepresentationformula}, if we take into account the fact that there are no zero divisors in $\mathbb{H}$. The second statement then follows from~\eqref{eq:sliceconjugate} and~\eqref{eq:normalfunction}. For the special case of Lemma~\ref{lem:quaternioniczeros} when $f$ is slice regular, we refer the reader to~\cite[Chapter 3]{librospringer2}, which also includes the next result.

\begin{thm}
  Assume $\Omega_D$ to be a slice domain in $A=\mathbb{H}$ and let $f\in\mathcal{SR}(\Omega_D)$, with $f$ not identically zero. Then the zero set $V(f)$ of $f$ consists of isolated points or isolated 2-spheres of the form $\alpha+\beta\mathbb{S}_\mathbb{H}$.
\end{thm}

When $\Omega_D$ is a product domain, the zero sets can be more peculiar, see~\cite{gporientation}.

It is also possible to consider quotients of slice functions, in the following sense.

\begin{defn}
Let $\Omega_D$ be a circular domain in $A=\mathbb{H}$ and let $f\in\mathcal{S}(\Omega_D)$. If $N(f)\not\equiv0$, then the \emph{slice reciprocal} $f^{-\bullet}$ is defined as 
\begin{equation*}
    f^{-\bullet}(x)\coloneqq (N(f)(x))^{-1}f^c(x)
\end{equation*}
for every $x\in\Omega_D\setminus V(N(f))$. We also set the notation $f^{-\bullet \ell}\coloneqq (f^{-\bullet})^{\bullet\ell}$ for all $\ell\in\mathbb{N}$.
\end{defn}

When $h\in\mathcal{S}(\Omega_D)$ is slice-preserving, it is not hard to prove that $h^{-\bullet}$ is a slice-preserving element of $\mathcal{S}(\Omega_D)$ whose value at any $x\in\Omega_D\setminus V(h)$ is $h(x)^{-1}$. This applies, in particular, to $h=N(f)$, whence $f^{-\bullet}=N(f)^{-\bullet}\cdot f^c$ is an element of $\mathcal{S}(\Omega_D\setminus V(N(f)))$. Moreover, $f\cdot f^{-\bullet}=f^{-\bullet}\cdot f\equiv1$. For all $g\in\mathcal{S}(\Omega_D)$, we also have
\begin{align*}
   (f^{-\bullet}\cdot g)(x)&=(N(f)^{-\bullet}\cdot f^c\cdot g)(x)=(N(f)(x))^{-1}(f^c\cdot g)(x)\\
   (g\cdot f^{-\bullet})(x)&=(g\cdot N(f)^{-\bullet}\cdot f^c)(x)=(N(f)^{-\bullet}\cdot g\cdot f^c)(x)=(N(f)(x))^{-1}(g\cdot f^c)(x)
\end{align*}
for all $x\in\Omega_D\setminus V(N(f))$.

We now study the real differential of a quaternionic slice function $f$ induced by a $C^1$ stem function. If $f$ is slice regular, the following result is well-known (see, e.g.,~\cite[Remark 8.15]{librospringer2}).

\begin{lem}\label{lem:differential}
    Let $\Omega_D$ be a slice domain in $A=\mathbb{H}$ and let $f\in\mathcal{SR}(\Omega_D)$. Let $p=\phi_I(z)\in\mathbb{C}_I\setminus\mathbb{R}$, for some $I\in\mathbb{S}_\mathbb{H},z\in D$ and let us decompose $T_p\Omega_D\simeq\mathbb{H}$ as $\mathbb{C}_I\oplus\mathbb{C}_I^\perp$.
    Then
    \begin{equation*}
        df_{p}(v+w)=v f'_c(p)+w f'_s(p),\qquad\forall v\in\mathbb{C}_I,\quad w\in\mathbb{C}_I^\perp.
    \end{equation*}
    As a consequence, $\rk df_p\in\{0,2,4\}$. If we pick, instead, $p\in\Omega_D\cap\mathbb{R}$, then $df_{p}(v)=vf'_c(p)$ for all $v\in T_p\Omega_D\simeq\mathbb{H}$ and $\rk df_p\in\{0,4\}$.
\end{lem}

Let us now extend Lemma~\ref{lem:differential} to cover the case when $f$ is not regular. The directional derivative of $f$ at $p\in\mathbb{C}_I$ in the real direction equals $\partial_cf(p)+\overline{\partial}_cf(p)$; the directional derivative of $f$ at $p$ along $I$ equals $I\,(\partial_cf(p)-\overline{\partial}_cf(p))=I\,\partial_cf(p)+I^c\,\overline{\partial}_cf(p)$; if $p\not\in\mathbb{R}$, the directional derivative of $f$ at $p$ along any $w\in\mathbb{C}_I^\perp$ still equals $w\,\partial_sf(p)$. We can therefore make the next remark.

\begin{rmk}\label{rmk:differential}
    Let $\Omega_D$ be a circular domain in $A=\mathbb{H}$, let $p=\phi_I(z)\in\mathbb{C}_I\setminus\mathbb{R}$, for some $I\in\mathbb{S}_\mathbb{H},z\in D$ and let us decompose $T_p\Omega_D\simeq\mathbb{H}$ as $\mathbb{C}_I\oplus\mathbb{C}_I^\perp$. If $f=\mathcal{I}(F)\in\mathcal{S}^1(\Omega_D)$, then
    \begin{align*}
        df_{p}(v+w)&=\alpha(\partial_c+\overline{\partial}_c)f(p)+\beta I(\partial_c-\overline{\partial}_c)f(p)+w\,\partial_sf(p)\\
        &=v\,\partial_cf(p)+v^c\,\overline{\partial}_cf(p)+w\,\partial_sf(p)
    \end{align*}
    for all $v=\alpha+I\beta\in\mathbb{C}_I,\ w\in\mathbb{C}_I^\perp$. In the special case when $f$ is a circular slice function (i.e., $F=F_0,F_1\equiv0$ and $f(\alpha+I\beta)=F_0(\alpha+e_1\beta)$ for all $\alpha+e_1\beta\in D$ and all $I\in\mathbb{S}_\mathbb{H}$), we get $\partial_sf\equiv0$ and
    \begin{equation*}
        df_{p}(\alpha+I\beta+w)=df_{p}(\alpha+I\beta)=\alpha\frac{\partial F_0}{\partial\alpha}(z)+\beta\frac{\partial F_0}{\partial\beta}(z)
        ,\qquad\forall \alpha,\beta\in\mathbb{R},I\in\mathbb{S}_\mathbb{H},\ w\in\mathbb{C}_I^\perp.
    \end{equation*}
    We point out that the last expression does not depend on the choice of $I$ in $\mathbb{S}_\mathbb{H}$ nor on the choice of $w\in\mathbb{C}_I^\perp$.

    If we pick, instead, $p\in\Omega_D\cap\mathbb{R}=D\cap\mathbb{R}$ and $v=\alpha+I\beta\in T_p\Omega_D\simeq\mathbb{H}$, then $df_{p}(v)=v\,\partial_cf(p)+v^c\,\overline{\partial}_cf(p)$ for general  $f\in\mathcal{S}^3(\Omega_D)\subset C^1(\Omega_D,A)$ and $df_{p}(v)=df_{p}(\alpha+I\beta)=\alpha\frac{\partial F_0}{\partial\alpha}(p)+\beta\frac{\partial F_0}{\partial\beta}(p)$ for any circular element $f=\mathcal{I}(F_0)$ of $\mathcal{S}^3(\Omega_D)$.
\end{rmk}

\subsection{Dual quaternionic slice regular functions}\label{subsec:dualquatenionicslicefunctions}

In this subsection we focus only on the real $\ast$-algebra (and bilateral $\mathbb{H}$-module) $\mathbb{DH}=\mathbb{H}+\mathbb{H}\epsilon=\mathbb{H}+\epsilon\mathbb{H}$ of dual quaternions. For stem functions, we will heavily use the decomposition $\mathbb{DH}_\mathbb{C}=\mathbb{H}_\mathbb{C}+\epsilon\mathbb{H}_\mathbb{C}$ and the fact that $\mathbb{H}_\mathbb{C}$ is a real $\ast$-subalgebra (and a bilateral $\mathbb{H}_\mathbb{C}$-submodule) of $\mathbb{DH}_\mathbb{C}$. We denote by $\pi_{\mathbb{H}_\mathbb{C}}:\mathbb{DH}_\mathbb{C}\to\mathbb{H}_\mathbb{C}$ the natural projection $\pi_{\mathbb{H}_\mathbb{C}}(a+e_1b)=\pi_\mathbb{H}(a)+e_1\pi_\mathbb{H}(b)$.

Slice and slice regular functions over $\mathbb{DH}$ have been studied in~\cite{gstdualquaternions}, whence we recall several results and definitions.

\begin{rmk}
    Let $f,g\in\mathcal{S}(\Omega_D)$. If $f\in\mathcal{S}_\mathbb{DR}(\Omega_D)$, then $(f\cdot g)(x)=f(x)g(x)=(g\cdot f)(x)$ for all $x\in\Omega_D$. In particular, $\mathcal{S}_\mathbb{DR}(\Omega_D)$ is the center of $\mathcal{S}(\Omega_D)$ and $\mathcal{SR}_\mathbb{DR}(\Omega_D)$ is the center of $\mathcal{SR}(\Omega_D)$.
\end{rmk}

\begin{lem}\label{lem:dualquaternioniczeros}
    Let $\Omega_D$ be a circular domain in $A=\mathbb{DH}$ and let $\alpha,\beta\in\mathbb{R}$ (with $\beta\neq0$) be such that $\alpha+\beta\mathbb{S}_\mathbb{DH}\subset\Omega_D$. If $f\in\mathcal{S}(\Omega_D)$, one of the following facts holds true:
    \begin{enumerate}
        \item $V(f)\cap(\alpha+\beta\mathbb{S}_\mathbb{DH})=\emptyset$;
        \item $V(f)\cap(\alpha+\beta\mathbb{S}_\mathbb{DH})$ is a singleton and the constant values of $f^\circ_s,f'_s$ in $\alpha+\beta\mathbb{S}_\mathbb{DH}$ are invertible;
        \item $V(f)\cap(\alpha+\beta\mathbb{S}_\mathbb{DH})=T_{\alpha+I\beta}$ for some $I\in\mathbb{S}_\mathbb{DH}$ and the constant values of $f^\circ_s,f'_s$ in $\alpha+\beta\mathbb{S}_\mathbb{DH}$ belong to $\epsilon\mathbb{H}^*$;
        \item $\alpha+\beta\mathbb{S}_\mathbb{DH}\subseteq V(f)$ and $f^\circ_s,f'_s$ both vanish identically in $\alpha+\beta\mathbb{S}_\mathbb{DH}$.
    \end{enumerate}
\end{lem}

\begin{ex}\label{ex:deltaslice}
    Fix $y\in Q_\mathbb{DH}\setminus\mathbb{R}$ and set $f(x)\coloneqq x-y$, which yields $f\in\mathcal{SR}^{pol}(Q_\mathbb{DH})$. Then $\epsilon f(x)=\epsilon(x-y_1)$ and $f^c(x)=x-y^c$. Moreover, $N(f)(x)=x^2-xt(y)+n(y)$ is an element of $\mathcal{SR}^{pol}_\mathbb{R}(Q_\mathbb{DH})$, which we denote by $\Delta_y(x)$. It is the restriction to $Q_\mathbb{DH}$ of the minimal polynomial of $y$. We have $V(f)=\{y\}$, $V(\epsilon f)=T_{y_1}$ and $V(\Delta_y)=\mathbb{S}^\mathbb{DH}_y$.
\end{ex}

Lemma~\ref{lem:dualquaternioniczeros} was proven in~\cite[Theorem 5.2]{gstdualquaternions}. It is a consequence of formula~\eqref{eq:sphericalrepresentationformula}, if we take into account the fact that $\epsilon\mathbb{H}^*$ is the set of zero divisors of $\mathbb{DH}$.

\begin{defn}
    Given a stem function $F\in Stem(D,\mathbb{DH}_\mathbb{C})$, we define $^\pi\!F\in Stem(D,\mathbb{H}_\mathbb{C})$ as $^\pi\!F\coloneqq \pi_{\mathbb{H}_\mathbb{C}}\circ F$.
    
    Given a slice function $f=\mathcal{I}_\mathbb{DH}(F):\Omega_D\to\mathbb{DH}$, its \emph{primal part} is the quaternionic slice function $^\pi\!f:\Omega_D\cap\mathbb{H}\to\mathbb{H}$ defined as the value at $^\pi\!F$ of $\mathcal{I}_\mathbb{H}:Stem(D,\mathbb{H}_\mathbb{C})\to\mathcal{S}(\Omega_D\cap\mathbb{H})$.
\end{defn}
Moreover,~\cite[Remark 4.7]{gstdualquaternions} proved that $\pi_\mathbb{H}\circ f = {^\pi\!f}\circ(\pi_\mathbb{H})_{|_{\Omega_D}}$. Taking into account that $\epsilon(y_1+\epsilon y_2)=\epsilon y_1$ for all $y_1,y_2\in\mathbb{H}$, we conclude that $\epsilon h(x_1+\epsilon x_2)=\epsilon\,^\pi\!h(x_1)$ for all $h\in\mathcal{S}(\Omega_D)$ and for all $x_1+\epsilon x_2\in\Omega_D$.

Our next aim is decomposing dual quaternionic stem functions and slice functions into appropriately chosen ``quaternionic'' components. We begin with stem functions.

\begin{rmk}
    $Stem(D,\mathbb{H}_\mathbb{C})$ is a real $^\ast$-subalgebra of $Stem(D,\mathbb{DH}_\mathbb{C})$. This is a consequence of the fact that addition, multiplication and conjugation of stem functions are defined pointwise. For the same reason,
    \begin{equation}\label{eq:stemdecomposition}
        Stem(D,\mathbb{DH}_\mathbb{C})=Stem(D,\mathbb{H}_\mathbb{C})+\epsilon Stem(D,\mathbb{H}_\mathbb{C})=Stem(D,\mathbb{H}_\mathbb{C})+ Stem(D,\mathbb{H}_\mathbb{C})\epsilon\,,
    \end{equation}
    where both sums are direct. Finally, $Stem(D,\mathbb{H}_\mathbb{C})=\{\pi_{\mathbb{H}_\mathbb{C}}\circ F: F\in Stem(D,\mathbb{DH}_\mathbb{C})\}$.
\end{rmk}

We are now in a position to define a special subclass within the class of slice functions over dual quaternions.

\begin{defn}
    We denote the image through $\mathcal{I}_\mathbb{DH}:Stem(D,\mathbb{DH}_\mathbb{C})\to\mathcal{S}(\Omega_D)$ of the real $^\ast$-subalgebra (and right $\mathbb{H}$-module) $Stem(D,\mathbb{H}_\mathbb{C})$ by the symbol $\mathcal{S}_{\mathbb{H}}(\Omega_D)$ and call its elements \emph{quaternion-preserving} slice functions $\Omega_D\to\mathbb{DH}$. We also set $\mathcal{S}^k_\mathbb{H}(\Omega_D)\coloneqq\mathcal{S}^k(\Omega_D)\cap \mathcal{S}_\mathbb{H}(\Omega_D)$, for $k\in\mathbb{N}\cup\{\infty,\omega\}$. Finally, we denote by $\mathcal{SR}_{\mathbb{H}}(\Omega_D)$ the real $^\ast$-subalgebra (and right $\mathbb{H}$-submodule) of $\mathcal{S}_{\mathbb{H}}(\Omega_D)$ consisting of slice regular functions.
\end{defn}

Our terminology is justified by the next remark, where we take into account~\eqref{eq:domainprojection} and denote by $\Omega_D$ the domain $\Omega_{D,\mathbb{DH}}$.

\begin{rmk}
    If $f\in\mathcal{S}_{\mathbb{H}}(\Omega_D)$, then $f(\Omega_D\cap\mathbb{H})\subseteq\mathbb{H}$. Moreover, the restriction of $f$ to $\Omega_D\cap\mathbb{H}$ is $^\pi\!f$.
\end{rmk}

It is easy to see that a circular slice function $f$, i.e., a slice function induced by a $\mathbb{DH}$-valued stem function $F=F_0$, is quaternion-preserving if, and only if, it takes values in $\mathbb{H}$. This is because $\mathbb{H}=\mathbb{DH}\cap\mathbb{H}_\mathbb{C}$. In general, quaternion-preserving slice functions can be characterized as follows.

\begin{rmk}
    A slice function $f=\mathcal{I}(F)$ is quaternion-preserving if, and only if, $f^\circ_s,f'_s$ take values in $\mathbb{H}$. This follows from two facts: on the one hand, $F$ takes values in $\mathbb{H}_\mathbb{C}$ if, and only if, $F_0,F_1$ take values in $\mathbb{H}$; on the other hand $f^\circ_s,f'_s$ are circular slice functions, induced respectively by $F_0$ and by $\alpha+e_1\beta\mapsto F_1(\alpha+e_1\beta)/\beta$.
\end{rmk}

In particular, the last remark tells us that case {\it 3} in Lemma~\ref{lem:dualquaternioniczeros} never occurs when $f$ is quaternion-preserving.

We are now ready for the announced decomposition of dual quaternionic slice functions into appropriately chosen ``quaternionic'' components. Taking into account the decomposition~\eqref{eq:stemdecomposition} and the fact that $\mathcal{I}_\mathbb{DH}:Stem(D,\mathbb{DH}_\mathbb{C})\to\mathcal{S}(\Omega_D)$ is a right $\mathbb{DH}$-module isomorphism, we conclude that (for $k\in\mathbb{N}\cup\{\infty,\omega\}$)
\begin{align}
    \mathcal{S}(\Omega_D)&=\mathcal{S}_{\mathbb{H}}(\Omega_D)+\mathcal{S}_{\mathbb{H}}(\Omega_D)\epsilon=\mathcal{S}_{\mathbb{H}}(\Omega_D)+\epsilon\mathcal{S}_{\mathbb{H}}(\Omega_D)\,,\notag\\
    \mathcal{S}^k(\Omega_D)&=\mathcal{S}^k_{\mathbb{H}}(\Omega_D)+\mathcal{S}^k_{\mathbb{H}}(\Omega_D)\epsilon=\mathcal{S}^k_{\mathbb{H}}(\Omega_D)+\epsilon\mathcal{S}^k_{\mathbb{H}}(\Omega_D)\,,\notag\\
    \mathcal{SR}(\Omega_D)&=\mathcal{SR}_{\mathbb{H}}(\Omega_D)+\mathcal{SR}_{\mathbb{H}}(\Omega_D)\epsilon=\mathcal{SR}_{\mathbb{H}}(\Omega_D)+\epsilon\mathcal{SR}_{\mathbb{H}}(\Omega_D)\,,\label{eq:decomposition of slice regular}
\end{align}
where all sums are direct.

\begin{defn}
    Given $f\in\mathcal{S}(\Omega_D)$, the unique elements $g$ and $h$ of $\mathcal{S}_\mathbb{H}(\Omega_D)$ such that $f=g+\epsilon h$ are called the first and second \emph{quaternion-preserving components} of $f$.
\end{defn}

By construction, the first quaternion-preserving component of $f=\mathcal{I}_\mathbb{DH}(F)$ is $g=\mathcal{I}_\mathbb{DH}(\pi_{\mathbb{H}_\mathbb{C}}\circ F)=\mathcal{I}_\mathbb{DH}(^\pi\!F)$. Thus, $g$ is the (unique) slice extension to $\Omega_D$ of the primal part $^\pi\!f=\mathcal{I}_\mathbb{H}(^\pi\!F):\Omega_D\cap\mathbb{H}\to \mathbb{H}$ of $f$.
If $f=g+\epsilon h\in\mathcal{S}(\Omega_D)$, then $f^\circ_s=g^\circ_s+\epsilon h^\circ_s$ and $f'_s=g'_s+\epsilon h'_s$ where $g^\circ_s,g'_s,h^\circ_s,h'_s$ are $\mathbb{H}$-valued circular slice functions, whence circular elements of $\mathcal{S}_\mathbb{H}(\Omega_D)$. This allows a new interpretation of the cases listed in Lemma~\ref{lem:dualquaternioniczeros}.

\begin{rmk}
    Let $f=g+\epsilon h\in\mathcal{S}(\Omega_D)$, with $g,h\in\mathcal{S}_\mathbb{H}(\Omega_D)$. If the set $V(f)\cap(\alpha+\beta\mathbb{S}_\mathbb{DH})$ is a singleton, then $g'_s\neq0$ in $\alpha+\beta\mathbb{S}_\mathbb{DH}$. If it equals $T_{\alpha+I\beta}$ for some $I\in\mathbb{S}_\mathbb{DH}$, then $g^\circ_s\equiv0\equiv g'_s$ in $\alpha+\beta\mathbb{S}_\mathbb{DH}$ and $h^\circ_s\neq0\neq h'_s$ in $\alpha+\beta\mathbb{S}_\mathbb{DH}$. Finally, if $\alpha+\beta\mathbb{S}_\mathbb{DH}\subseteq V(f)$, then $g^\circ_s=g'_s=h^\circ_s=h'_s\equiv0$ in $\alpha+\beta\mathbb{S}_\mathbb{DH}$.
\end{rmk}

We now prove a new result.

\begin{lem}\label{lem:decompositionusingsphericalderivative}
    Let $\Omega_D$ be a circular domain in $Q_{\mathbb{DH}}$ and let $f=g+\epsilon h\in\mathcal{S}(\Omega_D)$ (with $g,h\in\mathcal{S}_\mathbb{H}(\Omega_D)$). Then, for any $x=x_1+\epsilon x_2\in\Omega_D\setminus\mathbb{R}$, it holds
    \begin{equation*}
        f(x)=f(x_1)+\epsilon x_2 f'_s(x_1)=f(x_1)+\epsilon x_2 (^\pi\!f)'_s(x_1)=g(x_1)+\epsilon (h(x_1)+x_2 g'_s(x_1)).
    \end{equation*}
    In particular: if $g$ is circular, then $x_1+\epsilon x_2\mapsto f(x_1+\epsilon x_2)=g(x_1)+\epsilon h(x_1)$ is constant in $x_2$.
\end{lem}
\begin{proof}
    Formula~\eqref{eq:sphericalrepresentationformula} guarantees that $f(x)=f^\circ_s(x)+\operatorname{Im}(x)f'_s(x)$. Taking into account that $f^\circ_s,f'_s$ are circular slice functions, the equality $\mathbb{S}_{x}^\mathbb{DH}=\mathbb{S}_{x_1}^\mathbb{DH}$ and the second among equalities~\eqref{eq:realandimaginarypartdq}, we find that
    \begin{equation*}
        f(x)=f^\circ_s(x_1)+\left(\operatorname{Im}(x_1)+\epsilon x_2\right)f'_s(x_1)=f(x_1)+\epsilon x_2 f'_s(x_1).
    \end{equation*}
    Here, $\epsilon f'_s(x_1)=\epsilon g'_s(x_1)=\epsilon\,{^\pi\!f'_s}(x_1)$ and $f(x_1)=g(x_1)+\epsilon h(x_1)$. The first statement follows. Finally: if $g$ is circular, then $g'_s\equiv0$ and the general formula reduces to $f(x)=f(x_1)=g(x_1)+\epsilon h(x_1)$.
\end{proof}

Let $f=g+\epsilon h\in\mathcal{S}(\Omega_D)$ (with $g,h\in\mathcal{S}_\mathbb{H}(\Omega_D)$) and let $x=x_1+\epsilon x_2\in\Omega_D$ (with $x_1,x_2\in\mathbb{H}$). Then
\begin{align}\label{eq:primalpart}
    \pi_\mathbb{H}(f(x))&={^\pi\!f(x_1)}=g(x_1)=\pi_\mathbb{H}(g(x))\,,\\
    \pi_{\epsilon\mathbb{H}}(f(x))&=\pi_{\epsilon\mathbb{H}}(g(x))+\epsilon h(x)=
    \begin{cases}
        \epsilon (h(x_1)+x_2 g'_s(x_1)) & \text{if\ }x\in\Omega_D\setminus\mathbb{R}\\
        \epsilon h(x_1)& \text{if\ }x\in\Omega_D\cap\mathbb{R}
    \end{cases}\,.
\end{align}

%%%%%%%%%%%%%%%%%%%%%%%%%%%%%

\section{Beyond the quadratic cone}\label{sec:beyond}

While slice regular functions over dual quaternions have been studied in detail in~\cite{gstdualquaternions}, an interest remains in classes of $\mathbb{DH}$-valued functions with domains not limited to the quadratic cone $Q_{\mathbb{DH}}$. An example are the functions $\mathbb{DH}\to\mathbb{DH}$ associated to dual quaternionic polynomials. This interest will lead us to the introduction of a new function class in the forthcoming Section~\ref{sec:dualslice}. In preparation, the present section  establishes a useful decomposition of $\mathbb{DH}$ and sets up appropriate domains.

Our decomposition of $\mathbb{DH}$ is analogous to~\eqref{eq:decompositionofcone}, with $Q_A$ replaced by $\mathbb{DH}$ and the ``complex slices'' $\mathbb{C}_J$ replaced by ``dual complex slices'', defined as follows.

\begin{defn}
    Set $\mathbb{R}_\mathbb{C}^+\coloneqq\{a_1+e_1b_1:a_1,b_1\in\mathbb{R},b_1>0\}$ and $\mathbb{DR}_\mathbb{C}^+\coloneqq\mathbb{R}_\mathbb{C}^++\epsilon\mathbb{R}_\mathbb{C}=\{a_1+\epsilon a_2+e_1(b_1+\epsilon b_2):a_1,b_1,a_2,b_2\in\mathbb{R},b_1>0\}$. For any $J=J_1+\epsilon J_2\in\mathbb{S}_\mathbb{DH}$, we set
    \[\mathbb{DC}_J\coloneqq\phi_J(\mathbb{DR}_\mathbb{C})=\mathbb{C}_J+\epsilon\mathbb{C}_J=\spn(1,J,\epsilon,\epsilon J)=\spn(1,J,\epsilon,\epsilon J_1)\,.\]
    We also set $\mathbb{C}_J^+\coloneqq\phi_J(\mathbb{R}_\mathbb{C}^+)$ and $\mathbb{DC}_J^+\coloneqq\phi_J(\mathbb{DR}_\mathbb{C}^+)$.
\end{defn}

\begin{prop}\label{prop:decompositionintoDCs}
    For every choice of $I,J\in\mathbb{S}_\mathbb{DH}$, the following properties hold true.
    \begin{enumerate}
        \item $\mathbb{DC}_J$ is a $\ast$-subalgebra of $\mathbb{DH}$, isomorphic to the commutative $\ast$-algebra $\mathbb{DR}_\mathbb{C}$. In particular this is true for $\mathbb{DC}_i$, which is the $\ast$-algebra of dual complex numbers $\mathbb{DC}\coloneqq \mathbb{C}+\epsilon\mathbb{C}$.
        \item If $I=\pm J$, then $\mathbb{DC}_I=\mathbb{DC}_J$. If $I\neq\pm J$ but $I_1=\pm J_1$, then
        \begin{equation*}
            \mathbb{DR}\subset\mathbb{DC}_I\cap\mathbb{DC}_J=\spn(1,\epsilon,\epsilon I_1)\subset\mathbb{R}+\epsilon\mathbb{H}\,.
        \end{equation*}
        If $I_1\neq\pm J_1$, then $\mathbb{DR}=\mathbb{DC}_I\cap\mathbb{DC}_J\subset\mathbb{R}+\epsilon\mathbb{H}$. Finally, $\mathbb{DC}_I^+\cap\mathbb{DC}_J^+=\emptyset$ if $I\neq J$.
        \item $\mathbb{DH}=\bigcup_{J\in\mathbb{S}_\mathbb{DH}}\mathbb{DC}_J$ and the map 
        \begin{equation*}
            \Phi:\mathbb{S}_\mathbb{DH}\times\mathbb{DR}_\mathbb{C}\to\mathbb{DH},\quad(J,a+e_1b)\mapsto\phi_J(a+e_1b)=a+Jb
        \end{equation*}
        is surjective. So is the restriction $\Phi:\mathbb{S}_\mathbb{DH}\times\mathbb{R}_\mathbb{C}\to Q_\mathbb{DH}$. Moreover, for all $a,b=b_1+\epsilon b_2\in\mathbb{DR}$ with $b_1>0$ and for all $J\in\mathbb{S}_\mathbb{DH}$, the following equalities hold true:
        \begin{align*}
            &\Phi^{-1}(a+Jb)=\{(J,a+e_1b),(-J,a-e_1b)\}\,,\\
            &\Phi^{-1}(a+J\epsilon b_2)=(T_{J_1}\times\{a+e_1\epsilon b_2\})\cup(T_{-J_1}\times\{a-e_1\epsilon b_2\})\,,\\
            &\Phi^{-1}(a)=\mathbb{S}_\mathbb{DH}\times\{a\}\,.
        \end{align*}
        As a consequence, the following restrictions are homeomorphisms:
        \begin{align*}
            &\Phi:\mathbb{S}_\mathbb{DH}\times\mathbb{DR}_\mathbb{C}^+\to\bigcup_{J\in\mathbb{S}_\mathbb{DH}}\mathbb{DC}_J^+=\mathbb{DH}\setminus(\mathbb{R}+\epsilon\mathbb{H})\,,\\
            &\Phi:\mathbb{S}_\mathbb{DH}\times\mathbb{R}_\mathbb{C}^+\to\bigcup_{J\in\mathbb{S}_\mathbb{DH}}\mathbb{C}_J^+=Q_\mathbb{DH}\setminus\mathbb{R}\,.
        \end{align*}
    \end{enumerate}
\end{prop}

\begin{proof}
    \begin{enumerate}
        \item Using Remark~\ref{rmk:phiJ} and the fact that $a,b\in\mathbb{DR},a+Jb=0$ imply $a=0=b$, it is easy to see that the restriction of $\phi_J$ to $\mathbb{DR}_\mathbb{C}$ is a $^\ast$-algebra isomorphism from $(\mathbb{DR}_\mathbb{C},+,\cdot,\bar{\phantom{z}})$ to $(\mathbb{DC}_J,+,\cdot,\,^c)$.
        \item The inclusion $\mathbb{DC}_I\cap\mathbb{DC}_J\supseteq\spn(1,\epsilon)=\mathbb{DR}$ follows by definition. We are left with studying the intersection $\spn(I,\epsilon I_1)\cap\spn(J,\epsilon J_1)$. For $\lambda,\mu\in\mathbb{R}$, we have $\epsilon I_1=\lambda J+\mu J_1=\lambda J_1+\epsilon(\lambda J_2+\mu J_1)$ if, and only if, $\lambda=0,I_1=\mu\epsilon J_1$, which is the same as $I_1=\pm J_1$. Similarly, for $\lambda,\mu\in\mathbb{R}$, the equality $I=I_1+\epsilon I_2=\lambda J+\mu\epsilon J_1=\lambda J_1+\epsilon(\lambda J_2+\mu J_1)$ is equivalent to $I_1=\lambda J_1, I_2=\lambda J_2+\mu J_1$, which is the same as $I_1=\pm J_1, I_2=\pm J_2+\mu J_1$. Taking into account that $I_1\perp I_2$, the last equalities are equivalent to $I=\pm J$.\\
        Since $\mathbb{DC}_J^+\cap(\mathbb{R}+\epsilon\mathbb{H})=\emptyset$, it follows at once that $\mathbb{DC}_I^+\cap\mathbb{DC}_J^+=\emptyset$ when $I\neq J$.
        \item The set $\Phi(\mathbb{S}_\mathbb{DH}\times\mathbb{DR}_\mathbb{C}^+)=\bigcup_{J\in\mathbb{S}_\mathbb{DH}}\mathbb{DC}_J^+$ coincides with $\mathbb{DH}\setminus(\mathbb{R+\epsilon\mathbb{H}})$ because every $x=x_1+\epsilon x_2\in\mathbb{DH}\setminus(\mathbb{R+\epsilon\mathbb{H}})$ can be expressed as $x=a+Jb=a+bJ$, where $a\coloneqq\frac{t(x)}{2}=\operatorname{Re}(x_1)+\epsilon\operatorname{Re}(x_2)\in\mathbb{DR}$, $b\coloneqq|\operatorname{Im}(x_1)|+\epsilon\frac{\langle\operatorname{Im}(x_1),x_2\rangle}{|\operatorname{Im}(x_1)|}\in\mathbb{DR}$ and $J\coloneqq(x-a)b^{-1}=b^{-1}(x-a)=(2b)^{-1}(x-x^c)$. This $J$ is, indeed, an element of $\mathbb{S}_\mathbb{DH}$ because
        \begin{align*}
            t(J)&=(2b)^{-1}t(x-x^c)=(2b)^{-1}(x+x^c-x^c-x)=0\,,\\
            n(J)&=(2b)^{-2}n(x-x^c)=b^{-2}n(\operatorname{Im}(x_1)+\epsilon\operatorname{Im}(x_2))\\
            &=(|\operatorname{Im}(x_1)|^2+2\epsilon\langle\operatorname{Im}(x_1),x_2\rangle)^{-1}(|\operatorname{Im}(x_1)|^2+2\epsilon\langle\operatorname{Im}(x_1),\operatorname{Im}(x_2)\rangle)=1
        \end{align*}
        Using property {\it 2}, we see that $x=a+Jb\in\mathbb{DC}_J^+$ belongs to $\mathbb{DC}_I$ if, and only if, $I=\pm J$. Thus, such an $x$ has exactly two possible decompositions: namely, $x=a+Jb$ and $x=a+(-J)(-b)$. We conclude that $\Phi^{-1}(a+Jb)=\{(J,a+e_1b),(-J,a-e_1b)\}$ and that $\Phi_{|_{\mathbb{S}_\mathbb{DH}\times\mathbb{DR}_\mathbb{C}^+}}$ is bijective. Its inverse
         \begin{align*}
            &\Psi:\mathbb{DH}\setminus(\mathbb{R+\epsilon\mathbb{H}})\to\mathbb{S}_\mathbb{DH}\times\mathbb{DR}_\mathbb{C}^+,\\
            &x\mapsto\left(\frac{x-x^c}{2}\left(|\operatorname{Im}(x_1)|+\epsilon\frac{\langle\operatorname{Im}(x_1),x_2\rangle}{|\operatorname{Im}(x_1)|}\right)^{-1},\ \frac{x+x^c}{2}+e_1\left(|\operatorname{Im}(x_1)|+\epsilon\frac{\langle\operatorname{Im}(x_1),x_2\rangle}{|\operatorname{Im}(x_1)|}\right)\right)
        \end{align*}
        is clearly continuous. So is the restriction $\Psi:Q_\mathbb{DH}\setminus\mathbb{R}\to\mathbb{S}_\mathbb{DH}\times\mathbb{R}_\mathbb{C}^+$.\\
        Additionally: if $x\in\mathbb{R}+\epsilon(\mathbb{H}\setminus\mathbb{R})$, then $x=a+Jb$ with $a,b\in\mathbb{DR}$ is equivalent to $a=\frac{t(x)}{2}$, $b=\pm\epsilon|\operatorname{Im}(x_2)|$ and $J\in T_{\pm\frac{\operatorname{Im}(x_2)}{|\operatorname{Im}(x_2)|}}$; if $x\in\mathbb{DR}$, then $x=a+Jb$ with $a,b\in\mathbb{DR}$ is equivalent to $a=\frac{t(x)}{2},b=0,J\in\mathbb{S}_\mathbb{DH}$. It follows at once that $\Phi(\mathbb{S}_\mathbb{DH}\times\mathbb{DR}_\mathbb{C})=\bigcup_{J\in\mathbb{S}_\mathbb{DH}}\mathbb{DC}_J$ equals $\mathbb{DH}$ and that $\Phi^{-1}(a+J\epsilon b_2)=(T_{J_1}\times\{a+e_1\epsilon b_2\})\cup(T_{-J_1}\times\{a-e_1\epsilon b_2\})$ and $\Phi^{-1}(a)=\mathbb{S}_\mathbb{DH}\times\{a\}$ for all $a\in\mathbb{DR},\epsilon b_2\in\epsilon\mathbb{R}$.\qedhere
    \end{enumerate}
\end{proof}

We now define and study a useful projection from $\mathbb{DH}$ onto its quadratic cone.

\begin{defn}
     The map $\pi_Q:\mathbb{DH}\to Q_{\mathbb{DH}}$ is defined to make the following diagram commutative:
\begin{center}
    \begin{tikzcd}[ampersand replacement=\&,cramped]
    {\mathbb{S}_\mathbb{DH}\times\mathbb{DR}_\mathbb{C}} \&\& {\mathbb{S}_\mathbb{DH}\times\mathbb{R}_\mathbb{C}} \\
	\& \circlearrowleft \\
	{\mathbb{DH}} \&\& {Q_\mathbb{DH}}
	\arrow["{(id,\pi_{\mathbb{H}_\mathbb{C}})}", from=1-1, to=1-3]
	\arrow["\Phi"', from=1-1, to=3-1]
	\arrow["\Phi", from=1-3, to=3-3]
	\arrow["{\pi_Q}"', from=3-1, to=3-3]
    \end{tikzcd}
\end{center}
In other words, we define $\pi_Q(a+Jb)\coloneqq a_1+Jb_1$ for all $a=a_1+\epsilon a_2,b=b_1+\epsilon b_2\in\mathbb{DR}$ and all $J\in\mathbb{S}_\mathbb{DH}$.
\end{defn}

The map $\pi_Q$ is well defined because: if $b_1\neq0$, then $a'+J'b'=a+Jb$ is equivalent to $a'=a,b'=\pm b,J'=\pm J$, whence $a'_1+J'b'_1=a_1+Jb_1$; if $b_1=0$, then $a'+J'b'=a+Jb$ implies $a'=a,b_1'=0$, whence $a'_1+J'b'_1=a'_1=a_1=a_1+Jb_1$. We now propose an alternative presentation of $\pi_Q$, based upon the decomposition $\mathbb{DH}=\bigcup_{x_1\in\mathbb{H}}H_{x_1}$. We recall that $H_{x_1}\supset T_{x_1}$ for all $x_1\in\mathbb{H}\setminus\mathbb{R}$.

\begin{rmk}\label{rmk:Qprojection}
    If $x_1\in\mathbb{R}$, then $\pi_Q(x_1+\epsilon x_2)=x_1$. If we fix, instead, $x_1\in\mathbb{H}\setminus\mathbb{R}$, then the restriction of $\pi_Q$ to the affine $4$-space $H_{x_1}$ is the orthogonal projection onto the affine $2$-plane $T_{x_1}$. In particular, $\pi_\mathbb{H}\circ\pi_Q=\pi_\mathbb{H}$.
\end{rmk}

If $x_1\in\mathbb{C}_{J_1}\setminus\mathbb{R}$ with $J_1\in\mathbb{S}_\mathbb{H}$ and, if we decompose $x_2\in\mathbb{H}$ as $x_2=x_2^\parallel+x_2^\perp$ with $x_2^\parallel\in\mathbb{C}_{J_1},x_2^\perp\in\mathbb{C}_{J_1}^\perp$, then
\begin{equation*}
    \pi_Q(x_1+\epsilon x_2)=x_1+\epsilon x_2^\perp=x_1+\epsilon(x_2-\langle x_2, 1\rangle-\langle x_2, J_1\rangle J_1).
\end{equation*}
In real components: if $x=r_0+ir_1+jr_2+kr_3+\epsilon(r_4+ir_5+jr_6+kr_7)\in \mathbb{DH}\setminus(\mathbb{R}+\epsilon\mathbb{H})$, then
\begin{equation*}
    \pi_Q(x)=r_0+ir_1+jr_2+kr_3+\epsilon\left(ir_5+jr_6+kr_7-\frac{r_1r_5+r_2r_6+r_3r_7}{r_1^2+r_2^2+r_3^2}(ir_1+jr_2+kr_3)\right);
\end{equation*}
while if $x=r_0+\epsilon(r_4+ir_5+jr_6+kr_7)\in\mathbb{R}+\epsilon\mathbb{H}$, then $\pi_Q(x)=r_0$. We are now ready for the next proposition.

\begin{prop}\label{prop:Qprojection}
    $\pi_Q:\mathbb{DH}\to Q_{\mathbb{DH}}$ is a quasi-open surjective map, which coincides in $Q_{\mathbb{DH}}$ with the identity map. On the one hand, $\pi_Q$ is not continuous at any point of the vector $5$-space $\mathbb{R}+\epsilon\mathbb{H}$. On the other hand, the restriction $\pi_Q:\mathbb{DH}\setminus(\mathbb{R}+\epsilon\mathbb{H})\to Q_{\mathbb{DH}}\setminus\mathbb{R}$ is a trivial vector fiber bundle with fiber $\epsilon\mathbb{R}_\mathbb{C}$.
\end{prop}

\begin{proof}
By direct inspection, the restriction of the map $\pi_Q$ to $Q_{\mathbb{DH}}$ is the identity map. It follows immediately that $\pi_Q$ is surjective. Moreover, $\pi_Q$ is not continuous at any point $r_0+\epsilon(r_4+ir_5+jr_6+kr_7)$ of the vector $5$-space $\mathbb{R}+\epsilon\mathbb{H}$ because
\begin{equation*}
    \lim_{(r_1,r_2,r_3)\to(0,0,0)}\frac{r_1r_5+r_2r_6+r_3r_7}{r_1^2+r_2^2+r_3^2}(ir_1+jr_2+kr_3)
\end{equation*}
does not exist.

Now, we have a commutative diagram
\[\begin{tikzcd}[ampersand replacement=\&,cramped]
	{\mathbb{S}_\mathbb{DH}\times\mathbb{DR}_\mathbb{C}^+} \&\& {\mathbb{S}_\mathbb{DH}\times\mathbb{R}_\mathbb{C}^+} \\
	\& \circlearrowleft \\
	{\mathbb{DH}\setminus(\mathbb{R}+\epsilon\mathbb{H})} \&\& {Q_\mathbb{DH}\setminus\mathbb{R}}
	\arrow["{(id,\pi_{\mathbb{H}_\mathbb{C}})}", from=1-1, to=1-3]
	\arrow["\Phi"', from=1-1, to=3-1]
	\arrow["\Phi", from=1-3, to=3-3]
	\arrow["{\pi_Q}"', from=3-1, to=3-3]
\end{tikzcd}\]
where the arrows marked with $\Phi$ are homeomorphisms. Clearly,
\[(id,\pi_{\mathbb{H}_\mathbb{C}}):\mathbb{S}_\mathbb{DH}\times\mathbb{DR}_\mathbb{C}^+=\mathbb{S}_\mathbb{DH}\times(\mathbb{R}_\mathbb{C}^++\epsilon\mathbb{R}_\mathbb{C})\to\mathbb{S}_\mathbb{DH}\times\mathbb{R}_\mathbb{C}^+\]
is a trivial fiber bundle with fiber $\epsilon\mathbb{R}_\mathbb{C}$, whose trivialization is 
\[((id,\pi_{\mathbb{H}_\mathbb{C}}),\pi_{\epsilon\mathbb{H}_\mathbb{C}}\circ\pi_2):\mathbb{S}_\mathbb{DH}\times\mathbb{DR}_\mathbb{C}^+\to(\mathbb{S}_\mathbb{DH}\times\mathbb{R}_\mathbb{C}^+)\times\epsilon\mathbb{R}_\mathbb{C}\,,\]
where $\pi_2:\mathbb{S}_\mathbb{DH}\times\mathbb{DR}_\mathbb{C}^+\to\mathbb{DR}_\mathbb{C}^+$ is the projection on the second component. Thus $\pi_Q:\mathbb{DH}\setminus(\mathbb{R}+\epsilon\mathbb{H})\to Q_{\mathbb{DH}}\setminus\mathbb{R}$ is a trivial fiber bundle with fiber $\epsilon\mathbb{R}_\mathbb{C}$, whose trivialization is
\[\mathbb{DH}\setminus(\mathbb{R}+\epsilon\mathbb{H})\to(Q_{\mathbb{DH}}\setminus\mathbb{R})\times\epsilon\mathbb{R}_\mathbb{C}\,,\quad x\mapsto(\pi_Q(x),\pi_{\epsilon\mathbb{H}_\mathbb{C}}(\pi_2(\Phi^{-1}(x))))\,.\]

Finally, the unrestricted map $\pi_Q:\mathbb{DH}\to Q_{\mathbb{DH}}$ is quasi-open because, for every open subset $U$ of $\mathbb{DH}$, its image $\Phi(U)$ contains $\Phi(U\setminus(\mathbb{R}+\epsilon\mathbb{H}))$, which is an open subset of $Q_\mathbb{DH}\setminus\mathbb{R}$, whence of $Q_\mathbb{DH}$.
\end{proof}

We conclude this section with a construction that will be particularly useful later in the paper and that generalizes, in a sense, the construction of $\mathbb{S}_y^\mathbb{DH}$ from the case $y\in Q_\mathbb{DH}$ to all $y\in\mathbb{DH}$.

\begin{defn}
    We call two elements $y,y'$ of $\mathbb{DH}$ \emph{equivalent} and write $y\sim y'$ if $t(y)=t(y')$ and $n(y)=n(y')$. The symbol $[y]$ will denote the equivalence class of $y$ in $\mathbb{DH}$.
\end{defn}

\begin{rmk}\label{rmk:classes}
    If $x\in\mathbb{R}+\epsilon\mathbb{H}$, then $[x]=[x_1+\epsilon\operatorname{Re}(x_2)]$. If we decompose $x\in\mathbb{DH}\setminus(\mathbb{R}+\epsilon\mathbb{H})$ as $x=x_1+\epsilon x_2=x_1+\epsilon x_2^\perp+\epsilon x_2^\parallel$, then $[x]=[x_1+\epsilon x_2^\parallel]$. In particular, for $x\in Q_\mathbb{DH}$ we get $[x]=[x_1]$. The second statement is true because $x\sim x_1+\epsilon x_2^\parallel$, as a consequence of the following chains of equalities:
    \begin{align*}
        t(x)&=t(x_1)+\epsilon t(x_2)=t(x_1)+\epsilon t(x_2^\parallel)=t(x_1+\epsilon x_2^\parallel)\,,\\
        n(x)&=n(x_1)+\epsilon t(x_1x_2^c)=n(x_1)+2\epsilon\langle x_1,x_2\rangle=n(x_1)+2\epsilon\langle x_1,x_2^\parallel\rangle\\
        &=n(x_1)+\epsilon\,t\left(x_1(x_2^\parallel)^c\right)=n(x_1+\epsilon x_2^\parallel)\,.
    \end{align*}
    Similar computations prove the first statement.
\end{rmk}

Remark~\ref{rmk:classes} motivates the following definition, which will soon be useful.

\begin{defn}
    For every $x\in\mathbb{DH}\setminus(\mathbb{R}+\epsilon\mathbb{H})$, 
    we decompose $x\in\mathbb{DH}$ as $x=x_1+\epsilon x_2=x_1+\epsilon x_2^\perp+\epsilon x_2^\parallel$ and set
    \[T_x\coloneqq T_{x_1}+\epsilon x_2^\parallel\,.\]
\end{defn}

For all $a,b\in\mathbb{DR},J\in\mathbb{S}_\mathbb{DH}$, with $b\not\in\epsilon\mathbb{R}$ we point out that
\[T_{a+Jb}=a_1+J_1b_1+\epsilon\mathbb{C}_{J_1}^\perp+\epsilon(a_2+J_1b_2)=a+J_1b+\epsilon\mathbb{C}_{J_1}^\perp=a+J_1b+\epsilon b_1\mathbb{C}_{J_1}^\perp=a+bT_{J_1}.\]
In other words, $T_{a+Jb}$ is the set of points of the form $a+Ib$ with $I_1=J_1$. In particular, $T_{a+Jb}=T_{a+J_1b}$. We are now in a position to compute equivalence classes.

\begin{prop}\label{prop:classes}
    Let $a,b\in\mathbb{DR},I\in\mathbb{S}_\mathbb{DH}$ and set $y\coloneqq a+Ib$. We have $t(a+Ib)=2a$ and $n(a+Ib)=a^2+b^2$. If $b\in\epsilon\mathbb{R}$, then $[a+Ib]=a+\epsilon\operatorname{Im}(\mathbb{H})\supset a+b\mathbb{S}_\mathbb{H}=a+b\mathbb{S}_\mathbb{DH}$. If $b\not\in\epsilon\mathbb{R}$, then
    \[[a+Ib]=a+b\mathbb{S}_\mathbb{DH}=\bigcup_{J_1\in\mathbb{S}_\mathbb{H}}T_{a+J_1b}\,.\]
    In particular, $[y]=[y^c]$. Moreover, $\pi_Q([y])=[y_1]=\mathbb{S}_{y_1}^\mathbb{DH}=a_1+b_1\mathbb{S}_\mathbb{DH}$ and $\pi_\mathbb{H}([y])=\mathbb{S}_{y_1}^\mathbb{H}=a_1+b_1\mathbb{S}_\mathbb{H}$.
\end{prop}

\begin{proof}
    Since $\phi_I:\mathbb{DR}_\mathbb{C}\to\mathbb{DC}_I$ is a $\ast$-isomorphism, a direct computation shows that
    \begin{align*}
        t(a+Ib)&=t(\phi_I(a+e_1b))=\phi_I(a+e_1b)+\phi_I(a+e_1b)^c\\
        &=\phi_I\left(a+e_1b+\overline{a+e_1b}\right)=\phi_I(2a)=2a\,,\\
        n(a+Ib)&=n(\phi_I(a+e_1b))=\phi_I(a+e_1b)\,\phi_I(a+e_1b)^c\\
        &=\phi_I\left((a+e_1b)(\overline{a+e_1b})\right)=\phi_I(a^2+b^2)=a^2+b^2\,.
    \end{align*}
    It follows at once that $[a+Ib]\supseteq a+b\mathbb{S}_\mathbb{DH}$. Conversely, $x\in[a+Ib]$ implies $t(x)=2a$ and $n(x)=a^2+b^2$. If $b\notin\epsilon\mathbb{R}$, then such an $x$ takes the form $a+Jb$ for some $J\in\mathbb{S}_\mathbb{DH}$, whence $x\in a+b\mathbb{S}_\mathbb{DH}$. If, instead, $b\in\epsilon\mathbb{R}$ (whence $b^2=0$), then such an $x$ takes the form $a+\epsilon c$ for any $c\in\operatorname{Im}(\mathbb{H})$.
    Furthermore, when $b\notin\epsilon\mathbb{R}$,
    \begin{align*}
        a+b\mathbb{S}_\mathbb{DH}&=a+b\bigcup_{J_1\in\mathbb{S}_\mathbb{H}}T_{J_1}=\bigcup_{J_1\in\mathbb{S}_\mathbb{H}}(a+bT_{J_1})=\bigcup_{J_1\in\mathbb{S}_\mathbb{H}}T_{a+J_1b}\,.
    \end{align*}
    We remark that $\pi_Q([y])=\pi_Q(a+b\mathbb{S}_\mathbb{DH})=a_1+b_1\mathbb{S}_\mathbb{DH}=\mathbb{S}_{y_1}^\mathbb{DH}=[y_1]$. The last equality in the statement can now be derived using the fact that $\pi_\mathbb{H}\circ\pi_Q=\pi_\mathbb{H}$.
\end{proof}

We make one more remark about equivalence classes.

\begin{rmk}
    Pick $a=a_1+e_1a_2,b=b_1+e_1b_2\in\mathbb{DR}$ and $I\in\mathbb{S}_\mathbb{DH}$. If $b_1>0$, then $\Phi^{-1}([a+Ib])=\Phi^{-1}(a+b\mathbb{S}_\mathbb{DH})=\mathbb{S}_\mathbb{DH}\times\{a\pm e_1b\}$. If $b_1=0$, then $\Phi^{-1}(a+b\mathbb{S}_\mathbb{DH})=\mathbb{S}_\mathbb{DH}\times\{a\pm e_1b\}$ but $\Phi^{-1}([a+Ib])=\mathbb{S}_\mathbb{DH}\times(a+e_1\epsilon\mathbb{R})$.
\end{rmk}

To conclude this section, we define and study a few symmetry notions in $\mathbb{DH}$.

\begin{defn}
    A \emph{symmetric} subset $T$ of $\mathbb{DH}$ is a $T\subseteq\mathbb{DH}$ such that, for every $x\in T$, the whole equivalence class $[x]$ is contained in $T$; i.e., such that $T=\bigcup_{x\in T}[x]$. The \emph{symmetric completion} of $T\subset\mathbb{DH}$ is the smallest symmetric subset containing $T$.
\end{defn}

In the special case when $T\subseteq Q_\mathbb{DH}\setminus\mathbb{R}$, the symmetric completion of $T$ is the smallest circular subset of $Q_\mathbb{DH}$ containing $T$. In particular, if $y\in Q_\mathbb{DH}\setminus\mathbb{R}$, then the symmetric completion of $\{y\}$ is $[y]=\mathbb{S}_y^\mathbb{DH}$. In contrast, if $y=y_1\in\mathbb{R}$, then the symmetric completion of $\{y_1\}$ is $[y_1]=y_1+\epsilon\operatorname{Im}(\mathbb{H})$.

\begin{defn}
    A symmetric subset $T$ of $\mathbb{DH}$ is termed \emph{maximally symmetric} if, for every $x=x_1+\epsilon x_2\in T$, the affine $4$-space $H_{x_1}=x_1+\epsilon\mathbb{H}$ is contained in $T$; in other words, $T=\bigcup_{x\in T}H_{x_1}$.
    The \emph{maximally symmetric completion} $\check{T}$ of $T\subseteq\mathbb{DH}$ is the smallest maximally symmetric subset containing $T$. The symbol $\check{\mathbb{S}}$ will denote the maximally symmetric completion of $\mathbb{S}_\mathbb{DH}$ or, equivalently, of $\{i\}$. For any $y\in\mathbb{DH}$, the symbol $\check{\mathbb{S}}_y$ will denote the maximally symmetric completion of $\{y\}$. 
\end{defn}

\begin{prop}
    For all $T\subseteq\mathbb{DH}$, then
    \[\check{T}\supseteq T+\epsilon\mathbb{H}=\pi_Q(T)+\epsilon\mathbb{H}=\pi_\mathbb{H}(T)+\epsilon\mathbb{H}.\]
    Moreover, each of the following conditions implies the subsequent one:
    \begin{enumerate}
        \item $T$ is symmetric;
        \item $T\setminus(\mathbb{R}+\epsilon\mathbb{H})$ is symmetric;
        \item $\pi_Q(T)$ is a circular subset of $Q_\mathbb{DH}$;
        \item $\pi_\mathbb{H}(T)$ is a circular subset of $\mathbb{H}$;
        \item $T+\epsilon\mathbb{H}$ is symmetric;
        \item $T+\epsilon\mathbb{H}$ is maximally symmetric;
        \item $\check{T}=T+\epsilon\mathbb{H}$.
    \end{enumerate}
    Conditions {\it 4}, {\it 5}, {\it 6}, {\it 7} are, in fact, equivalent.
\end{prop}

\begin{proof}
    The first statement follows directly from the previous definition.
    
    We now turn to the second statement. If $T$ is symmetric, then $T\setminus(\mathbb{R}+\epsilon\mathbb{H})$ is automatically symmetric because $x\not\in\mathbb{R}+\epsilon\mathbb{H}$ implies $[x]\cap(\mathbb{R}+\epsilon\mathbb{H})=\emptyset$. Assuming $T\setminus(\mathbb{R}+\epsilon\mathbb{H})$ to be symmetric is the same as assuming $T\setminus(\mathbb{R}+\epsilon\mathbb{H})=\bigcup_{x\in T\setminus(\mathbb{R}+\epsilon\mathbb{H})}[x]$; taking into account that $\pi_Q([x])=[x_1]=\mathbb{S}_{x_1}^\mathbb{DH}$, we derive that  $\pi_Q(T)\setminus\mathbb{R}=\pi_Q(T\setminus(\mathbb{R}+\epsilon\mathbb{H}))=\bigcup_{x\in T\setminus(\mathbb{R}+\epsilon\mathbb{H})}\mathbb{S}_{x_1}^\mathbb{DH}$; it follows that $\pi_Q(T)\setminus\mathbb{R}$ (whence $\pi_Q(T)$) is a circular subset of $Q_\mathbb{DH}$. Assuming $\pi_Q(T)$ to be circular is the same as assuming $\pi_Q(T)=\bigcup_{x\in \pi_Q(T)}\mathbb{S}_x^\mathbb{DH}=\bigcup_{x\in \pi_Q(T)}\mathbb{S}_{x_1}^\mathbb{DH}$; taking into account Remark~\ref{rmk:Qprojection}, we get $\pi_\mathbb{H}(T)=\pi_\mathbb{H}(\pi_Q(T))=\bigcup_{x\in \pi_Q(T)}\pi_\mathbb{H}(\mathbb{S}_{x_1}^\mathbb{DH})=\bigcup_{x\in \pi_Q(T)}\mathbb{S}_{x_1}^\mathbb{H}$, whence $\pi_\mathbb{H}(T)$ is a circular subset of $\mathbb{H}$. If $\pi_\mathbb{H}(T)$ is a circular subset of $\mathbb{H}$, then $T+\epsilon\mathbb{H}=\pi_\mathbb{H}(T)+\epsilon\mathbb{H}$ is symmetric because, for every $x\in \pi_\mathbb{H}(T)+\epsilon\mathbb{H}$, we have $x_1\in\pi_\mathbb{H}(T)$, whence $\mathbb{S}_{x_1}^\mathbb{H}\subseteq\pi_\mathbb{H}(T)$ and $\pi_\mathbb{H}(T)+\epsilon\mathbb{H}\supseteq\mathbb{S}_{x_1}^\mathbb{H}+\epsilon\mathbb{H}\supseteq[x]$. If $T+\epsilon\mathbb{H}$ is symmetric, then it is automatically maximally symmetric. If $T+\epsilon\mathbb{H}$ is maximally symmetric, taking into account that $T\subseteq T+\epsilon\mathbb{H}\subseteq\check{T}$, we conclude that $T+\epsilon\mathbb{H}$ is indeed equal to $\check{T}$, as desired.

    Let us now prove the third statement. If $T+\epsilon\mathbb{H}$ is the maximally symmetric completion of a set, then in particular it is maximally symmetric. If $T+\epsilon\mathbb{H}$ is maximally symmetric, then a fortiori it is symmetric. Finally, if $T+\epsilon\mathbb{H}$ is symmetric, then $\pi_\mathbb{H}(T)$ is a circular subset of $\mathbb{H}$: indeed, for every $x_1\in\pi_\mathbb{H}(T)=\pi_\mathbb{H}(T+\epsilon\mathbb{H})$ there exists $y\in T+\epsilon\mathbb{H}$ with $y_1=x_1$ and the inclusion $[y]\subseteq T+\epsilon\mathbb{H}$ implies that $\mathbb{S}_{x_1}^\mathbb{H}=\pi_\mathbb{H}([y])\subseteq\pi_\mathbb{H}(T+\epsilon\mathbb{H})=\pi_\mathbb{H}(T)$.
\end{proof}

As a consequence of the last proposition, we have $\widecheck{Q}_\mathbb{DH}=\widecheck{\mathbb{H}}=\mathbb{DH}$ and we can make the next remark.

\begin{rmk}
    If $\Omega_D$ denotes $\Omega_{D,\mathbb{DH}}$, then
    \begin{align}\label{eq:scheck}
        &\check{\mathbb{S}}_y=[y]+\epsilon\mathbb{H}=[y_1]+\epsilon\mathbb{H}=\mathbb{S}_{y_1}^\mathbb{DH}+\epsilon\mathbb{H}=\mathbb{S}_{y_1}^\mathbb{H}+\epsilon\mathbb{H}=\Phi(\mathbb{S}_\mathbb{DH}\times(\{a_1+e_1b_1\}+\epsilon\mathbb{R}_\mathbb{C}))\,,\\
        &\widecheck{\Omega}_D=\Omega_D+\epsilon\mathbb{H}=\Omega_{D,\mathbb{H}}+\epsilon\mathbb{H}=\Phi(\mathbb{S}_\mathbb{DH}\times(D+\epsilon\mathbb{R}_\mathbb{C}))\notag
    \end{align}
     for all $y=a+Ib\in\mathbb{DH}$ and all $D\subseteq\mathbb{R}_\mathbb{C}$. The last equality in either chain follows from the fact that $\Phi(\mathbb{S}_\mathbb{DH}\times(D+\epsilon\mathbb{R}_\mathbb{C}))=\{a+Ib:a,b\in\mathbb{DR},I\in\mathbb{S}_\mathbb{DH},a_1+e_1b_1\in D\}=\{a_1+I_1b_1+\epsilon(a_2+I_1b_2+I_2b_1):a_1,b_1,a_2,b_2\in\mathbb{R},I_1\in\mathbb{S}_\mathbb{H},I_2\in\operatorname{Im}(\mathbb{H}),I_1\perp I_2,a_1+I_1b_1\in\Omega_{D,\mathbb{H}}\}=\Omega_{D,\mathbb{H}}+\epsilon\mathbb{H}$.
\end{rmk}

In particular,
    \begin{align*}
        \pi_Q(\widecheck{\Omega}_D)&=\widecheck{\Omega}_D\cap Q_\mathbb{DH}=\Omega_D=\Phi(\mathbb{S}_\mathbb{DH}\times D)\\
        \pi_\mathbb{H}(\widecheck{\Omega}_D)&=\widecheck{\Omega}_D\cap\mathbb{H}=\Omega_{D,\mathbb{H}}=\Phi(\mathbb{S}_\mathbb{H}\times D).
    \end{align*}
We also point out that $\widecheck{\Omega}_{D\setminus\mathbb{R}}=\widecheck{\Omega}_D\setminus(\mathbb{R}+\epsilon\mathbb{H})$. We conclude this section with a few other decompositions of $\check{\mathbb{S}}_y$, which will be useful later.

\begin{lem}\label{lem:decompositionofscheck}
    Let $y\in\mathbb{DH}$. Then
    \[\check{\mathbb{S}}_y=\bigcup_{x_1=y_1}[x]=\bigcup_{x_1\in\mathbb{S}^\mathbb{H}_{y_1}}H_{x_1}\,.\]
    If, moreover, $y\not\in\mathbb{R}+\epsilon\mathbb{H}$, then $H_{y_1}=\bigcup_{x_1=y_1}T_x$ and $\check{\mathbb{S}}_y=\bigcup_{x_1\in\mathbb{S}^\mathbb{H}_{y_1}}T_x$.
\end{lem}    

\begin{proof}
    Using~\eqref{eq:scheck}, we compute
    \[\check{\mathbb{S}}_y=\mathbb{S}^\mathbb{H}_{y_1}+\epsilon\mathbb{H}=\bigcup_{x_1\in\mathbb{S}^\mathbb{H}_{y_1}}(x_1+\epsilon\mathbb{H})=\bigcup_{x_1\in\mathbb{S}^\mathbb{H}_{y_1}}H_{x_1}\,.\]
    If $y\in\mathbb{R}+\epsilon\mathbb{H}$, then $\mathbb{S}_{y_1}^\mathbb{H}=\{y_1\}$, whence
        \[\bigcup_{x_1\in\mathbb{S}^\mathbb{H}_{y_1}}H_{x_1}=H_{y_1}=y_1+\epsilon\mathbb{H}=y_1+\epsilon\mathbb{R}+\epsilon\operatorname{Im}(\mathbb{H})=\bigcup_{a\in\mathbb{DR},a_1=y_1}(a+\epsilon\operatorname{Im}(\mathbb{H}))=\bigcup_{x_1=y_1}[x]\,.\]
    Now let us assume, instead, $y\not\in\mathbb{R}+\epsilon\mathbb{H}$. For all $a_1,b_1\in\mathbb{R},J_1\in\mathbb{S}_\mathbb{H}$, we have
    \begin{align*}
        H_{a_1+J_1b_1}&=a_1+J_1b_1+\epsilon\mathbb{H}=T_{a_1+J_1b_1}+\epsilon\mathbb{C}_{J_1}=\bigcup_{A,B\in\mathbb{DR},A_1=a_1,B_1=b_1}T_{A+J_1B}\\
        &=\bigcup_{A,B\in\mathbb{DR},I\in\mathbb{S}_\mathbb{DH},A_1=a_1,B_1=b_1,I_1=J_1}T_{A+IB}=\bigcup_{x_1=a_1+J_1b_1}T_x\,,
    \end{align*}
    as desired. It follows at once that
    \[\check{\mathbb{S}}_y=\bigcup_{x_1\in\mathbb{S}^\mathbb{H}_{y_1}}H_{x_1}=\bigcup_{x_1\in\mathbb{S}^\mathbb{H}_{y_1}}T_x\,.\]
    To conclude the proof, it suffices to remark that
        \begin{align*}
        \bigcup_{x_1\in\mathbb{S}^\mathbb{H}_{y_1}}T_x&=\bigcup_{J\in\mathbb{S}_\mathbb{DH},a,b\in\mathbb{DR},a_1+J_1b_1\in\mathbb{S}^\mathbb{H}_{y_1}}T_{a+Jb}=\bigcup_{J_1\in\mathbb{S}_\mathbb{H},a,b\in\mathbb{DR},a_1+J_1b_1\in\mathbb{S}^\mathbb{H}_{y_1}}T_{a+J_1b}\\
        &=\bigcup_{a,b\in\mathbb{DR},y_1\in a_1+b_1\mathbb{S}_\mathbb{H}}(a+b\mathbb{S}_\mathbb{DH})=\bigcup_{x_1=y_1}[x]\,.
        \end{align*}
    For the last two equalities, we used Proposition~\ref{prop:classes}.
\end{proof}

%%%%%%%%%%%%%%%%%%%%%%%%%%%%%

\section{Extension of stem functions}\label{sec:dualstem}

Our aim is to construct a new function class over $\mathbb{DH}$, containing the class of functions $\mathbb{DH}\to\mathbb{DH}$ associated to dual quaternionic polynomials. As a first step, we appropriately extend stem functions from $\mathbb{R}_\mathbb{C}$ to $\mathbb{DR}_\mathbb{C}$. Throughout this section, we fix an open subset $D$ of $\mathbb{R}_\mathbb{C}$ that is invariant with respect to conjugation.

\begin{defn}\label{def:extendedstem}
    Set $Stem^{h,k}(D,\mathbb{DH}_\mathbb{C})\coloneqq Stem^h(D,\mathbb{H}_\mathbb{C})+\epsilon\, Stem^k(D,\mathbb{H}_\mathbb{C})$ for all $h,k\in\mathbb{N}\cup\{\infty,\omega\}$. Pick $G\in Stem^1(D,\mathbb{H}_\mathbb{C}),H\in Stem^0(D,\mathbb{H}_\mathbb{C})$ and consider the element $F\coloneqq G+\epsilon H$ of $Stem^{1,0}(D,\mathbb{DH}_\mathbb{C})$. The \emph{natural extension} of $F$ to $\check{D}\coloneqq D+\epsilon\mathbb{R}_\mathbb{C}$ is the continuous function $\check{F}:\check{D}\to\mathbb{DH}_\mathbb{C}$ defined by the formula
    \begin{align*}
        \check{F}(z_1+\epsilon z_2)&\coloneqq G(z_1)+\epsilon\left(H(z_1)+z_2\,\frac{\partial G}{\partial z}(z_1)+\bar z_2\,\frac{\partial G}{\partial \overline{z}}(z_1)\right)\\
        &=F(z_1)+\epsilon\left(z_2\,\frac{\partial G}{\partial z}(z_1)+\bar z_2\,\frac{\partial G}{\partial \overline{z}}(z_1)\right)\,,
    \end{align*}
    or, equivalently,
    \[\check{F}(a+e_1b)\coloneqq G(a_1+e_1b_1)+\epsilon\left(H(a_1+e_1b_1)+a_2\,\frac{\partial G}{\partial \alpha}(a_1+e_1b_1)+b_2\,\frac{\partial G}{\partial \beta}(a_1+e_1b_1)\right)\,.\]
    We call all such functions $\check{F}$ \emph{dual stem functions} on $\check{D}$ and denote their set by the symbol $DStem^0(\check{D},\mathbb{DH}_\mathbb{C})$. For $k\in\mathbb{N}^*\cup\{\infty,\omega\}$, we define the sets $DStem^k(\check{D},\mathbb{DH}_\mathbb{C})\coloneqq\{\check{F}:F\in Stem^{k+1,k}(D,\mathbb{DH}_\mathbb{C})\}\subset\mathcal{C}^k(\check{D},\mathbb{DH}_\mathbb{C})$ and $DStem^{pol}(\check{D},\mathbb{DH}_\mathbb{C})\coloneqq\{\check{F}:F\in Stem^{pol}(D,\mathbb{DH}_\mathbb{C})\}$.
\end{defn}

When $F\in Stem^1(D,\mathbb{DH}_\mathbb{C})=Stem^{1,1}(D,\mathbb{DH}_\mathbb{C})$, then the equality
    \[\check{F}(z_1+\epsilon z_2)=F(z_1)+\epsilon\left(z_2\,\frac{\partial F}{\partial z}(z_1)+\bar z_2\,\frac{\partial F}{\partial \overline{z}}(z_1)\right)\]
also holds true for all $z_1+\epsilon z_2\in \check{D}$.

For $z_1\in D$ fixed, the map $\epsilon\mathbb{R}_\mathbb{C}\to\mathbb{DH}_\mathbb{C},\ \epsilon z_2\mapsto\check{F}(z_1+\epsilon z_2)$ is a real affine map whose value at $0$ is $F(z_1)$ and whose differential at $0$ maps $\epsilon\mathbb{R}_\mathbb{C}$ into $\epsilon\mathbb{H}_\mathbb{C}$ in accordance with the nature of the real differential of $G$ at $z_1$. In particular, $\check{F}_{|_D}=F$. We are now ready for the next theorem.

\begin{thm}
    The class $DStem^0(\check{D},\mathbb{DH}_\mathbb{C})$ is both a right $\mathbb{DH}$-module and a real $\ast$-algebra, when endowed with the pointwise operations $(\check{F}+\check{G})(z)\coloneqq\check{F}(z)+\check{G}(z), (\check{F}h)(z)\coloneqq\check{F}(z)h,(\check{F}\check{G})(z)\coloneqq\check{F}(z)\check{G}(z),\check{F}^c(z)\coloneqq\check{F}(z)^c$. Natural extension $Stem^{1,0}(D,\mathbb{DH}_\mathbb{C})\to DStem^0(\check{D},\mathbb{DH}_\mathbb{C}),\ F\mapsto\check{F}$ is both a right $\mathbb{DH}$-module isomorphism and a $\ast$-algebra isomorphism, whose inverse map is restriction to $D$.
\end{thm}

\begin{proof}
    Definition~\ref{def:extendedstem} implies at once that $DStem^0(\check{D},\mathbb{DH}_\mathbb{C})$ is a right $\mathbb{DH}$-module and that natural extension is a right $\mathbb{DH}$-module morphism. Since we already remarked that $\check{F}_{|_D}=F$ for all $F\in Stem^{1,0}(D,\mathbb{DH}_\mathbb{C})$, natural extension is a right $\mathbb{DH}$-module isomorphism whose inverse is restriction to $D$. We are left with proving that natural extension is a $\ast$-algebra morphism. Pick $R,G\in Stem^1(D,\mathbb{H}_\mathbb{C}),S,H\in Stem^0(D,\mathbb{H}_\mathbb{C})$. By applying Definition~\ref{def:extendedstem} to both $P\coloneqq R+\epsilon S$ and $F\coloneqq G+\epsilon H$, we get
    \begin{align*}
        \check{P}(z)\check{F}(z)&=R(z_1)G(z_1)+\epsilon\left(S(z_1)+z_2\,\frac{\partial R}{\partial z}(z_1)+\bar z_2\,\frac{\partial R}{\partial \overline{z}}(z_1)\right)G(z_1)\\
        &\quad+R(z_1)\,\epsilon\left(H(z_1)+z_2\,\frac{\partial G}{\partial z}(z_1)+\bar z_2\,\frac{\partial G}{\partial \overline{z}}(z_1)\right)\\
        &=P(z_1)F(z_1)+\epsilon z_2\left(\frac{\partial R}{\partial z}(z_1)\,G(z_1)+R(z_1)\,\frac{\partial G}{\partial z}(z_1))\right)\\
        &\quad+\epsilon \bar z_2\left(\frac{\partial R}{\partial \overline{z}}(z_1)\,G(z_1)+R(z_1)\,\frac{\partial G}{\partial \overline{z}}(z_1)\right)\\
        &=(PF)(z_1)+\epsilon\left(z_2\,\frac{\partial(RG)}{\partial z}(z_1)+\bar z_2\,\frac{\partial(RG)}{\partial \overline{z}}(z_1)\right)=\widecheck{PF}(z)
    \end{align*}
    and
    \begin{align*}
        \check{F}(z)^c&=G(z_1)^c+\epsilon\left(H(z_1)+z_2\,\frac{\partial G}{\partial z}(z_1)+\bar z_2\,\frac{\partial G}{\partial \overline{z}}(z_1)\right)^c\\
        &=F^c(z_1)+\epsilon\left(z_2\,\frac{\partial G^c}{\partial z}(z_1)+\bar z_2\,\frac{\partial G^c}{\partial \overline{z}}(z_1)\right)=\widecheck{F^c}(z)\,,
    \end{align*}
    as desired. For the second and fourth equalities, we took into account that $PF=RG+\epsilon(SG+RH)$.
\end{proof}

For $k\in\mathbb{N}\cup\{\infty,\omega\}$, by construction, the $\ast$-subalgebras $Stem^{k+1,k}(D,\mathbb{DH}_\mathbb{C}),Stem^{pol}(D,\mathbb{DH}_\mathbb{C})$ are isomorphically mapped by natural extension into $DStem^k(\check{D},\mathbb{DH}_\mathbb{C})$ and $DStem^{pol}(\check{D},\mathbb{DH}_\mathbb{C})$, respectively. 

\begin{ex}\label{ex:REIM}
    Consider the following polynomial stem functions $\mathbb{R}_\mathbb{C}\to\mathbb{R}_\mathbb{C}\subset\mathbb{DH}_\mathbb{C}$:
    \[ID(z_1)=z_1,\quad\overline{ID}(z_1)=\bar z_1,\quad RE(z_1)\coloneqq\frac{z_1+\bar z_1}{2},\quad IM(z_1)\coloneqq\frac{z_1-\bar z_1}{2}\,.\]
    For the last two, we also have the expressions $RE(a_1+e_1b_1)=a_1,IM(a_1+e_1b_1)=e_1b_1$, valid for all $a_1,b_1\in\mathbb{R}$. These stem functions induce relevant slice functions: we have $id_{Q_\mathbb{DH}}=\mathcal{I}(ID),\operatorname{Re}=\mathcal{I}(RE),\operatorname{Im}=\mathcal{I}(IM)\in\mathcal{S}^{pol}(Q_\mathbb{DH})$, while $\mathcal{I}(\overline{ID})$ is the map $Q_\mathbb{DH}\to Q_\mathbb{DH},\ x\mapsto x^c$. Since $\frac{\partial ID}{\partial z}\equiv1\equiv\frac{\partial\overline{ID}}{\partial \overline{z}},\frac{\partial ID}{\partial \overline{z}}\equiv0\equiv\frac{\partial\overline{ID}}{\partial z},$ and $\frac{\partial RE}{\partial z}=\frac{\partial RE}{\partial \overline{z}}=\frac{\partial IM}{\partial z}=-\frac{\partial IM}{\partial \overline{z}}\equiv\frac12$, by definition
    \begin{align*}
        \widecheck{ID}(z)&=z_1+\epsilon(z_2\,1+\bar z_2\,0)=z_1+\epsilon z_2=z\,,\\
        \widecheck{\overline{ID}}(z)&=\bar z_1+\epsilon(z_2\,0+\bar z_2\,1)=\bar z_1+\epsilon\bar z_2=\bar z\,,\\
        \widecheck{RE}(z)&=\widecheck{RE}(z_1+\epsilon z_2)=RE(z_1)+\epsilon\left(\frac{z_2}{2}+\frac{\bar z_2}{2}\right)=\frac{z_1+\bar z_1}{2}+\epsilon\frac{z_2+\bar z_2}{2}=\frac{z+\bar z}{2}\,,\\
        \widecheck{IM}(z)&=\widecheck{IM}(z_1+\epsilon z_2)=IM(z_1)+\epsilon\left(\frac{z_2}{2}-\frac{\bar z_2}{2}\right)=\frac{z_1-\bar z_1}{2}+\epsilon\frac{z_2-\bar z_2}{2}=\frac{z-\bar z}{2}
    \end{align*}
    for all $z\in\check{\mathbb{R}}_\mathbb{C}=\mathbb{DR}_\mathbb{C}$. For all $a,b\in\mathbb{DR}$, we get $\widecheck{RE}(a+e_1b)=a$ and $\widecheck{IM}(a+e_1b)=e_1b$. We point out that $\widecheck{IM}^2(a+e_1b)=-b^2$ is a dual stem function on $\mathbb{DR}_\mathbb{C}$ too, extending the polynomial stem function $IM^2(a_1+e_1b_1)=-b_1^2$. Moreover, if we set $D\coloneqq\mathbb{R}_\mathbb{C}\setminus\mathbb{R}$, then $\widecheck{IM}^{-1}(a+e_1b)=(e_1b)^{-1}=-e_1b^{-1}$ is a dual stem function on $\check{D}$, extending the element $IM^{-1}$ of $Stem^\omega(D,\mathbb{R}_\mathbb{C})$ (where $IM^{-1}(a_1+e_1b_1)=(e_1b_1)^{-1}=-e_1b_1^{-1}$).
\end{ex}

We are now in a position to make the following remark.

\begin{rmk}
    Any element of $DStem^{pol}(\check{D},\mathbb{DH}_\mathbb{C})$ is the natural extension $\check{F}$ of some $F\in Stem^{pol}(D,\mathbb{DH}_\mathbb{C})$, which by Remark~\ref{rmk:stempol} takes the form $F(z_1)=\sum_{\ell,m=0}^nz_1^\ell\bar z_1^ma_{\ell m}$ for some $n\in\mathbb{N}$ and some finite sequence $\{a_{\ell m}\}_{\ell,m=0}^n\subset\mathbb{DH}$. Since natural extension is a right $\mathbb{DH}$-module isomorphism and a $\ast$-algebra isomorphism, the computations made in the last example prove that $\check{F}(z)=\sum_{\ell,m=0}^nz^\ell\bar z^ma_{\ell m}$ for all $z\in\check{\mathbb{R}}_\mathbb{C}=\mathbb{DR}_\mathbb{C}$.
\end{rmk}

Along the same lines, for the natural extension $\check{F}$ of some $F\in Stem^{pol}(\mathbb{R}_\mathbb{C},\mathbb{DH}_\mathbb{C})$, we remark: for any dual numbers $a,b$, we may compute $\check{F}(a+e_1b)$ simply by replacing each $a_1$ or $b_1$ in the expression of $F(a_1+e_1b_1)$ by $a$ or $b$, respectively.

If $F=G+\epsilon H=G_0+\epsilon H_0+e_1(G_1+\epsilon H_1)$ with $G_0,H_0,G_1,H_1:D\to\mathbb{H}$ and if we define functions $\check{F}_0,\check{F}_1:\check{D}\to\mathbb{DH}$ by the formula
\begin{equation}\label{eq:componentsofextendedstem}
        \check{F}_\ell(a+e_1b)\coloneqq G_\ell(a_1+e_1b_1)+\epsilon\left(H_\ell(a_1+e_1b_1)+a_2\,\frac{\partial G_\ell}{\partial \alpha}(a_1+e_1b_1)+b_2\,\frac{\partial G_\ell}{\partial \beta}(a_1+e_1b_1)\right)\,,
\end{equation}
then: $\check{F}=\check{F}_0+e_1\check{F}_1$; $\check{F}_0$ is the natural extension of $G_0+\epsilon H_0$; and $e_1\check{F}_1$ is the natural extension of $e_1(G_1+\epsilon H_1)$. On the other hand, we point out that the natural extension of $G_1+\epsilon H_1$ is undefined because $G_1+\epsilon H_1$ is not a stem function. Moreover, we remark that $\overline{\check{F}}=\check{F}_0-e_1\check{F}_1$ is the natural extension of $G_0+\epsilon H_0-e_1(G_1+\epsilon H_1)=\overline{F}$, i.e., that $\overline{\check{F}}=\widecheck{\overline{F}}$. Finally, since $\check{F}_1\equiv0$ is equivalent to $F_1\equiv0$, the function $\check{F}$ takes values in $\mathbb{DH}$ if, and only if, $F$ does.

\begin{ex}\label{ex:REIMcomponents}
    Consider again the polynomial stem functions $RE,IM,IM^2:\mathbb{R}_\mathbb{C}\to\mathbb{R}_\mathbb{C}$ and the analytic stem function $IM^{-1}:\mathbb{R}_\mathbb{C}\setminus\mathbb{R}\to\mathbb{R}_\mathbb{C}$ of Example~\ref{ex:REIM}. By direct inspection, $RE_0(a_1+e_1b_1)=a_1,RE_1\equiv0,IM_0\equiv0,IM_1(a_1+e_1b_1)=b_1,(IM^2)_0(a_1+e_1b_1)=-b_1^2,(IM^2)_1\equiv0,(IM^{-1})_0\equiv0,(IM^{-1})_1(a_1+e_1b_1)=-b_1^{-1}$. Consistently, we have
    \footnotesize\begin{align*}
    &\widecheck{RE}_0(a+e_1b)=a_1+\epsilon(0+a_2+0)=a\,,& &\widecheck{RE}_1(a+e_1b)\equiv0\,,\\
        &\widecheck{IM}_0(a+e_1b)\equiv0\,,& &\widecheck{IM}_1(a+e_1b)=b_1+\epsilon(0+0+b_2)=b\,,\\
        &(\widecheck{IM}^2)_0(a+e_1b)=-b_1^2+\epsilon(0+0-2b_2b_1)=-b^2\,,& &(\widecheck{IM}^2)_1(a+e_1b)\equiv0\,,\\
        &(\widecheck{IM}^{-1})_0(a+e_1b)\equiv0\,,& &(\widecheck{IM}^{-1})_1(a+e_1b)=-b_1^{-1}+\epsilon(0+0+b_2b_1^{-2})=-b^{-1}\,.
    \end{align*}
\end{ex}

By formula~\eqref{eq:componentsofextendedstem}, the function $\check{F}$ takes values in $\mathbb{DR}_\mathbb{C}$ if, and only if, $F$ does. This motivates us to set
\[DStem^k(\check{D},\mathbb{DR}_\mathbb{C})\coloneqq\{\check{F}:F\in Stem^{k+1,k}(D,\mathbb{DR}_\mathbb{C})\}\subset\mathcal{C}^k(\check{D},\mathbb{DR}_\mathbb{C})\]
for all $k\in\mathbb{N}\cup\{\infty,\omega\}$. On the other hand, if $F$ takes values in $\mathbb{H}_\mathbb{C}$, then $\check{F}$ maps $D=\check{D}\cap\mathbb{R}_\mathbb{C}$ (but not all of $\check{D}$, in general) into $\mathbb{H}_\mathbb{C}$. We add a remark that will be useful later.

\begin{rmk}\label{rmk:specialproductofstem}
    Since $Stem^{1,0}(D,\mathbb{DR}_\mathbb{C})$ is contained in the center of $Stem^{1,0}(D,\mathbb{DH}_\mathbb{C})$, we remark that 
    $DStem^0(\check{D},\mathbb{DR}_\mathbb{C})$ is contained in the center of $DStem^0(\check{D},\mathbb{DH}_\mathbb{C})$.
\end{rmk}

Let us add some further properties of dual stem functions.

\begin{prop}\label{prop:symmetriesofdualstem}
    Let $F=G+\epsilon H\in Stem^{1,0}(D,\mathbb{DH}_\mathbb{C})$ and $a,b\in\mathbb{DR}$. If $a+e_1b\in\check{D}$, then $\check{F}\left(\overline{a+e_1b}\right)=\overline{\check{F}(a+e_1b)}$. If $a+e_1\epsilon b_2\in\check{D}$, then $\check{F}_0(a+e_1\epsilon b_2)=G_0(a_1)+\epsilon\left(H_0(a_1)+a_2\,\frac{\partial G_0}{\partial \alpha}(a_1)\right)$ and $\check{F}_1(a+e_1\epsilon b_2)=\epsilon b_2\,\frac{\partial G_1}{\partial \beta}(a_1)\in\epsilon\mathbb{H}$. In particular, $\check{F}_1(a)=0$ if $a\in\check{D}\cap\mathbb{DR}$.
\end{prop}

\begin{proof}
        
    Since $F=G+\epsilon H=G_0+\epsilon H_0+e_1(G_1+\epsilon H_1),\frac{\partial G}{\partial z}$ and $\frac{\partial G}{\partial \overline{z}}$ are stem functions, they commute with the involution $z\mapsto\bar z$. In particular, $G_1+\epsilon H_1$ (whence $\frac{\partial G_1}{\partial\alpha}$) vanishes identically in $D\cap\mathbb{R}$. So does $\frac{\partial G_0}{\partial\beta}$. For all $z=z_1+\epsilon z_2\in\check{D}$,
    \begin{align*}
        \check{F}(\bar z)&=\check{F}(\bar z_1+\epsilon \bar z_2)=F(\bar z_1)+\epsilon\left(\bar z_2\,\frac{\partial G}{\partial z}(\bar z_1)+z_2\,\frac{\partial G}{\partial \overline{z}}(\bar z_1)\right)\\
        &=\overline{F(z_1)}+\epsilon\left(\bar z_2\,\overline{\frac{\partial G}{\partial z}(z_1)}+z_2\,\overline{\frac{\partial G}{\partial \overline{z}}(z_1)}\right)\\
        &=\overline{F(z_1)+\epsilon\left(z_2\,\frac{\partial G}{\partial z}(z_1)+\bar z_2\,\frac{\partial G}{\partial \overline{z}}(z_1)\right)}=\overline{\check{F}(z)}\,,
    \end{align*}
    where we have used that fact that $\epsilon,z_2,\bar z_2$ belong to the center $\mathbb{DR}_\mathbb{C}$ of $\mathbb{DH}_\mathbb{C}$. If $a=a_1+\epsilon a_2\in\mathbb{DR}$ has  $a_1\in D$ and if $b_2\in\mathbb{R}$, then (using~\eqref{eq:componentsofextendedstem})
    \begin{align*}
            \check{F}_0(a+e_1\epsilon b_2)&=G_0(a_1)+\epsilon\left(H_0(a_1)+a_2\,\frac{\partial G_0}{\partial \alpha}(a_1)+b_2\,\frac{\partial G_0}{\partial \beta}(a_1)\right)\\
            &=G_0(a_1)+\epsilon\left(H_0(a_1)+a_2\,\frac{\partial G_0}{\partial \alpha}(a_1)\right)\,,\\
           \check{F}_1(a+e_1\epsilon b_2)&=G_1(a_1)+\epsilon\left(H_1(a_1)+a_2\,\frac{\partial G_1}{\partial \alpha}(a_1)+b_2\,\frac{\partial G_1}{\partial \beta}(a_1)\right)=\epsilon b_2\,\frac{\partial G_1}{\partial \beta}(a_1)\,.
    \end{align*}
    It follows at once that $\check{F}_1(a+e_1\epsilon b_2)\in\epsilon\mathbb{H}$ and that $\check{F}_1(a)=0$, as desired.
\end{proof}

%%%%%%%%%%%%%%%%%%%%%%%%%%%%%

\section{Dual slice functions}\label{sec:dualslice}

We are now ready for the announced construction of a new function class over $\mathbb{DH}$. Again, we fix an open subset $D$ of $\mathbb{R}_\mathbb{C}$ that is invariant with respect to conjugation.

\begin{defn}\label{def:dualslice}
    Let $\check{F}=\check{F}_0+e_1\check{F}_1\in DStem^0(\check{D},\mathbb{DH}_\mathbb{C})$. The \emph{dual slice function induced by $\check{F}$} is the function $\mathcal{I}(\check{F}):\widecheck{\Omega}_D\to\mathbb{DH}$ with
    \[\mathcal{I}(\check{F})(a+Ib)=\check{F}_0(a+e_1b)+I\check{F}_1(a+e_1b)\]
    for all $a+e_1b\in\check{D},I\in\mathbb{S}_\mathbb{DH}$. The function $\mathcal{I}(\check{F})$ is called \emph{circular} if $\check{F}$ takes values in $\mathbb{DH}$, i.e., if $\check{F}_1\equiv0$.
    
    We define the set $\mathcal{DS}(\widecheck{\Omega}_D)\coloneqq\mathcal{I}(DStem^0(\check{D},\mathbb{DH}_\mathbb{C}))$ and call its elements \emph{dual slice functions} on $\widecheck{\Omega}_D$. We also set $\mathcal{DS}^k(\widecheck{\Omega}_D)\coloneqq\mathcal{I}(DStem^k(\check{D},\mathbb{DH}_\mathbb{C}))$ for all $k\in\mathbb{N}\cup\{\infty,\omega\}$. In particular, $\mathcal{DS}^0(\widecheck{\Omega}_D)=\mathcal{DS}(\widecheck{\Omega}_D)$.
\end{defn}

In other words, $\mathcal{I}(\check{F})$ is defined to make the following diagrams commute for every $I\in\mathbb{S}_\mathbb{DH}$:
\begin{center}
    \begin{tikzcd}[ampersand replacement=\&,cramped]
	{\check{D}} \&\& {\mathbb{DH}_\mathbb{C}} \\
	\& \circlearrowleft \\
	{\widecheck{\Omega}_D} \&\& {\mathbb{DH}}
	\arrow["{\check{F}}", from=1-1, to=1-3]
	\arrow["{\phi_I}"', from=1-1, to=3-1]
	\arrow["{\phi_I}", from=1-3, to=3-3]
	\arrow["\mathcal{I}(\check{F})"', from=3-1, to=3-3].
    \end{tikzcd}
    \qquad
    \begin{tikzcd}[ampersand replacement=\&,cramped]
	{\mathbb{S}_\mathbb{DH}\times\check{D}} \&\& {\mathbb{S}_\mathbb{DH}\times\mathbb{DH}_\mathbb{C}} \\
	\& \circlearrowleft \\
	{\widecheck{\Omega}_D} \&\& {\mathbb{DH}}
	\arrow["{(id,\check{F})}", from=1-1, to=1-3]
	\arrow["{\Phi}"', from=1-1, to=3-1]
	\arrow["{\Phi}", from=1-3, to=3-3]
	\arrow["\mathcal{I}(\check{F})"', from=3-1, to=3-3].
    \end{tikzcd}
\end{center}

We can prove that $\mathcal{I}(\check{F})$ is well defined using property {\it 3} in Proposition~\ref{prop:decompositionintoDCs}. Indeed, for $a,b=b_1+e_1b_2\in\mathbb{DR}$ with $b_1>0$ and all $I,J\in\mathbb{S}_\mathbb{DH}$ we have
    \begin{align*}
        \Phi(I,\check{F}(a+e_1b))&=\check{F}_0(a+e_1b)+I\check{F}_1(a+e_1b)=\check{F}_0(a-e_1b)-I\check{F}_1(a-e_1b)\\
        &=\Phi(-I,\check{F}(a-e_1b))\\
        \Phi(I,\check{F}(a+e_1\epsilon b_2))&=\check{F}_0(a+e_1\epsilon b_2)+I\check{F}_1(a+e_1\epsilon b_2)=\check{F}_0(a+e_1\epsilon b_2)+I_1\check{F}_1(a+e_1\epsilon b_2)\\
        &=\Phi(I_1,\check{F}(a+e_1\epsilon b_2))=\Phi(-I_1,\check{F}(a-e_1\epsilon b_2))\\
        \Phi(I,\check{F}(a))&=\check{F}_0(a)+I\check{F}_1(a)=\check{F}_0(a)=\Phi(J,\check{F}(a))
    \end{align*}
whenever these expressions are defined. Here, we have repeatedly used Proposition~\ref{prop:symmetriesofdualstem}.

Dual slice functions $\widecheck{\Omega}_D\to\mathbb{DH}$ can also be seen as extensions of slice functions $\Omega_D\to\mathbb{DH}$. Indeed, for all $a_1+Ib_1\in\Omega_D$ we have $\mathcal{I}(\check{F})(a_1+Ib_1)=\check{F}_0(a_1+e_1b_1)+I\check{F}_1(a_1+e_1b_1)=F_0(a_1+e_1b_1)+IF_1(a_1+e_1b_1)=\mathcal{I}(F)(a_1+Ib_1)$. In other words, $\mathcal{I}(\check{F})_{|_{\Omega_D}}=\mathcal{I}(F)$. This justifies the next definition.

\begin{defn}\label{def:dualslice2}
    Let $F\in Stem^{1,0}(D,\mathbb{DH}_\mathbb{C})$ and $f\coloneqq\mathcal{I}(F)\in\mathcal{S}^{1,0}(\Omega_D,\mathbb{DH})$. Then $\check{f}\coloneqq\mathcal{I}(\check{F})$ is called the \emph{natural extension} of $f$ to $\widecheck{\Omega}_D$.
    
    We define the set $\mathcal{DSR}(\widecheck{\Omega}_D)\coloneqq\{\check{f}:f\in\mathcal{SR}(\Omega_D)\}$ and call its elements \emph{dual slice regular functions} on $\widecheck{\Omega}_D$. The symbols $\mathcal{DSR}_\mathbb{R}(\widecheck{\Omega}_D),\mathcal{DSR}_\mathbb{DR}(\widecheck{\Omega}_D)$ will denote the sets of extensions of elements of $\mathcal{SR}_\mathbb{R}(\Omega_D),\mathcal{SR}_\mathbb{DR}(\Omega_D)$, respectively.
\end{defn}

Let $h,k\in\mathbb{N}\cup\{\infty,\omega\}$. If we set $\mathcal{S}^{h,k}(\Omega_D)\coloneqq\mathcal{I}(Stem^{h,k}(D,\mathbb{DH}_\mathbb{C}))=\mathcal{S}^{h}_\mathbb{H}(\Omega_D)+\epsilon\,\mathcal{S}^{k}_\mathbb{H}(\Omega_D)$, then $\mathcal{DS}^k(\widecheck{\Omega}_D)$ is the set of natural extensions of elements of $\mathcal{S}^{k+1,k}(\Omega_D)=\mathcal{S}^{k+1}_\mathbb{H}(\Omega_D)+\epsilon\,\mathcal{S}^{k}_\mathbb{H}(\Omega_D)$. For future use, we also define $\mathcal{S}_\mathbb{DR}^{h,k}(\Omega_D)\coloneqq\mathcal{I}(Stem^{h,k}(D,\mathbb{DR}_\mathbb{C}))=\mathcal{S}^{h}_\mathbb{R}(\Omega_D)+\epsilon\,\mathcal{S}^{k}_\mathbb{R}(\Omega_D)\supseteq\mathcal{S}^{h}_\mathbb{R}(\Omega_D)$.

\begin{ex}
    Consider the (slice-preserving) polynomial slice functions $\operatorname{Re},\operatorname{Im}\in\mathcal{S}^{pol}(Q_\mathbb{DH})$, mapping every $x=a_1+Ib_1\in Q_\mathbb{DH}$ (with $a_1,b_1\in\mathbb{R},I\in\mathbb{S}_\mathbb{DH}$) into $\operatorname{Re}(x)=\frac{x+x^c}{2}=\frac{t(x)}{2}=a_1$ and $\operatorname{Im}(x)=\frac{x-x^c}{2}=x-\frac{t(x)}{2}=Ib_1$, respectively. They are induced by stem functions $RE,IM:\mathbb{R}_\mathbb{C}\to\mathbb{R}_\mathbb{C}$, whose natural extensions we already treated in Examples~\ref{ex:REIM} and~\ref{ex:REIMcomponents}. Perusing the formulas in the latter example, we find that the slice-preserving dual slice functions $\widecheck{\operatorname{Re}}=\mathcal{I}(\widecheck{RE}),\widecheck{\operatorname{Im}}=\mathcal{I}(\widecheck{IM})$ can be expressed as \[\widecheck{\operatorname{Re}}(a+Ib)=a,\quad\widecheck{\operatorname{Im}}(a+Ib)=Ib\]
    for all $a,b\in\mathbb{DR}$. Taking into account Proposition~\ref{prop:classes}, in each equivalence class $[a+Ib]$ we see that $\widecheck{\operatorname{Re}}$ is constantly equal to $a$ and $\widecheck{\operatorname{Im}}$ coincides with the translation $x\mapsto x-a$. We also have the alternative expressions
    \begin{align*}
        \widecheck{\operatorname{Re}}(x)=\frac{t(x)}{2}=\frac{x+x^c}{2}\,,\quad\widecheck{\operatorname{Im}}(x)=x-\frac{t(x)}{2}=\frac{x-x^c}{2}\,,
    \end{align*}
    valid for all $x\in\mathbb{DH}$. We remark that $\widecheck{\operatorname{Re}}(x_1+\epsilon x_2)=\operatorname{Re}(x_1)+\epsilon\operatorname{Re}(x_2),\widecheck{\operatorname{Im}}(x_1+\epsilon x_2)=\operatorname{Im}(x_1)+\epsilon\operatorname{Im}(x_2)$ for all $x_1,x_2\in\mathbb{H}$. In other words, $\widecheck{\operatorname{Re}}$ is the natural projection $\mathbb{DH}\to\mathbb{DR}$ and $\widecheck{\operatorname{Im}}$ is the natural projection $\mathbb{DH}\to\operatorname{Im}(\mathbb{H})+\epsilon\operatorname{Im}(\mathbb{H})$. In particular, the zero set of $\widecheck{\operatorname{Im}}$ is $\mathbb{DR}$ and the set where $\widecheck{\operatorname{Im}}$ takes values that are not invertible is $\mathbb{R}+\epsilon\mathbb{H}$.
\end{ex}

We point out that circular dual slice functions on $\widecheck{\Omega}_D$ are exactly the natural extensions to $\widecheck{\Omega}_D$ of circular slice functions on $\Omega_D$. Just as the latter were constant on each sphere $\mathbb{S}_x^\mathbb{DH}$ with $x\in\Omega_D$, the former are constant on classes.

\begin{prop}\label{prop:circular}
    If $\check{f}\in\mathcal{DS}(\widecheck{\Omega}_D)$ is circular, then it is constant in $[y]$ for every $y\in\widecheck{\Omega}_D$.
\end{prop}

\begin{proof}
    Let us inspect Definition~\ref{def:dualslice}. Since $\check{f}$ is circular, $\check{F}=\check{F}_0,\check{F}_1\equiv0$ and the value $\check{f}(a+Ib)=\check{F}_0(a+e_1b)$ is independent of the choice of $I$. Thus, $\check{f}$ is constant in every $a+b\mathbb{S}_\mathbb{DH}$. This proves the statement for all $y=a+Ib\in\widecheck{\Omega}_D\setminus(\mathbb{R}+\epsilon\mathbb{H})$, which have $[y]=a+b\mathbb{S}_\mathbb{DH}$ by Proposition~\ref{prop:classes}.
    
    Now assume $y=a+\epsilon Ib_2\in\widecheck{\Omega}_D\cap(\mathbb{R}+\epsilon\mathbb{H})$, whence $[y]=a+\epsilon\operatorname{Im}(\mathbb{H})$ by Proposition~\ref{prop:classes}. By Proposition~\ref{prop:symmetriesofdualstem}, we have
    \begin{align*}
        \check{f}(a+\epsilon Ib_2)&=\check{F}_0(a+e_1\epsilon b_2)=F_0(a_1)+\epsilon a_2\,\frac{\partial G_0}{\partial \alpha}(a_1)\,.
    \end{align*}
    Since this expression does not depend on $b_2$, we conclude that $\check{f}$ is constant in $a+\epsilon\operatorname{Im}(\mathbb{H})$, as desired.
\end{proof}

Circular dual slice functions include the natural extensions $\check{f}^\circ_s,\check{f}'_s$ of $f^\circ_s,f'_s$, defined, respectively, in $\widecheck{\Omega}_D,\widecheck{\Omega}_{D\setminus\mathbb{R}}$ (or both in $\widecheck{\Omega}_D$ if $f\in\mathcal{S}^{2,1}(\Omega_D)$, which guarantees $f'_s\in\mathcal{S}^{1,0}(\Omega_D)$).

\begin{rmk}
    Assume $f=\mathcal{I}(F)\in\mathcal{S}^{1,0}(\Omega_D)$ with $F=F_0+e_1F_1$. Taking into account that $f^\circ_s=\mathcal{I}(F_0)$ and that $f'_s$ is the circular slice function induced by $a_1+e_1b_1\mapsto b_1^{-1}F_1(a_1+e_1b_1)$, we conclude that $\check{f}^\circ_s=\mathcal{I}(\check F_0)$ and that $\check{f}'_s$ is the circular dual slice function induced by $a+e_1b\mapsto b^{-1}\check F_1(a+e_1b)$. Thus, for all $a+e_1b\in\widecheck{D},I\in\mathbb{S}_\mathbb{DH}$, $\check{f}^\circ_s(a+Ib)=\check{F}_0(a+e_1b)$ and, provided $b_1\neq0$, $\check{f}'_s(a+Ib)=b^{-1}\check{F}_1(a+e_1b)$. An inspection in Definition~\ref{def:dualslice} allows us to conclude that $\check{f}^\circ_s(x)=\frac{\check{f}(x)+\check{f}(x^c)}{2}$ for all $x\in\widecheck{\Omega}_D$ and $\check{f}'_s(x)=(x-x^c)^{-1}\left(\check{f}(x)-\check{f}(x^c)\right)$ for all $x\in\widecheck{\Omega}_{D\setminus\mathbb{R}}$. In particular, $\check{f}^\circ_s(x)=\check{f}(x)$ for all $x\in\widecheck{\Omega}_D\cap\mathbb{DR}$.
\end{rmk}

%%%%%%%%%%%%%%%%%%%%%%%%%%%%%

\section{The algebra of dual slice functions}\label{sec:dualslicealgebra}

The set of dual slice functions admits the algebraic structures described in the next result.

\begin{prop}
        The class $\mathcal{DS}(\widecheck{\Omega}_D)$ is a right $\mathbb{DH}$-module and natural extension is a right $\mathbb{DH}$-module isomorphism $\mathcal{S}^{1,0}(\Omega_D)\to\mathcal{DS}(\widecheck{\Omega}_D)$, whose inverse map is restriction to $\Omega_D$. The real vector space $\mathcal{DS}(\widecheck{\Omega}_D)$ admits a unique real $\ast$-algebra structure $(\mathcal{DS}(\widecheck{\Omega}_D),+,\cdot,\,^c)$ making natural extension a real $\ast$-algebra isomorphism. Additionally: for any $f\in\mathcal{S}^{1,0}(\Omega_D)$, the natural extension of the map $x\mapsto f(x^c)$ is the map $x\mapsto\check{f}(x^c)$, which is therefore itself an element of $\mathcal{DS}(\widecheck{\Omega}_D)$.
\end{prop}

\begin{proof}
    By direct inspection in Definition~\ref{def:dualslice}, the class $\mathcal{DS}(\widecheck{\Omega}_D)$ is a right $\mathbb{DH}$-module and the extension map $\widecheck{\ }:\mathcal{S}^{1,0}(\Omega_D)\to\mathcal{DS}(\widecheck{\Omega}_D)$ is a right $\mathbb{DH}$-module morphism. Since we already established that $\check{f}_{|_{\Omega_D}}=f$ for all $f\in\mathcal{S}^{1,0}(\Omega_D)$, the extension map $\widecheck{\ }$ is a right $\mathbb{DH}$-module isomorphism whose inverse map is restriction to $\Omega_D$. This proves the first statement. The second statement follows if we set $\check{f}\cdot\check{g}\coloneqq\widecheck{f\cdot g}$ and $\left(\check{f}\right)^c\coloneqq \widecheck{f^c}$. To prove the third statement, we recall the following facts: if $f=\mathcal{I}(F)$, then the map $x\mapsto f(x^c)$ is the slice function induced by $\overline{F}$; $\widecheck{\overline{F}}=\overline{\check{F}}$. Using Definition~\ref{def:dualslice}, we get that the natural extension of $x\mapsto f(x^c)$ is the map $\mathcal{I}\left(\widecheck{\overline{F}}\right)=\mathcal{I}\left(\overline{\check{F}}\right)$, where
    \begin{align*}
        \mathcal{I}\left(\overline{\check{F}}\right)(a+Ib)&=\check{F}_0(a+e_1b)-I\check{F}_1(a+e_1b)=\mathcal{I}(\check{F})(a-Ib)=\check{f}(a-Ib)=\check{f}\left((a+Ib)^c\right)
    \end{align*}
    for all $a+Ib\in\widecheck{\Omega}_D$.
\end{proof}

For all $k\in\mathbb{N}\cup\{\infty,\omega\}$, by construction, $\widecheck{\ }$ maps $\mathcal{S}^{k+1,k}(\Omega_D),\mathcal{SR}(\Omega_D)$ isomorphically into $\mathcal{DS}^k(\widecheck{\Omega}_D),\mathcal{DSR}(\widecheck{\Omega}_D)$, respectively. We enrich the structure of $\mathcal{DS}(\widecheck{\Omega}_D)$ with the next definition.

\begin{defn}
    For all $\check{f}\in\mathcal{DS}(\widecheck{\Omega}_D)$, we set $N(\check{f})\coloneqq\check{f}\cdot\check{f}^c$ and $T(\check{f})\coloneqq\check{f}+\check{f}^c$.
\end{defn}

By construction, $N(\check{f})=\widecheck{N(f)}$ and $T(\check{f})=\widecheck{T(f)}$. Since extension $f\mapsto\check{f}$ is a $*$-algebra isomorphism and since $N(f)=N(f^c)$ for all $f\in\mathcal{S}^{1,0}(\Omega_D)$, it is easy to see that $N(\check{f})=N(\check{f}^c)=\check{f}^c\cdot\check{f}$.

We now establish a direct correspondence between stem functions and dual slice functions.

\begin{rmk}
    We can define a right $\mathbb{DH}$-module isomorphism and $\ast$-algebra isomorphism $\mathscr{I}:Stem^{1,0}(D,\mathbb{DH}_\mathbb{C})\to\mathcal{DS}(\widecheck{\Omega}_D)$ by precomposing the extension isomorphism $\widecheck{\ }:\mathcal{S}^{1,0}(\Omega_D)\to\mathcal{DS}(\widecheck{\Omega}_D)$ with the isomorphism $\mathcal{I}:Stem^{1,0}(D,\mathbb{DH}_\mathbb{C})\to\mathcal{S}^{1,0}(\Omega_D)$. In other words, $\mathscr{I}(F)=\widecheck{\mathcal{I}(F)}$. We remark that $\check{f}=\mathscr{I}(F)$ implies that $\mathscr{I}(\overline{F})(x)=\check{f}(x^c)$ for all $x\in\widecheck{\Omega}_D$.
    
    By construction, $\mathscr{I}(F)=\mathcal{I}(\check{F})$. It follows that the isomorphism $\mathscr{I}$ is also the map $\mathcal{I}:DStem^0(D,\mathbb{DH}_\mathbb{C})\to\mathcal{DS}(\widecheck{\Omega}_D)$, precomposed with the isomorphism $\widecheck{\ }:Stem^{1,0}(D,\mathbb{DH}_\mathbb{C})\to DStem^0(\check{D},\mathbb{DH}_\mathbb{C})$. In other words, the diagram
    \[\begin{tikzcd}[ampersand replacement=\&,cramped]
	   {Stem^{1,0}(D,\mathbb{DH}_\mathbb{C})} \&\& {DStem^0(\check{D},\mathbb{DH}_\mathbb{C})} \\
	   \\
	   {\mathcal{S}^{1,0}(\Omega_D)} \&\& {\mathcal{DS}(\widecheck{\Omega}_D)}
	   \arrow["{\widecheck{\phantom{GG}}}", from=1-1, to=1-3]
	   \arrow["{\mathcal{I}}"', from=1-1, to=3-1]
	   \arrow["{\mathscr{I}}", from=1-1, to=3-3]
	   \arrow["{\mathcal{I}}", from=1-3, to=3-3]
	   \arrow["{\widecheck{\phantom{GG}}}"', from=3-1, to=3-3]
    \end{tikzcd}\]
    is commutative. As a consequence, the map $\mathcal{I}:DStem^0(D,\mathbb{DH}_\mathbb{C})\to\mathcal{DS}(\widecheck{\Omega}_D)$ is both a right $\mathbb{DH}$-module isomorphism and $\ast$-algebra isomorphism. It also maps the involution $\check{F}\mapsto\overline{\check{F}}$ into precomposition with $x\mapsto x^c$. In other words, $\mathcal{I}(\check{F}+\check{G})=\mathcal{I}(\check{F})+\mathcal{I}(\check{G}),\mathcal{I}(\check{F}h)=\mathcal{I}(\check{F})h,\mathcal{I}(\check{F}\check{G})=\mathcal{I}(\check{F})\cdot\mathcal{I}(\check{G}),\mathcal{I}(\check{F}^c)=\mathcal{I}(\check{F})^c$ and $\mathcal{I}\left(\overline{\check{F}}\right)(x)=\mathcal{I}(\check{F})(x^c)$ for all $\check{F},\check{G}\in DStem^0(D,\mathbb{DH}_\mathbb{C})$, all $h\in\mathbb{DH}$ and all $x\in\widecheck{\Omega}_D$.
\end{rmk}

We use the isomorphism $\mathscr{I}$ studied in the last remark to give the next definition.

\begin{defn}
    We define $\mathcal{DS}_\mathbb{R}(\widecheck{\Omega}_D)\coloneqq\mathscr{I}(Stem^{1,0}(D,\mathbb{R}_\mathbb{C}))=\widecheck{\mathcal{S}_\mathbb{R}^{1,0}(\Omega_D)}$ and call its elements \emph{slice-preserving} dual slice functions. We define $\mathcal{DS}_\mathbb{DR}(\widecheck{\Omega}_D)\coloneqq\mathscr{I}(Stem^{1,0}(D,\mathbb{DR}_\mathbb{C}))=\widecheck{\mathcal{S}_\mathbb{DR}^{1,0}(\Omega_D)}=\mathcal{I}(DStem^0(\check{D},\mathbb{DR}_\mathbb{C}))$ and call its elements \emph{dual slice-preserving} dual slice functions. We define $\mathcal{DS}_\mathbb{H}(\widecheck{\Omega}_D)\coloneqq\mathscr{I}(Stem^{1,0}(D,\mathbb{H}_\mathbb{C}))=\widecheck{\mathcal{S}_\mathbb{H}^{1,0}(\Omega_D)}$ and call its elements \emph{quaternion-preserving} dual slice functions. For $A'\in\{\mathbb{R},\mathbb{DR},\mathbb{H}\}$ and $k\in\mathbb{N}\cup\{\infty,\omega\}$, we set $\mathcal{DS}^k_{A'}(\widecheck{\Omega}_D)\coloneqq\mathcal{DS}_{A'}(\widecheck{\Omega}_D)\cap\mathcal{DS}^k(\widecheck{\Omega}_D)$ and $\mathcal{DSR}_{A'}(\widecheck{\Omega}_D)\coloneqq\mathcal{DS}_{A'}(\widecheck{\Omega}_D)\cap\mathcal{DSR}(\widecheck{\Omega}_D)$.
\end{defn}

The terminology we chose is motivated by the next remark.

\begin{rmk}
    If $\check{f}$ is slice-preserving, then $\check{f}(\widecheck{\Omega}_D\cap\mathbb{C}_I)\subseteq\mathbb{C}_I$ for all $I\in\mathbb{S}_\mathbb{DH}$. If $\check{f}$ is dual slice-preserving, then $\check{f}(\widecheck{\Omega}_D\cap\mathbb{DC}_I)\subseteq\mathbb{DC}_I$ for all $I\in\mathbb{S}_\mathbb{DH}$. Finally: if $\check{f}$ is quaternion-preserving, then $f$ is quaternion-preserving; taking into account that $\mathbb{H}\subset Q_\mathbb{DH}$, we conclude that $\check f(\widecheck{\Omega}_D\cap\mathbb{H})=f(\Omega_D\cap\mathbb{H})\subseteq\mathbb{H}$. 
\end{rmk}

We have $\mathcal{DS}_\mathbb{R}(\widecheck{\Omega}_D)=\mathcal{DS}_\mathbb{H}(\widecheck{\Omega}_D)\cap\mathcal{DS}_\mathbb{DR}(\widecheck{\Omega}_D)$. In other words, $\check{f}$ is slice-preserving if, and only if, it is both quaternion-preserving and dual slice-preserving. For any $A'\in\{\mathbb{R},\mathbb{DR},\mathbb{H}\}$ and any $k\in\mathbb{N}\cup\{\infty,\omega\}$, the properties of $\mathscr{I}$ yield that $\mathcal{DS}_{A'}(\widecheck{\Omega}_D),\mathcal{DS}^k_{A'}(\widecheck{\Omega}_D),\mathcal{DSR}_{A'}(\widecheck{\Omega}_D)$ are nested real $^\ast$-subalgebras and right $A'$-submodules of $\mathcal{DS}(\widecheck{\Omega}_D)$. As a consequence of \eqref{eq:decomposition of slice regular}, we find that
\begin{align*}
    \mathcal{DS}^k(\widecheck{\Omega}_D)&=\mathcal{DS}^k_\mathbb{H}(\widecheck{\Omega}_D)+\mathcal{DS}^k_\mathbb{H}(\widecheck{\Omega}_D)\epsilon=\mathcal{DS}^k_\mathbb{H}(\widecheck{\Omega}_D)+\epsilon\mathcal{DS}^k_\mathbb{H}(\widecheck{\Omega}_D)\,,\\
    \mathcal{DSR}(\widecheck{\Omega}_D)&=\mathcal{DSR}_\mathbb{H}(\widecheck{\Omega}_D)+\mathcal{DSR}_\mathbb{H}(\widecheck{\Omega}_D)\epsilon=\mathcal{DSR}_\mathbb{H}(\widecheck{\Omega}_D)+\epsilon\mathcal{DSR}_\mathbb{H}(\widecheck{\Omega}_D)
\end{align*}
for any $k\in\mathbb{N}\cup\{\infty,\omega\}$.

\begin{defn}
    Given $\check f\in\mathcal{DS}(\widecheck{\Omega}_D)$, the unique elements $\check g$ and $\check h$ of $\mathcal{DS}_\mathbb{H}
    (\widecheck{\Omega}_D)$ such that $\check f=\check g+\epsilon \check h$ are called the first and second \emph{quaternion-preserving components} of $\check f$.
\end{defn}

Now fix $f=g+\epsilon h$ with $g=\mathcal{I}(G)\in\mathcal{S}^1_\mathbb{H}(\Omega_D),h=\mathcal{I}(H)\in\mathcal{S}^0_\mathbb{H}(\Omega_D)$. By construction, the first quaternion-preserving component $\check{g}$ of $\check{f}$ is the natural extension of the first quaternion-preserving component $g$ of $f=\check f|_{\Omega_D}$. As for the second quaternion-preserving component $\check{h}$, we remark that  $\epsilon\check{h}$ is the natural extension of $\epsilon h$, where $h$ is the second quaternion-preserving component of $f$. We point out that $\check{f}$ is: quaternion-preserving if, and only if, $\check{h}\equiv0$; dual slice-preserving if, and only if, $\check{g},\check{h}$ are slice-preserving; slice-preserving if, and only if, $\check{g}$ is slice-preserving and $\check{h}\equiv0$.

\begin{rmk}
    If $\check f=\check g+\epsilon \check h,\underline{\check f}=\underline{\check g}+\epsilon \underline{\check h}$ with $\check g,\check h,\underline{\check g},\underline{\check h}\in\mathcal{DS}_\mathbb{H}(\widecheck{\Omega}_D)$, then
    \[\check f\cdot\underline{\check f}=\check g\cdot\underline{\check g}+\epsilon(\check g\cdot\underline{\check h}+\check h\cdot\underline{\check g})\,,\]
    where $\check g\cdot\underline{\check g},\check g\cdot\underline{\check h}+\check h\cdot\underline{\check g}$ both belong to $\mathcal{DS}_\mathbb{H}(\widecheck{\Omega}_D)$.
\end{rmk}

The $\cdot$ product of dual slice functions coincides with the pointwise product in some special cases.

\begin{prop}\label{prop:productofdualslicepointwise}
        Let $\check{f},\check{g}\in\mathcal{DS}(\widecheck{\Omega}_D)$ and let $x\in\widecheck{\Omega}_D$. The equality
        \begin{equation}\label{eq:specialproductofdualslice}
            (\check{f}\cdot\check{g})(x)=\check{f}(x)\check{g}(x)\,,
        \end{equation}        
        while false in general, holds true under any of the following additional hypotheses:
        \begin{enumerate}
            \item $\check{f}$ is dual slice-preserving, which also yields $\check{f}\cdot\check{g}=\check{g}\cdot\check{f}$;
            \item $\check{f},\check{g}$ are circular;
            \item $x\in\widecheck{\Omega}_D\cap\mathbb{DR}$.
        \end{enumerate}
\end{prop}

\begin{proof}
    Assume $\check{f}=\mathcal{I}(\check{F}),\check{g}=\mathcal{I}(\check{G})$, where $\check{F}=\check{F}_0+e_1\check{F}_1,\check{G}=\check{G}_0+e_1\check{G}_1$. We compute
    \[\check{F}\check{G}=\check{F}_0\check{G}_0-\check{F}_1\check{G}_1+e_1(\check{F}_0\check{G}_1+\check{F}_1\check{G}_0)\,.\]
    We fix $a+Ib\in\widecheck{\Omega}_D$ and compute
    \begin{align*}
        (\check{f}\cdot\check{g})(a+Ib)&=(\mathcal{I}(\check{F})\cdot\mathcal{I}(\check{G}))(a+Ib)=\mathcal{I}(\check{F}\check{G})(a+Ib)=\\
        &=(\check{F}_0\check{G}_0-\check{F}_1\check{G}_1)(a+e_1b)+I(\check{F}_0\check{G}_1+\check{F}_1\check{G}_0)(a+e_1b)\,.
    \end{align*}
    If $\check{f}$ is dual slice-preserving, then $\check{F}$ takes values in $\mathbb{DR}_\mathbb{C}$. On the one hand, this implies (using Remark~\ref{rmk:specialproductofstem}) that $\check{f}$ belongs to the center of $\mathcal{DS}(\widecheck{\Omega}_D)$, whence $\check{f}\cdot\check{g}=\check{g}\cdot\check{f}$. On the other hand, it implies that $\check{F}_0,\check{F}_1$ take values in $\mathbb{DR}$. Under this additional hypothesis, we obtain
    \begin{align*}
        (\check{f}\cdot\check{g})(a+Ib)&=(\check{F}_0(a+e_1b)+I\check{F}_1(a+e_1b))(\check{G}_0(a+e_1b)+I\check{G}_1(a+e_1b))\\
        &=\check{f}(a+Ib)\,\check{g}(a+Ib)\,,
    \end{align*}
    as desired.

    If $\check{f},\check{g}$ are circular or $b=0$, then $\check{F}_1(a+e_1b)=0=\check{G}_1(a+e_1b)$ (by Definition~\ref{def:dualslice} or by Proposition~\ref{prop:symmetriesofdualstem}, respectively). We obtain
    \begin{align*}
        (\check{f}\cdot\check{g})(a+Ib)&=(\check{F}_0\check{G}_0)(a+e_1b)=\check{f}(a+Ib)\,\check{g}(a+Ib)\,,
    \end{align*}
    as desired.
\end{proof}

\begin{ex}\label{ex:Im^2}
    The dual slice function $\widecheck{\operatorname{Im}}:\mathbb{DH}\to\operatorname{Im}(\mathbb{H})+\epsilon\operatorname{Im}(\mathbb{H})$ has $(\widecheck{\operatorname{Im}}\cdot\check{g})(x)=\widecheck{\operatorname{Im}}(x)\,\check{g}(x)$ for any slice function $g$. In particular, $(\widecheck{\operatorname{Im}}\cdot\widecheck{\operatorname{Im}})(a+Ib)=\widecheck{\operatorname{Im}}^2(a+Ib)=-b^2$. Thus, $\widecheck{\operatorname{Im}}^2=\mathcal{I}\left(\widecheck{IM}^2\right)$. We also have $\widecheck{\operatorname{Im}}^2=\widecheck{\operatorname{Im}}\cdot\widecheck{\operatorname{Im}}=\widecheck{(\operatorname{Im}\cdot\operatorname{Im})}=\widecheck{\operatorname{Im}^2}$.
\end{ex}

All ingredients are now ready for the following representation formula of dual slice functions on classes.

\begin{rmk}\label{rmk:representationformula}
    Let $f\in\mathcal{S}^{1,0}(\Omega_D)$. The equality $f=f^\circ_s+\operatorname{Im}\cdot f'_s$, valid in $\Omega_{D\setminus\mathbb{R}}$, yields that
    \[\check{f}=\check{f}^\circ_s+\widecheck{\operatorname{Im}}\cdot\check{f}'_s=\check{f}^\circ_s+\widecheck{\operatorname{Im}}\,\check{f}'_s\]
    in $\widecheck{\Omega}_{D\setminus\mathbb{R}}$. The same is true throughout $\widecheck{\Omega}_D$ if $f\in\mathcal{S}^{2,1}(\Omega_D)$. Now fix $w\in\widecheck{\Omega}_{D\setminus\mathbb{R}}$ (or any $w\in\widecheck{\Omega}_D$ if $f\in\mathcal{S}^{2,1}(\Omega_D)$): the functions $\check{f}^\circ_s,\check{f}'_s$ are constant in $[w]$ and the restriction
        \[[w]\to\mathbb{DH}\quad x\mapsto\check{f}(x)=\check{f}^\circ_s(w)+\widecheck{\operatorname{Im}}(x)\check{f}'_s(w)=\check{f}^\circ_s(w)+\frac{x-x^c}{2}\check{f}'_s(w)\]
    is a real affine map.
\end{rmk}

We also prove the next result.

\begin{prop}\label{prop:dualsliceproductofcomponents}
    For all $f,g\in\mathcal{S}^{1,0}(\Omega_D)$,
    \begin{align*}
        (\check f\cdot\check  g)(x)&=\check f^\circ_s(x)\check g^\circ_s(x)+\widecheck{\operatorname{Im}}(x)^2\check f'_s(x)\check g'_s(x)+\widecheck{\operatorname{Im}}(x)\left(\check f^\circ_s(x)\check g'_s(x)+\check f'_s(x)\check g^\circ_s(x)\right)\,,\\
        \check f^c(x)&=\check f^\circ_s(x)^c+\widecheck{\operatorname{Im}}(x)\,\check f'_s(x)^c\,,\\
        N(f)(x)&=n\left(\check f^\circ_s(x)\right)+\widecheck{\operatorname{Im}}(x)^2\,n\left(\check f'_s(x)\right)+\widecheck{\operatorname{Im}}(x)\,t\left(\check f^\circ_s(x)\check f'_s(x)^c\right)\,,
    \end{align*}
    in $\widecheck{\Omega}_{D\setminus\mathbb{R}}$. The same is true throughout $\widecheck{\Omega}_D$ if $f,g\in\mathcal{S}^{2,1}(\Omega_D)$.
\end{prop}

\begin{proof}
    Remark~\ref{rmk:sliceproductofcomponents} immediately implies that
    \begin{align*}
        \check f\cdot\check g&=\check f^\circ_s\cdot\check g^\circ_s+\widecheck{\operatorname{Im}^2}\cdot \check f'_s\cdot \check g'_s+\widecheck{\operatorname{Im}}\cdot \left(\check f^\circ_s\cdot\check g'_s+\check f'_s\cdot \check g^\circ_s\right)\,,\\
        \check f^c&=(\check f^\circ_s)^c+\widecheck{\operatorname{Im}}\cdot(\check f'_s)^c\,,\\
        N(\check f)&=N\left(\check f^\circ_s\right)+\widecheck{\operatorname{Im}^2}\cdot N\left(\check f'_s\right)+\widecheck{\operatorname{Im}}\cdot T\left(\check f^\circ_s\cdot(\check f'_s)^c\right)
    \end{align*}
    in $\widecheck{\Omega}_D$. The thesis now follows from the considerations made in Example~\ref{ex:Im^2} and from an application of Proposition~\ref{prop:productofdualslicepointwise} to the circular dual slice functions $\check f^\circ_s,\check f'_s,\check g^\circ_s,\check g'_s$.
\end{proof}

We now provide an explicit formula for $\check f\cdot\check g$. To this end, we first recall~\cite[Definition 2.10]{gstdualquaternions}.

\begin{defn}
We set
    \begin{equation*}\label{def:actiongen}
    \mathscr{C} :\mathbb{DH} \times \mathbb{DH} \longrightarrow \mathbb{DH}, \qquad (h,l)\longmapsto \mathscr{C}(h,l)\vcentcolon=
    \begin{cases}
        h^{-1}lh & \text{if\ }h\in\mathbb{DH}\setminus\epsilon\mathbb{H}\\
        h_2^{-1}lh_2 & \text{if\ }h\in\epsilon\mathbb{H}\setminus\{0\}\\
        l & \text{if\ }h=0
    \end{cases}
    \end{equation*}
and, for all $h,l\in\mathbb{DH}$, $\mathscr{C}_h(l):=\mathscr{C}(h,l)$.
\end{defn}

The same work~\cite{gstdualquaternions} proved that
\begin{equation}\label{eq:con}
    h\mathscr{C}_h(l)=lh
\end{equation}
for all $h,l \in \mathbb{DH}$ and that, for $l\in\mathbb{DH}$ fixed, $h_2^{-1}lh_2$ is a limit point of $h^{-1}lh$ as $h_1\to0$ and $l$ is a limit point of $h^{-1}lh$ as $h\to0$. It also stated a special case of the next lemma. We are going to use the commutator $[x_1,y_1]\coloneqq x_1y_1-y_1x_1$ of $x_1,y_1\in\mathbb{H}$, which equals $2\operatorname{Im}(x_1)\wedge\operatorname{Im}(y_1)$.

\begin{lem}\label{lem:C}
    Let $x=x_1+\epsilon x_2\in\mathbb{DH}$ and $h\in\mathbb{DH}$. Then $\mathscr{C}$ maps $\mathbb{DH}\times[x]$ surjectively into $[x]$ and the equalities $\widecheck{\operatorname{Re}}(\mathscr{C}_h(x))=\widecheck{\operatorname{Re}}(x),\widecheck{\operatorname{Im}}(\mathscr{C}_h(x))=\mathscr{C}_h(\widecheck{\operatorname{Im}}(x))$ and $T_{\mathscr{C}_h(x)}=T_{\mathscr{C}_h(x_1+\epsilon x_2^\parallel)}$ hold true. The equalities $\mathscr{C}_h(x)_1=\mathscr{C}_h(x_1),\mathscr{C}_h(x)_2=\mathscr{C}_h(x_2),\mathscr{C}_h(x)_2^\parallel=\mathscr{C}_h(x_2^\parallel),\mathscr{C}_h(x)_2^\perp=\mathscr{C}_h(x_2^\perp)$ hold true when $h_1=0$ or $h_2=0$. If, instead, $h_1\neq0\neq h_2$, then $\mathscr{C}_h(x)_1=\mathscr{C}_{h_1}(x_1),\mathscr{C}_h(x)_2=\mathscr{C}_{h_1}(x_2)+\left[\mathscr{C}_{h_1}(x_1),h_1^{-1}h_2\right],\mathscr{C}_h(x)_2^\parallel=\mathscr{C}_{h_1}(x_2^\parallel),\mathscr{C}_h(x)_2^\perp=\mathscr{C}_{h_1}(x_2^\perp)+\left[\mathscr{C}_{h_1}(x_1),h_1^{-1}h_2\right]$. 
\end{lem}

\begin{proof}
    First assume $h=h_1+\epsilon h_2$ to be an invertible element of $\mathbb{DH}$. Since $h^{-1}=h_1^{-1}-\epsilon h_1^{-1}h_2h_1^{-1}$, we get
    \begin{align*}
    \mathscr{C}_h(x)&=\mathscr{C}_h(x_1)+\epsilon\mathscr{C}_h(x_2)=h^{-1}x_1h+\epsilon h^{-1}x_2h\\
        &=\left(h_1^{-1}-\epsilon h_1^{-1}h_2h_1^{-1}\right)(x_1h_1+\epsilon x_1h_2)+\epsilon h_1^{-1}x_2h_1\\
        &=h_1^{-1}x_1h_1+\epsilon\left(h_1^{-1}x_2h_1+h_1^{-1}x_1h_2-h_1^{-1}h_2h_1^{-1}x_1h_1\right)\\
        &=\mathscr{C}_{h_1}(x_1)+\epsilon\left(\mathscr{C}_{h_1}(x_2)+\mathscr{C}_{h_1}(x_1)h_1^{-1}h_2-h_1^{-1}h_2\mathscr{C}_{h_1}(x_1)\right)\\
        &=\mathscr{C}_{h_1}(x_1)+\epsilon\left(\mathscr{C}_{h_1}(x_2)+\left[\mathscr{C}_{h_1}(x_1),h_1^{-1}h_2\right]\right)\\
        &=\mathscr{C}_{h_1}(x)+\epsilon\left[\mathscr{C}_{h_1}(x_1),h_1^{-1}h_2\right]\,,
    \end{align*}
    whence $\mathscr{C}_h(x)_1=\mathscr{C}_{h_1}(x_1),\mathscr{C}_h(x)_2=\mathscr{C}_{h_1}(x_2)+\left[\mathscr{C}_{h_1}(x_1),h_1^{-1}h_2\right]$. In the special case when $h_2=0$, we obtain $\mathscr{C}_h(x)_1=\mathscr{C}_h(x_1),\mathscr{C}_h(x)_2=\mathscr{C}_{h_1}(x_2)+0=\mathscr{C}_h(x_2)$. Assuming $x=a+Ib$ with $a,b\in\mathbb{DR}$ and $I=I_1+\epsilon I_2\in\mathbb{S}_\mathbb{DH}$, we get $x_1=a_1+I_1b_1,x_2^\parallel=a_2+I_1b_2,x_2^\perp=I_2b_1$ and
    \begin{align*}
    \mathscr{C}_h(x)
        &=a_1+\mathscr{C}_{h_1}(I_1)b_1+\epsilon\left(a_2+\mathscr{C}_{h_1}(I_1)b_2+\mathscr{C}_{h_1}(I_2)b_1+\left[a_1+\mathscr{C}_{h_1}(I_1)b_1,h_1^{-1}h_2\right]\right)\,.
    \end{align*}
    Here, $\langle\mathscr{C}_{h_1}(I_1),\mathscr{C}_{h_1}(I_2)\rangle=t\left(\mathscr{C}_{h_1}(I_1)\mathscr{C}_{h_1}(I_2)^c\right)=t\left(\mathscr{C}_{h_1}(I_1I_2^c)\right)=t(I_1I_2^c)=\langle I_1,I_2\rangle=0$.    On the one hand, we obtain $\mathscr{C}_h(x)_2^\parallel=\mathscr{C}_{h_1}(x_2^\parallel),\mathscr{C}_h(x)_2^\perp=\mathscr{C}_{h_1}(x_2^\perp)+\left[\mathscr{C}_{h_1}(x_1),h_1^{-1}h_2\right],T_{\mathscr{C}_h(x)}=T_{\mathscr{C}_{h_1}(x_1)+\epsilon\mathscr{C}_{h_1}(x_2^\parallel)}=T_{\mathscr{C}_h(x_1+\epsilon x_2^\parallel)}$, which yield $\mathscr{C}_h(x)_2^\parallel=\mathscr{C}_h(x_2^\parallel),\mathscr{C}_h(x)_2^\perp+0=\mathscr{C}_{h_1}(x_2^\perp)=\mathscr{C}_h(x_2^\perp)$ in the special case when $h_2=0$. On the other hand, we obtain that $\widecheck{\operatorname{Re}}(\mathscr{C}_h(x))=a_1+\epsilon a_2=a=\widecheck{\operatorname{Re}}(x)$ and that $\widecheck{\operatorname{Im}}(\mathscr{C}_h(x))=\mathscr{C}_{h_1}(I_1b_1)+\epsilon\left(\mathscr{C}_{h_1}(I_1b_2+I_2b_1)+\left[\mathscr{C}_{h_1}(I_1b_1),h_1^{-1}h_2\right]\right)=\mathscr{C}_h(Ib)=\mathscr{C}_h(\widecheck{\operatorname{Im}}(x))$.

    Now assume $h=\epsilon h_2$, whence $\mathscr{C}_h=\mathscr{C}_{h_2}$ by~\eqref{def:actiongen}. We may apply the previous part of the proof to $\mathscr{C}_{h_2}$ to get $\mathscr{C}_h(x)_1=\mathscr{C}_{h_2}(x)_1=\mathscr{C}_{h_2}(x_1)=\mathscr{C}_h(x_1),\mathscr{C}_h(x)_2=\mathscr{C}_{h_2}(x)_2=\mathscr{C}_{h_2}(x_2)=\mathscr{C}_h(x_2),\mathscr{C}_h(x)_2^\parallel=\mathscr{C}_{h_2}(x)_2^\parallel=\mathscr{C}_{h_2}(x_2^\parallel)=\mathscr{C}_h(x_2^\parallel),\mathscr{C}_h(x)_2^\perp=\mathscr{C}_{h_2}(x)_2^\perp=\mathscr{C}_{h_2}(x_2^\perp)=\mathscr{C}_h(x_2^\perp),T_{\mathscr{C}_h(x)}=T_{\mathscr{C}_{h_2}(x)}=T_{\mathscr{C}_{h_2}(x_1+\epsilon x_2^\parallel)}=T_{\mathscr{C}_h(x_1+\epsilon x_2^\parallel)}$. Finally, $\widecheck{\operatorname{Re}}(\mathscr{C}_h(x))=\widecheck{\operatorname{Re}}(\mathscr{C}_{h_2}(x))=\widecheck{\operatorname{Re}}(x),\widecheck{\operatorname{Im}}(\mathscr{C}_h(x))=\widecheck{\operatorname{Im}}(\mathscr{C}_{h_2}(x))=\mathscr{C}_{h_2}(\widecheck{\operatorname{Im}}(x))=\mathscr{C}_h(\widecheck{\operatorname{Im}}(x))$. The proof is now complete.
\end{proof}

We are now ready for the announced expression of $\check f\cdot\check g$.

\begin{thm}\label{thm:pointwiseprod}
Let $\check{f},\check{g}\in\mathcal{DS}(\widecheck{\Omega}_D)$ and fix $x\in\widecheck{\Omega}_D$. If $x\in\widecheck{\Omega}_D\cap\mathbb{DR}$, then $(\check{f}\cdot \check{g})(x)=\check{f}(x)\check{g}(x)$. Now assume $x\in \widecheck{\Omega}_{D\setminus\mathbb{R}}$ or $f\in\mathcal{S}^{2,1}(\Omega_D)$. If $y,z\in[x]$ fulfill one of the following (mutually equivalent) conditions:
\begin{enumerate}
\item $y\check{f}(y)-\check{f}(y)z=0$;
\item $z\check{f}^c(z)-\check{f}^c(z)y=0$;
\end{enumerate}
then
\[(\check{f}\cdot \check{g})(y)=\check{f}(y)\check{g}(z)\,.\]
Condition {\it 1} is equivalent to $z=\mathscr{C}(\check{f}(y),y)$, when $\check{f}(y)$ is invertible; it is equivalent to $z\in T_{\mathscr{C}(\check{f}(y),y)}$, when $\check{f}(y)$ is a zero divisor; it is automatically fulfilled when $\check{f}(y)=0$. Similarly, condition {\it 2} is equivalent to $y=\mathscr{C}(\check{f}^c(z),z)$ when $\check{f}^c(z)$ is invertible;  it is equivalent to $y\in T_{\mathscr{C}(\check{f}^c(z),z)}$ when $\check{f}^c(z)$ is a zero divisor;  it is automatically fulfilled when $\check{f}^c(z)=0$.
\end{thm}

\begin{proof}
Using Remark~\ref{rmk:representationformula} and Proposition~\ref{prop:dualsliceproductofcomponents}, it is easy to see that
\[(\check{f}\cdot \check{g})(y)=\check{f}(y)\check{g}^\circ_s(x)+\widecheck{\operatorname{Im}}(y)\check{f}(y)\check{g}'_s(x)\]
for all $y,z\in[x]$. The displayed expression coincides with
\[\check{f}(y)\check{g}(z)=\check{f}(y)\check{g}^\circ_s(x)+\check{f}(y)\widecheck{\operatorname{Im}}(z)\check{g}'_s(x)\]
whenever $\widecheck{\operatorname{Im}}(y)\check{f}(y)=\check{f}(y)\widecheck{\operatorname{Im}}(z)$, which is equivalent to condition {\it 1} because $\widecheck{\operatorname{Re}}(y)=\widecheck{\operatorname{Re}}(z)\in\mathbb{DR}$. The latter condition is automatically true if $\check{f}(y)=0$. It is equivalent to $z=\check{f}(y)^{-1}y\check{f}(y)=\mathscr{C}(\check{f}(y),y)$ if $\check{f}(y)$ is invertible. It is equivalent to $\epsilon(y_1 h_2-h_2z_1)$, whence to $z_1=h_2^{-1}y_1h_2=\mathscr{C}(\check{f}(y),y_1)$, in case $\check{f}(y)=\epsilon h_2$ for some $h_2\in\mathbb{H}$. Assuming $y=a+Ib$ with $a,b\in\mathbb{DR},I=I_1+\epsilon I_2\in\mathbb{S}_\mathbb{DH}$ (whence $y_1=a_1+I_1b_1$), the elements $z$ of $[x]=[y]=a+b\mathbb{S}_\mathbb{DH}$ having $z_1=\mathscr{C}(\check{f}(y),y_1)=a_1+\mathscr{C}(\check{f}(y),I_1)b_1$ are exactly those with $z_2^\parallel=a_2+\mathscr{C}(\check{f}(y),I_1)b_2=\mathscr{C}(\check{f}(y),y_2^\parallel)$, i.e., those with $z\in T_{\mathscr{C}(\check{f}(y),y_1+\epsilon y_2^\parallel)}=T_{\mathscr{C}(\check{f}(y),y)}$. The last equality follows from Lemma~\ref{lem:C}.

We now prove that conditions {\it 1} and {\it 2} are equivalent. Since $\widecheck{\operatorname{Re}}(y)=\widecheck{\operatorname{Re}}(z)\in\mathbb{DR}$, condition {\it 1} is equivalent to
\[\widecheck{\operatorname{Im}}(y)\,(\check{f}^\circ_s(x)+\widecheck{\operatorname{Im}}(y)\check{f}'_s(x))-(\check{f}^\circ_s(x)+\widecheck{\operatorname{Im}}(y)\check{f}'_s(x))\,\widecheck{\operatorname{Im}}(z)=0\,.\]
Taking conjugates on both sides, we find that condition {\it 1} is equivalent to each of the following conditions:
\begin{align*}
    &-(\check{f}^\circ_s(x)^c-\check{f}'_s(x)^c\widecheck{\operatorname{Im}}(y))\,\widecheck{\operatorname{Im}}(y)+\widecheck{\operatorname{Im}}(z)\,(\check{f}^\circ_s(x)^c-\check{f}'_s(x)^c\widecheck{\operatorname{Im}}(y))=0\,,\\
    &\check{f}'_s(x)^c\,\widecheck{\operatorname{Im}}(y)^2+\widecheck{\operatorname{Im}}(z)\check{f}^\circ_s(x)^c-(\check{f}^\circ_s(x)^c+\widecheck{\operatorname{Im}}(z)\check{f}'_s(x)^c)\,\widecheck{\operatorname{Im}}(y)=0\,,\\
    &\widecheck{\operatorname{Im}}(z)^2\check{f}'_s(x)^c+\widecheck{\operatorname{Im}}(z)\,\check{f}^\circ_s(x)^c-\check{f}^c(z)\,\widecheck{\operatorname{Im}}(y)=0\,,\\
    &\widecheck{\operatorname{Im}}(z)\,\check{f}^c(z)-\check{f}^c(z)\,\widecheck{\operatorname{Im}}(y)=0\,,
\end{align*}
the latter being equivalent to condition {\it 2}. For the second equivalence, we used the fact that $\widecheck{\operatorname{Im}}(y)^2=\widecheck{\operatorname{Im}}(z)^2\in\mathbb{DR}$.
\end{proof}

%%%%%%%%%%%%%%%%%%%%%%%%%%%%%

\section{Explicit expressions of the extension map}\label{sec:dualsliceexpression}

Let us explicitly compute the extension map $\widecheck{\ }:\mathcal{S}^{1,0}(\Omega_D)\to\mathcal{DS}(\widecheck{\Omega}_D)$.

\begin{prop}
    Let $g\in\mathcal{S}_\mathbb{H}^1(\Omega_D),h\in\mathcal{S}_\mathbb{H}^0(\Omega_D)$ and consider $f\coloneqq g+\epsilon h\in\mathcal{S}^{1,0}(\Omega_D)$. Then
    \begin{align}\label{eq:dualslice}
        \check{f}(x)=f(\pi_Q(x))+(x-\pi_Q(x))\,\partial_cg(\pi_{Q}(x))+(x-\pi_Q(x))^c\,\overline{\partial}_cg(\pi_{Q}(x))
    \end{align}
    for all $x\in\widecheck{\Omega}_D$. If $f\in\mathcal{S}^1(\Omega_D)=\mathcal{S}^{1,1}(\Omega_D)$, then~\eqref{eq:dualslice} holds true even with $f$ in lieu of $g$. Finally, if $f\in\mathcal{SR}(\Omega_D)$, then $
        \check{f}(x)=f(\pi_Q(x))+(x-\pi_Q(x))\,\partial_cf(\pi_{Q}(x))$
    for all $x\in\widecheck{\Omega}_D$.
\end{prop}

\begin{proof}
    Assume $x=a+Ib$ with $a,b\in\mathbb{DR},I\in\mathbb{S}_\mathbb{DH}$, whence $\pi_Q(x)=a_1+Ib_1$ and $x-\pi_Q(x)=\epsilon(a_2+Ib_2)$. The definition $\overline{\partial}_cg\coloneqq\mathcal{I}(\frac{\partial G}{\partial \overline{z}})$ and~\eqref{eq:CR} imply that
    \begin{align*}
        \overline{\partial}_cg(a_1+Ib_1)&=\left(\frac{\partial G}{\partial \overline{z}}\right)_0(a_1+e_1b_1)+I\left(\frac{\partial G}{\partial \overline{z}}\right)_1(a_1+e_1b_1)\\
        &=\frac{1}{2}\left(\frac{\partial G_0}{\partial\alpha}-\frac{\partial G_1}{\partial\beta}\right)(a_1+e_1b_1)+\frac{I}{2}\left(\frac{\partial G_1}{\partial\alpha}+\frac{\partial G_0}{\partial\beta}\right)(a_1+e_1b_1)
    \end{align*}
    Similarly,
        \begin{align*}
        \partial_cg(a_1+Ib_1)
        &=\frac{1}{2}\left(\frac{\partial G_0}{\partial\alpha}+\frac{\partial G_1}{\partial\beta}\right)(a_1+e_1b_1)+\frac{I}{2}\left(\frac{\partial G_1}{\partial\alpha}-\frac{\partial G_0}{\partial\beta}\right)(a_1+e_1b_1)
    \end{align*}
    This allows us to compute
    \begin{align*}
        &f(\pi_Q(x))+(x-\pi_Q(x))\,\partial_cg(\pi_Q(x))+(x-\pi_Q(x))\,\overline{\partial}_cg(\pi_Q(x))\\
        &=f(a_1+Ib_1)+\epsilon(a_2+Ib_2)\,\partial_cg(a_1+Ib_1)+\epsilon(a_2-Ib_2)\,\overline{\partial}_cg(a_1+Ib_1)\\
        &=F_0(a_1+e_1b_1)+IF_1(a_1+e_1b_1)+\epsilon a_2\left(\frac{\partial G_0}{\partial\alpha}+I\frac{\partial G_1}{\partial\alpha}\right)(a_1+e_1b_1)\\
        &\quad+\epsilon I b_2\left(\frac{\partial G_1}{\partial\beta}-I\frac{\partial G_0}{\partial\beta}\right)(a_1+e_1b_1)\\
        &=F_0(a_1+e_1b_1)+\epsilon\left(a_2\,\frac{\partial G_0}{\partial \alpha}(a_1+e_1b_1)+b_2\,\frac{\partial G_0}{\partial \beta}(a_1+e_1b_1)\right)\\
        &\quad+I\left(F_1(a_1+e_1b_1)+\epsilon\left(a_2\,\frac{\partial G_1}{\partial \alpha}(a_1+e_1b_1)+b_2\,\frac{\partial G_1}{\partial \beta}(a_1+e_1b_1)\right)\right)\\
        &=\check{F}_0(a+e_1b)+I\check{F}_1(a+e_1b)=\check{f}(a+Ib)=\check{f}(x)\,.
    \end{align*}
    For the third-before-last equality, we used formula~\eqref{eq:componentsofextendedstem}. The first statement is now proven. For the second statement, it suffices to remark that $\epsilon\partial_cf=\epsilon\partial_cg$ and $\epsilon\overline{\partial}_cf=\epsilon\overline{\partial}_cg$ when $f\in\mathcal{S}^1(\Omega_D)$. To prove the third statement, we simply remark that $f\in\mathcal{SR}(\Omega_D)$ implies both $f\in\mathcal{S}^1(\Omega_D)$ and $\overline{\partial}_cf\equiv0$.
\end{proof}

A very important property of dual slice functions is provided by the next theorem. Let $dg_{x_1}$ denote the real differential of a quaternionic slice function $g$ at a point $x_1$ of its domain. We are going to prove that, for $x_1\in\Omega_D\cap(\mathbb{H}\setminus\mathbb{R})$ fixed, the map $\mathbb{H}\to\mathbb{DH},\ x_2\mapsto\check{f}(x_1+\epsilon x_2)$ is a real affine map whose value at $0$ equals $f(x_1)$ and whose real differential at $0$ acts exactly like $\epsilon \,d\left(^\pi f\right)_{x_1}$.

\begin{thm}\label{thm:dualslice}
    Let $\check f\in\mathcal{DS}(\widecheck{\Omega}_D)$ be the natural extension to $\widecheck{\Omega}_D$ of $f=g+\epsilon h$ with $g\in\mathcal{S}_\mathbb{H}^1(\Omega_D),h\in\mathcal{S}_\mathbb{H}^0(\Omega_D)$. If $x=x_1+\epsilon x_2\in\widecheck{\Omega}_D$ with $x_1\in\mathbb{H}\setminus\mathbb{R}$ and $x_2\in\mathbb{H}$, then
    \[\check{f}(x)= f(x_1)+\epsilon \,d\left(^\pi f\right)_{x_1}\!(x_2)\,.\]
    If $f\in\mathcal{S}^{3,0}(\Omega_D)$ then the previous formula holds true for all $x\in\widecheck{\Omega}_D$. This is automatically true if $\check f\in\mathcal{DS}^2(\widecheck{\Omega}_D)$ and, in particular, when $\check f\in\mathcal{DSR}(\widecheck{\Omega}_D)$.
\end{thm}

\begin{proof}
    We first consider the case $x_1\not\in\mathbb{R}$ and use the decomposition $x_2=x_2^\parallel+x_2^\perp$, which implies $\pi_Q(x)=x_1+\epsilon x_2^\perp$. By formula~\eqref{eq:dualslice}, we have $\check f(x)=f(x_1+\epsilon x_2^\perp)+\epsilon x_2^\parallel \partial_cg(x_1+\epsilon x_2^\perp)+\epsilon (x_2^\parallel)^c \overline{\partial}_cg(x_1+\epsilon x_2^\perp)$. On the one hand, Lemma \ref{lem:decompositionusingsphericalderivative} implies $f(x_1+\epsilon x_2^\perp)=f(x_1)+\epsilon x_2^\perp \partial_s {^\pi\!f}(x_1)$. On the other hand, $\epsilon\,\partial_cg(x_1+\epsilon x_2^\perp)=\epsilon\,\partial_c {^\pi\! g}(x_1)=\epsilon\,\partial_c {^\pi\! f}(x_1)$ and $\epsilon\,\overline{\partial}_cg(x_1+\epsilon x_2^\perp)=\epsilon\,\overline{\partial}_c {^\pi\! f}(x_1)$. Thus,
    \[\check f(x)= f(x_1)+\epsilon x_2^\perp \partial_s {^\pi\!f}(x_1)+\epsilon x_2^\parallel\,\partial_c {^\pi\! f}(x_1)+\epsilon(x_2^\parallel)^c\,\overline{\partial}_c {^\pi\! f}(x_1)= f(x_1)+\epsilon \,d\left(^\pi f\right)_{x_1}\!(x_2)\,,\]
    where the last equality follows from Remark~\ref{rmk:differential}. The first statement is now proven. To prove the second statement, under the hypothesis $x_1\in\mathbb{R}$, we apply formula~\eqref{eq:dualslice} to $f\in\mathcal{S}^{3,0}(\Omega_D)$ and use the equality $\pi_Q(x)=x_1$ to get $\check f(x)=f(x_1)+\epsilon x_2\,\partial_c g(x_1)+\epsilon x_2^c\,\overline{\partial}_c g(x_1)=f(x_1)+\epsilon x_2\,\partial_c {^\pi\! f}(x_1)+\epsilon x_2^c\,\overline{\partial}_c {^\pi\! f}(x_1)=f(x_1)+\epsilon \,d\left(^\pi f\right)_{x_1}\!(x_2)$, where the last equality follows from Remark~\ref{rmk:differential}. The fourth statement follows at once from the previous ones if we take into account that
    \[\mathcal{DSR}(\widecheck{\Omega}_D)\subset\mathcal{DS}^2(\widecheck{\Omega}_D)=\widecheck{\mathcal{S}^{3,2}(\Omega_D)}\,.\qedhere\]
\end{proof}

As a consequence of Theorem~\ref{thm:dualslice},
\begin{align}\label{eq:explicitdualslice}
    \check f(x)&= g(x_1)+\epsilon\left(h(x_1)+dg_{x_1}\!(x_2)\right)\,,\\
    \pi_\mathbb{H}(\check f(x))&=g(x_1)\,,\notag\\
    \pi_{\epsilon\mathbb{H}}(\check f(x))&=\epsilon\left(h(x_1)+dg_{x_1}\!(x_2)\right)\notag\\
    &=\epsilon\left(h(x_1)+\operatorname{Re}(x_2^\parallel)(\partial_c+\overline{\partial}_c)g(x_1)+\operatorname{Im}(x_2^\parallel)(\partial_c-\overline{\partial}_c)g(x_1)+x_2^\perp\partial_sg(x_1)\right)\,,\notag
\end{align}
for all $x=x_1+\epsilon x_2\in\widecheck{\Omega}_D$ with $x_1\not\in\mathbb{R}$. If $g\in\mathcal{S}^3_\mathbb{H}(\Omega_D)$, then the first three equalities hold even for $x_1\in\mathbb{R}$. If $g\in\mathcal{SR}_\mathbb{H}(\Omega_D)$, then~\eqref{eq:explicitdualslice} reads as $\check f(x)=g(x_1)+\epsilon\left(h(x_1)+x_2^\parallel g'_c(x_1)+x_2^\perp g'_s(x_1)\right)$ for $x_1\not\in\mathbb{R}$ and as $\check f(x)=g(x_1)+\epsilon\left(h(x_1)+x_2g'_c(x_1)\right)$ for $x_1\in\mathbb{R}$. We point out, for future use, that $\epsilon \check f(x)=\epsilon g(x_1)=\epsilon\, {^\pi\!f(x_1)}=\epsilon f(x_1)$.

%%%%%%%%%%%%%%%%%%%%%%%%%%%%%

\section{Dual quaternionic polynomials}\label{sec:dualslicepolynomials}

Another important property of dual slice functions is that they include right-sided dual quaternionic polynomials, which were our original interest.

\begin{thm}\label{thm:dualslicepolynomials}
    $\mathbb{DH}[x]$ and $\mathbb{DH}[x,x^c]$ are nested real $\ast$-subalgebras (as well as nested right $\mathbb{DH}$-submodules) of $\mathcal{DS}(\mathbb{DH})$. Moreover, $\mathbb{DH}[x]$ is a real $\ast$-subalgebra and a right $\mathbb{DH}$-submodule of $\mathcal{DSR}(\mathbb{DH})$. Namely, each polynomial $\sum_{\ell,m=0}^nx^\ell(x^c)^ma_{\ell m}$ with $\{a_{\ell m}\}_{\ell,m=0}^n\subset\mathbb{DH}$ is the image through $\mathscr{I}$ of the polynomial stem function $\mathbb{R}_\mathbb{C}\to\mathbb{DH}_{\mathbb{C}},\ z\mapsto \sum_{\ell,m=0}^nz^\ell\bar z^ma_{\ell m}$.
\end{thm}

\begin{proof}
    Consider the inclusion $f:Q_\mathbb{DH}\to\mathbb{DH},\ x\mapsto x$ and the conjugation map $g:Q_\mathbb{DH}\to\mathbb{DH},\ x\mapsto x^c$. Now, $^\pi f:\mathbb{H}\to\mathbb{H}$ is the quaternionic identity map and $^\pi g:\mathbb{H}\to\mathbb{H}$ is quaternionic conjugation. By Theorem~\ref{thm:dualslice},
    \begin{align*}
        \check{f}(x_1+\epsilon x_2)&=x_1+\epsilon x_2\,,\\
        \check{g}(x_1+\epsilon x_2)&=x_1^c+\epsilon x_2^c=(x_1+\epsilon x_2)^c\,.
    \end{align*}
    In other words, $\check{f}:\mathbb{DH}\to\mathbb{DH}$ is the identity and $\check{g}:\mathbb{DH}\to\mathbb{DH}$ is conjugation $x\mapsto x^c$. By construction, $\check{f},\check{g}$ are slice-preserving. For all $\ell,m\in\mathbb{N}$, formula~\eqref{eq:specialproductofdualslice} yields that $x^\ell,(x^c)^m$ and their pointwise product $x^\ell(x^c)^m$ are all (slice-preserving) dual slice functions on $\mathbb{DH}$, induced by the stem functions  $z^\ell,(\bar z)^m$ and $z^\ell(\bar z)^m$ on $\mathbb{R}_\mathbb{C}$ (respectively).

 Since $\mathscr{I}$ is a right $\mathbb{DH}$-module isomorphism (mapping holomorphic stem functions into dual slice regular functions), it follows at once that $\mathbb{DH}[x]$ and $\mathbb{DH}[x,x^c]$ are nested right $\mathbb{DH}$-submodules of $\mathcal{DS}(\mathbb{DH})$ and that $\mathbb{DH}[x]$ is a right $\mathbb{DH}$-submodule of $\mathcal{DSR}(\mathbb{DH})$. Moreover, the real $\ast$-algebra structure we defined on $\mathbb{DH}[x]$ and $\mathbb{DH}[x,x^c]$ is compatible with the real $\ast$-algebra structure of $\mathcal{DS}(\mathbb{DH})$ because every dual slice function $\check{f}$ is uniquely determined by $f$ and because of the remarks made in Example~\ref{exm:Spol} and in Example~\ref{exm:SRpol}.
\end{proof}

Let us add a word about the slice-preserving, the dual slice-preserving and the quaternion-preserving cases among dual quaternionic polynomials. Thanks to the properties of $\mathscr{I}$,
\begin{align*}
    &\mathbb{R}[x,x^c]=\mathcal{DS}_{\mathbb{R}}(\mathbb{DH})\cap\mathbb{DH}[x,x^c]\,,&&\mathbb{R}[x]=\mathcal{DS}_{\mathbb{R}}(\mathbb{DH})\cap\mathbb{DH}[x]\,,\\
    &\mathbb{DR}[x,x^c]=\mathcal{DS}_{\mathbb{DR}}(\mathbb{DH})\cap\mathbb{DH}[x,x^c]\,,&&\mathbb{DR}[x]=\mathcal{DS}_{\mathbb{DR}}(\mathbb{DH})\cap\mathbb{DH}[x]\,,\\
    &\mathbb{H}[x,x^c]=\mathcal{DS}_{\mathbb{H}}(\mathbb{DH})\cap\mathbb{DH}[x,x^c]\,,&&\mathbb{H}[x]=\mathcal{DS}_{\mathbb{H}}(\mathbb{DH})\cap\mathbb{DH}[x]\,.
\end{align*}

A remark about quaternion-preserving components is in order.

\begin{rmk}\label{rmk:polynomialdecomposition}
    Consider the polynomial $\check{f}(x)=\sum_{\ell=0}^dx^\ell p_\ell\in\mathbb{DH}[x]$. After decomposing, for each $\ell\in\{0,\ldots,d\}$, the $\ell$-th coefficient as $p_\ell=q_\ell+\epsilon r_\ell$ with $q_\ell,r_\ell\in\mathbb{H}$, the quaternion-preserving components of $\check{f}$ are the elements $\check{g}(x)\coloneqq\sum_{\ell=0}^dx^\ell q_\ell,\check{h}(x)\coloneqq\sum_{\ell=0}^dx^\ell r_\ell$ of $\mathbb{H}[x]$. In particular, $\mathbb{DH}[x]=\mathbb{H}[x]+\epsilon\mathbb{H}[x]=\mathbb{H}[x]+\mathbb{H}[x]\epsilon$. Similar considerations apply to $\mathbb{DH}[x,x^c]$.
\end{rmk}

We conclude with a relevant example.

\begin{ex}\label{ex:decompositionofdeltacheck}
    Fix $y=y_1+\epsilon y_2\in\mathbb{DH}$ and set $\check{f}(x)\coloneqq x-y=x-y_1-\epsilon y_2$. We remark that $\check{f}^\circ_s(x)=\widecheck{\operatorname{Re}}(x)-y$ and $\check{f}'_s\equiv1$. Then $\check{f}^c(x)=x-y^c=x-y_1^c-\epsilon y_2^c$ and
    \begin{align*}
        N(\check{f})&=(x-y)\cdot(x-y^c)=x^2-x(y+y^c)+yy^c=x^2-xt(y)+n(y)\\
        &=x^2-xt(y_1)+n(y_1)+\epsilon(-xt(y_2)+t(y_1y_2^c))\in\mathbb{DR}[x]\,.
    \end{align*}
    We extend to all $y\in\mathbb{DH}$ the notation $\Delta_y(x)\coloneqq x^2-xt(y)+n(y)$ (for all $x\in Q_\mathbb{DH}$) set in Example~\ref{ex:deltaslice} and remark that $N(\check{f})=\check{\Delta}_y$. We also remark that $\check{\Delta}_y=\check{\Delta}_{y'}$ is equivalent to $y\sim y'$ and to $y'\in[y]$. The decomposition of Remark~\ref{rmk:polynomialdecomposition} takes the form
    \[\check{\Delta}_y(x)=\check{\Delta}_{y_1}(x)+\epsilon(-xt(y_2)+t(y_1y_2^c))\,,\]
    whence $^\pi\!\Delta_y(x_1)=\Delta_{y_1}(x_1)$ for all $x_1\in\mathbb{H}$. The following are equivalent:
    \begin{multicols}{3}
    \begin{enumerate}
        \item $y\in (\mathbb{R}+\epsilon\operatorname{Im}(\mathbb{H}))\cup Q_\mathbb{DH}$;
        \item $t(y_2)=0=t(y_1y_2^c)$;
        \item $\check{f}\in\mathcal{DS}_\mathbb{H}(\mathbb{DH})$;
        \item $\check{f}\in\mathcal{DS}_\mathbb{R}(\mathbb{DH})$;
        \item $\check{f}=\check{\Delta}_{y_1}$.
    \end{enumerate}
    \end{multicols}
    Condition {1} is the Study condition mentioned in Section~\ref{sec:introduction}. By Proposition~\ref{prop:dualsliceproductofcomponents}, 
    \begin{align*}
    (\check{\Delta}_y)^\circ_s(x)&=n\left(\check f^\circ_s(x)\right)+\widecheck{\operatorname{Im}}(x)^2\,n\left(\check f'_s(x)\right)=n\left(\widecheck{\operatorname{Re}}(x)-y\right)+\widecheck{\operatorname{Im}}(x)^2n(1)\\
    &=\widecheck{\operatorname{Re}}(x)^2+\widecheck{\operatorname{Im}}(x)^2-\widecheck{\operatorname{Re}}(x)t(y)+n(y)\,,\\
    (\check{\Delta}_y)'_s(x)&=t\left(\check f^\circ_s(x)\check f'_s(x)^c\right)=t\left(\widecheck{\operatorname{Re}}(x)-y\right)=2\widecheck{\operatorname{Re}}(x)-t(y)\,.
    \end{align*}
    Now, $\check{\Delta}_y=\mathscr{I}(F)=\mathcal{I}(\check{F})$ with $\check{F}(z)\coloneqq z^2-zt(y)+n(y)$. Furthermore, $\check{F}=\check{F}_0+e_1\check{F}_1$, where $\check{F}_0(a+e_1b)=a^2-b^2-at(y)+n(y)$ and $\check{F}_1(a+e_1b)=2ab-bt(y)$. We also have $(\Delta_{y_1})'_s(x_1)=2\operatorname{Re}(x_1)-t(y_1)$ and $(\Delta_{y_1})'_c(x_1)=2x_1-t(y_1)$. Thus, according to Theorem~\ref{thm:dualslice},
    \[\check{\Delta}_y(x)=
    \begin{cases}
            \Delta_{y}(x_1)+\epsilon\left(x_2^\parallel(2x_1-t(y_1))+x_2^\perp(2\operatorname{Re}(x_1)-t(y_1))\right)& \text{if\ }x_1\not\in\mathbb{R}\\
            \Delta_{y}(x_1)+\epsilon x_2(2x_1-t(y_1))& \text{if\ }x_1\in\mathbb{R}
        \end{cases}\,,\]
        while~\eqref{eq:explicitdualslice} gives
    \begin{align*}
        \pi_\mathbb{H}(\check{\Delta}_y(x))&=\Delta_{y_1}(x_1)\,,\\
        \pi_{\epsilon\mathbb{H}}(\check{\Delta}_y(x))&=
        \begin{cases}
            \epsilon\left(-x_1t(y_2)+t(y_1y_2^c)+x_2^\parallel(2x_1-t(y_1))+x_2^\perp(2\operatorname{Re}(x_1)-t(y_1))\right)& \text{if\ }x_1\not\in\mathbb{R}\\
            \epsilon\left(-x_1t(y_2)+t(y_1y_2^c)+x_2(2x_1-t(y_1))\right)& \text{if\ }x_1\in\mathbb{R}
        \end{cases}
    \end{align*}
    for all $x\in\mathbb{DH}$. We are going to use the last formulas to determine the zero set $V(\check{\Delta}_y)$. Specifically, we will prove that
    \[V(\check{\Delta}_y)=\begin{cases}
        [y]&\text{if }y\not\in\mathbb{R}+\epsilon\mathbb{H}\\
            H_{y_1}\supset[y]&\text{if }y\in\mathbb{R}+\epsilon\mathbb{H}
    \end{cases}\,.\]
    The equality $\pi_\mathbb{H}(\check{\Delta}_y(x))=\Delta_{y_1}(x_1)$ yields that $\check{\Delta}_y(x)$ can only vanish if $x_1\in V(\Delta_{y_1})=\mathbb{S}_{y_1}^\mathbb{H}$. If $y_1\in\mathbb{R}$, then $x_1\in\mathbb{S}_{y_1}^\mathbb{H}=\{y_1\}$ yields
    \begin{align*}
        \pi_{\epsilon\mathbb{H}}(\check{\Delta}_y(x))&=\epsilon\left(-y_1t(y_2)+t(y_1y_2^c)+x_2(2y_1-t(y_1))\right)\\
        &=\epsilon\left(y_1(-t(y_2)+t(y_2^c))+x_2(2y_1-2y_1)\right)=0\,,
    \end{align*}
    as desired. Now assume $x=A+IB,y=a+Jb$ for some $a,b,A,B\in\mathbb{DR},I,J\in\mathbb{S}_\mathbb{DH}$ (whence $t(y_2)=2a_2,t(y_1y_2^c)=2a_1a_2+2b_1b_2,x_2^\parallel=A_2+I_1B_2$), assume $y_1\not\in\mathbb{R}$ (i.e., $b_1\neq0$) and assume $x_1\in\mathbb{S}_{y_1}^\mathbb{H}$ (whence $A_1=a_1,B_1=b_1,x_1=a_1+I_1b_1,2\operatorname{Re}(x_1)=2a_1=t(y_1),\operatorname{Im}(x_1)=I_1b_1$): then
    \begin{align*}
        \pi_{\epsilon\mathbb{H}}(\check{\Delta}_y(x))&=\epsilon\left(-x_1t(y_2)+t(y_1y_2^c)+2x_2^\parallel\operatorname{Im}(x_1)\right)\\
        &=\epsilon\left(-2(a_1+I_1b_1))a_2+2a_1a_2+2b_1b_2+2(A_2+I_1B_2)I_1b_1\right)\\
        &=2\epsilon b_1\left(b_2-B_2+I_1(A_2-a_2)\right)
    \end{align*}
    vanishes if, and only if, $A_2=a_2,B_2=b_2$. This is the same as $x\in a+b\mathbb{S}_\mathbb{DH}=[y]$, as desired.
\end{ex}

%%%%%%%%%%%%%%%%%%%%%%%%%%%%%

\section{Concluding remarks}

In this work, we constructed a new class of functions from an open domain $\widecheck{\Omega}_D$ in the $\ast$-algebra $\mathbb{DH}$ of dual quaternions to $\mathbb{DH}$ itself: the class of dual slice functions $\mathcal{DS}(\widecheck{\Omega}_D)$. We provided this class with a real $\ast$-algebra structure and showed that it includes, as a $\ast$-subalgebra, the class $\mathbb{DH}[x]$ of dual quaternionic polynomials. The restriction of any element $\check{f}$ of $\mathcal{DS}(\widecheck{\Omega}_D)$ to the intersection $\Omega_D$ of its domain with the quadratic cone $Q_\mathbb{DH}$ is an element $f$ of $\mathcal{S}(\Omega_D)$, i.e., a dual quaternionic slice function according to~\cite{perotti}. The two respective $\ast$-algebra structures are compatible. A link to quaternionic slice functions was also established. The well-known properties of quaternionic and dual quaternionic slice functions (see~\cite{librospringer2,gstdualquaternions}) and the new results proven for dual slice functions thus become tools available to study $\mathbb{DH}[x]$.

We plan to develop in a forthcoming paper the study of zeros of dual slice functions and the related study of factorization in $\mathbb{DH}[x]$. Besides their intrinsic interest, dual quaternionic polynomials and their possible factorizations are important for their applications in mechanism science, developed in~\cite{hegedus,lischarlerschrocker,lischichoschrocker1,lischichoschrocker2,lischroeckerskopenkovscharler,pfurnerschroeckerhusty,schicho,siegelepfurnerschroecker,siegelescharlerschroecker} and overviewed in Section~\ref{sec:introduction}.

%%%%%%%%%%%%%%%%%%%%%%%%%%%%%

\section*{Acknowledgements}

The authors are supported by: INdAM, through GNSAGA, Finanziamento Premiale ``Splines for accUrate NumeRics: adaptIve models for Simulation Environments'', and Progetto GNSAGA ``Metodi reali in analisi e geometria ipercomplessa''; Universit\`a di Firenze, through Progetto ``Teoria delle funzioni ipercomplesse e applicazioni''; and MIUR, through PRIN 2022 ``Real and complex manifolds: geometry and holomorphic dynamics'' (2022AP8HZ9).\\\\
The authors gratefully acknowledge the anonymous reviewer's contribution to shaping the presentation of this work.

%%%%%%%%%%%%%%%%%%%%%%%%%%%%%

%% BIBLIOGRAPHY

%%%%%%%%%%%%%%%%%%%%%%%%%%%%%

\end{document}